\documentclass[11pt,twoside]{preprint}
\usepackage{hyperref}
\usepackage{times}
\usepackage{amssymb}
\usepackage{amsmath}
\usepackage{mhsymb}
\usepackage{mhequ}
\usepackage{mhenvs}
\usepackage{microtype}
\usepackage{mathrsfs}
\usepackage{graphicx}
\usepackage{tikz}
\usepackage[font=small,labelfont=bf]{caption}

\newtheorem{assumption}{Assumption}

\colorlet{symbols}{blue}
\colorlet{testcolor}{green!60!black}

\def\symbol#1{\textcolor{symbols}{#1}}
\def\1{\mathbf{\symbol{1}}}

\usetikzlibrary{shapes.misc}
\usetikzlibrary{shapes.symbols}
\usetikzlibrary{snakes}
\usetikzlibrary{decorations}

\tikzset{
	dot/.style={circle,fill=black,inner sep=0pt, minimum size=1mm},
	graydot/.style={circle,draw=gray,inner sep=0pt, minimum size=1mm},
	loopnode/.style={circle,draw=black,inner sep=0pt, minimum size=1.5mm},
	treenode/.style={circle,fill=black,inner sep=0pt, minimum size=1.5mm},
	charge/.style={circle,draw=black,inner sep=0pt, minimum size=3mm},
	loopline/.style={->,semithick,shorten >=1pt,shorten <=1pt},
	treeline/.style={semithick, densely dashed, shorten >=1pt,shorten <=1pt},
	homoge/.style={font=\scriptsize,draw=black,fill=white,inner sep=2pt},
	homo/.style={font=\scriptsize},
	KK/.style={thick,shorten >=1pt,shorten <=1pt, densely dashed, >=latex},
kerAlg/.style={thick},
kerAlg2/.style={semithick,dashed},
kerP/.style={very thick},
arrow/.style={->,thick,shorten >=3pt,shorten <=3pt},
renorm/.style={ultra thick},
	kernel/.style={semithick,shorten >=1pt,shorten <=1pt},
	}

\def\E{\mathbf{E}}
\def\P{\mathbf{P}}
\def\T{\mathbf{T}}

\def\${|\!|\!|}
\def\power{\wp}

\def\J{\mathcal{J}}

\def\s{\mathfrak{s}}
\def\diam{\mathop{\mathrm{diam}}}
\let\emptyset\varnothing
\def\Wick#1{{:\!#1\!:}}

\definecolor{darkred}{rgb}{0.9,0.1,0.1}

\definecolor{darkgreen}{rgb}{0.1,0.6,0.1}

\begin{document}

\title{The dynamical sine-Gordon model}
\author{Martin Hairer$^1$, Hao Shen$^2$}
\institute{University of Warwick, \email{M.Hairer@Warwick.ac.uk}
\and University of Warwick, \email{pkushenhao@gmail.com}}
\maketitle

\begin{abstract}
We introduce the dynamical sine-Gordon equation in two space dimensions
with parameter $\beta$, which is the natural dynamic associated to the usual quantum
sine-Gordon model. It is shown that when $\beta^2 \in (0,\frac{16\pi}{3})$
the Wick renormalised equation is well-posed. In the regime $\beta^2 \in (0,4\pi)$,
the Da Prato--Debussche method \cite{MR1941997,MR2016604} applies,
while for $\beta^2 \in [4\pi,\frac{16\pi}{3})$, the solution theory is provided
via the theory of regularity structures \cite{Regularity}.
We also show that this model arises naturally from a class of $2+1$-dimensional 
equilibrium interface fluctuation models with periodic nonlinearities.

The main mathematical difficulty arises in the construction of the model
for the associated regularity structure where the role of the noise is played
by a non-Gaussian random distribution similar to the complex multiplicative Gaussian
chaos recently analysed in \cite{RhodesVargas}.
\end{abstract}

\tableofcontents

\section{Introduction}

The aim of this work is to provide a solution theory for the 
stochastic PDE
\begin{equ}[e:model]
\d_t u = \frac{1}{2}\Delta u + c \sin\bigl( \beta u + \theta\bigr) 
 + \xi\;,
\end{equ}
where $c,\beta,\theta$ are real valued constants, $\xi$ denotes space-time white noise, and the spatial dimension is fixed to be $2$.

The model \eref{e:model} is interesting for a number of reasons. First and foremost, 
it is of purely mathematical interest as a 
very nice testbed for renormalisation techniques. Indeed, even though we work
with \textit{fixed} spatial dimension $2$, this model exhibits many features
comparable to those of various models arising in constructive quantum field theory (QFT)
and / or statistical mechanics,
but with the dimension $d$ of those models being a function of the parameter $\beta$. 

More precisely, at least at a heuristic level, \eref{e:model} is comparable to $\Phi^3_d$ Euclidean QFT
with $d = 2 + {\beta^2 \over 2\pi}$, $\Phi^4_d$ Euclidean QFT
with $d = 2 + {\beta^2 \over 4\pi}$, or the KPZ equation in dimension $d = {\beta^2 \over 4\pi}$.
In particular, one encounters divergencies when trying to define solutions
to \eref{e:model} or any of the models just mentioned
as soon as $\beta$ is non-zero. (In the case of the KPZ equation recall that, via
the Cole-Hopf transform it is equivalent to the stochastic heat equation. In dimension $0$,
this reduces to the SDE $du = u\,dW$ which is ill-posed if $W$ is a Wiener process but is
well-posed as soon as it is replaced by something more regular, say fractional Brownian
motion with Hurst parameter greater than $1/2$.)

These divergencies can however be cured in all of these models by Wick renormalisation as long as
$\beta^2 < 4\pi$. At $\beta^2 = 4\pi$ (corresponding to $\Phi^3_4$, $\Phi^4_3$, and KPZ in
dimension $1$), Wick renormalisation breaks down and higher order renormalisation schemes
need to be introduced. One still expects the theory to be renormalisable
until $\beta^2 = 8\pi$, which corresponds to $\Phi^3_6$, $\Phi^4_4$ and KPZ in dimension $2$,
at which point renormalisability breaks down. This suggests that the value $\beta^2 = 8\pi$ 
is critical for \eref{e:model} and that there is no hope to give it any non-trivial meaning
beyond that, see for example \cite{MR1777310,Falco} and, in a slightly different
context, \cite{RhodesVargas}. This heuristic (including the fact that Wick renormalisability breaks down
at $\beta^2 = 4\pi$) is well-known and has been demonstrated in 
\cite{Jurg,MR649810,MR702570,MR849210,MR1777310} at the
level of the behaviour of the partition function for the corresponding
lattice model.

From a more physical perspective, an interesting feature of \eref{e:model} is that it is closely
related to models of a globally neutral gas of interacting charges.
With this interpretation, the gas forms a plasma at high temperature (low $\beta$) 
and the various threshold values for $\beta$ could be interpreted
as threshold of formation of macroscopic fractions of dipoles / quadrupoles / etc.
The critical value $\beta^2 = 8\pi$ can be interpreted as the critical inverse
temperature at which total collapse takes place, giving rise to a 
Kosterlitz-Thouless phase transition \cite{KT,MR634447}.
Finally, the model \eref{e:model} has also been proposed as a model for 
the dynamic of crystal-vapour interfaces
at the roughening transition \cite{DynRough,Neudecker} and as a model of crystal
surface fluctuations in equilibrium \cite{SurfModel2,SurfModel1}.

In order to give a non-trivial meaning to \eqref{e:model}, we 
first replace $\xi$ by a smoothened version $\xi_\eps$ which has a correlation
length of order $\eps > 0$ and then study the limit $\eps \to 0$.
Since we are working in two space dimensions, one expects the solution to become
singular (distribution-valued) as $\eps \to 0$. As a consequence, there will be some
``averaging effect'' so that one expects to have $\sin(\beta u_\eps) \to 0$
in some weak sense as $\eps \to 0$. It therefore seems intuitively clear that if we wish to obtain
a non-trivial limit, we should simultaneously send the constant $c$ to $+\infty$.
This is indeed the case, see Theorem~\ref{theo:main} below.

We also study a class of $2+1$-dimensional 
equilibrium interface fluctuation models with more general periodic nonlinearities:
\begin{equ}[e:model-1]
\d_t u_\eps = \frac{1}{2}\Delta u_\eps 
 + c_\eps \,F_\beta(u_\eps) + \xi_\eps\;,
\end{equ}
where  $F_\beta$ is a trigonometric polynomial of the form
\begin{equ}[e:f_beta]
 F_\beta(u) = \sum_{k=1}^Z \zeta_k \sin\bigl( k \beta u + \theta_k \bigr)
 \qquad (Z\in\N) \;,
\end{equ}
and $\zeta_k$, $\beta$ and $\theta_k$ 
are real-valued constants. 
One may expect that the ``averaging effect'' of $\sin(k\beta u_\eps + \theta_k)$ 
is stronger for larger values of $k$. This is indeed the case and, as a consequence of this,
we will see that provided the constant $c_\eps\to \infty$ at a proper rate,
the limiting process obtained as $\eps \to 0$ only depends on 
$F_\beta$ via the values $\beta$, $\theta_1$, and $\zeta_1$.
In this sense, the equation \eqref{e:model} also arises as the limit 
of the models \eqref{e:model-1}.

The main result of this article can be formulated as follows. (See Section~\ref{sec:notations} below
for a definition of the spaces appearing in the statement; $\T^2$ denotes the
two-dimensional torus, and $\CD'$ denotes the space of distributions.)
\begin{theorem}\label{theo:main}
Let $0<\beta^2< {16\pi \over 3}$ and $\eta\in(-\frac{1}{3},0)$. 
For $u^{(0)} \in \CC^\eta(\T^2)$ fixed, 
consider the solution $u_\eps$ to
\begin{equ}
  \d_t u_\eps = 
  	\frac{1}{2}\Delta u_\eps 
	+  C_\rho \eps^{-\beta^2 / 4\pi} F_\beta(  u_\eps ) 
	+ \xi_\eps \;,
 \quad u(0,\cdot) = u^{(0)}\;,
\end{equ}
where $F_\beta$ is defined in \eref{e:f_beta}, $\xi_{\eps}=\rho_{\eps} \ast \xi$ with 
$\rho_{\eps}(t,x)= \eps^{-4}\rho(\eps^{-2}t,\eps^{-1}x)$
for some smooth and compactly supported function $\rho$ integrating to $1$. Then there
exists a constant $C_\rho$ (depending only on $\beta$ and the mollifier $\rho$)
such that the sequence
$u_{\eps}$ converges in probability to a limiting distributional
process $u$ which is independent of $\rho$. 

More precisely, there exist random variables $\tau > 0$ and $u \in \CD'(\R_+\times \T^2)$
such that, for every $T' > T > 0$, the natural restriction of $u$ to $\CD'((0,T)\times \T^2)$
belongs to $\CX_{T,\eta} = \CC([0,T], \CC^\eta(\T^2))$ on the set $\{\tau \ge T'\}$.
Furthermore, on the same set, one has $u_\eps \to u$ in probability in the 
topology of $\CX_{T,\eta}$.

Finally, one has $\lim_{t \to \tau} \left\Vert u(t,\cdot)\right\Vert _{\mathcal{C}^{\eta}(\mathbf{T}^{2})}=\infty$
on the set $\{\tau < \infty\}$. 
The limiting process $u$ depends on the numerical values $\beta$, $\zeta_1$ and $\theta_1$, but it
depends neither on the choice of mollifier $\rho$, nor on the numerical values $\zeta_k$ and $\theta_k$ for $k\geq 2$.
\end{theorem}

\begin{remark}\label{rem:thresholds}
As already mentioned, one expects the boundary $\beta^2 = {16 \pi \over 3}$ to be artificial and a
similar result is expected to hold for any $\beta^2 \in (0,8\pi)$. In fact, $8\pi$ 
is the natural boundary for the method of proof developed in \cite{Regularity} and employed here. 
However, as $\beta^2 \to 8\pi$,
the theory requires proofs of convergence of more and more auxiliary objects.
In the current context, we unfortunately do not have a general convergence result 
for all of these objects but instead we need to treat all of them separately ``by hand''. 
Furthermore, the bounds we have on the simplest ``second-order'' object unfortunately 
appear to break down at $\beta^2 = 6\pi$.
\end{remark}

\begin{remark}
It is interesting to note that for $\beta^2 \in (0,4\pi)$, we only need to construct one 
auxiliary process, and this construction does indeed involve a careful tracking of cancellations
due to the grouping of terms into ``dipoles'', while for  $\beta^2 \in [4\pi,{16\pi \over 3})$,
we need to build a second auxiliary process which requires to keep track of cancellations
obtained by considering ``quadrupoles''. See Section \ref{sec:linear} and 
Section \ref{sec:second-order} below for more details. 
\end{remark}

\begin{remark}
The limiting process $u$ is a continuous function of time,
taking values in a suitable space of spatial distributions. See Remark~\ref{rem:regularity} 
below for more details.
Regarding the right hand side of the equation however, 
it only makes sense as a random distribution at fixed time when $\beta^2 < 4\pi$.
For $\beta^2 \ge 4\pi$ however, it exists only as a random space-time distribution.
\end{remark}

\begin{remark}
The article \cite{AlbRuss} appears in principle to cover \eref{e:model} as part of a larger class
of nonlinearities. It is however unclear what the meaning of the solutions constructed there is
and how they relate to the construction given in the present article. The interpretation of the
solutions in \cite{AlbRuss} is that of a random Colombeau generalised function and it is not
clear at all whether this generalised function represents an actual distribution. 
In particular, the construction given there is completely
impervious to the presence of the Kosterlitz-Thouless transition and the collapse of 
multipoles which clearly transpire in our analysis.
\end{remark}

\subsection{Structure of the article}

The rest of this article is organised as follows. 
In Section~\ref{sec:method}, we give an overview of the proof of our main result.
In particular, we reduce it to the proof of convergence of a finite number of
processes $\Psi_\eps^k$, $\Psi_\eps^{kl}$, and  $\Psi_\eps^{k\bar l}$ (see \eref{e:defPsi}, \eref{e:Psi-ab}, \eref{e:renConst_kl}, \eref{e:Psipm2} and Remark~\ref{rem:init-2ndorder} below) 
to a limit in a suitable topology.
In Section~\ref{sec:linear},
we then prove Theorem~\ref{theo:convBasic}, which 
gives bounds on arbitrary moments of  the first order processes $\Psi_\eps^k$.
When combined with a simple second moment estimate, these 
bounds imply suitable convergence of the first order processes $\Psi_\eps^k$,
so that Theorem~\ref{theo:conv} is established, which in particular
implies Theorem~\ref{theo:main} for $\beta^2 < 4\pi$.
The main ingredient in proving Theorem~\ref{theo:convBasic}
is an inductive procedure, resulting in the bounds in Proposition~\ref{prop:hierarchical}, which greatly simplifies the expressions of the moments.

The last two sections of the article are devoted to the proof of 
Theorem~\ref{theo:second-order}, which gives bounds on arbitrary moments of 
second order processes $\Psi_\eps^{kl}$ and $\Psi_\eps^{k\bar l}$, as well as
their convergence.
It turns out that the proofs for the special case $k = l$ are quite
different from the proofs for $k \neq l$. Section~\ref{sec:second-order} 
only treats the former 
case, which is already sufficient to obtain Theorem~\ref{theo:main} 
in the particular case when $Z = 1$ in \eqref{e:f_beta}. The case $k \neq l$ is
then finally covered in Section~\ref{sec:kneql}.
In particular, among these second order processes, only $\Psi_\eps^{k\bar k}$
requires some renormalisations terms.
The convergence proof for these processes relies again on the 
procedure of Section~\ref{sec:linear},
but we have to incorporate into it additional cancellations created by 
the renormalisation constants.

\subsection{Some notations}
\label{sec:notations}

Throughout the article, we choose the scaling for our space-time
$\mathbf{R}^{3}$ to be the parabolic scaling $\s=(2,1,1)$, and thus the scaling dimension
of space-time is $|\s|=4$.
(See the conventions in \cite{Regularity}.)
This scaling defines a distance
$\left\Vert x-y\right\Vert _{\s}$ on $\mathbf{R}^{3}$ by
\begin{equ}
  \left\Vert x \right\Vert_{\s}^4 \eqdef |x_0|^2 + |x_1|^4 + |x_2|^4 \;.
\end{equ}
Recall from  \cite[Def.~3.7]{Regularity}
 that for $\alpha<0$, $r=-\left\lfloor \alpha\right\rfloor $,
$\mathfrak{D}\subseteq\mathbf{R}^{3}$, we say that a distribution
$\xi\in\mathcal{S}^{\prime}(\mathfrak{D})$ belongs to $\mathcal{C}_{\s}^{\alpha}(\mathfrak{D})$
if it belongs to the dual of $\mathcal{C}^{r}(\mathfrak{D})$, and
for every compact set $\mathfrak{R}\subseteq\mathfrak{D}$, there
exists a constant $C$ such that $\left\langle \xi,\mathcal{S}_{\s,x}^{\delta}\eta\right\rangle \leq C\delta^{\alpha}$
holds for all $\delta\leq1$, all $x\in\mathfrak{R}$, and all $\eta\in\mathcal{C}^{r}$
with $\left\Vert \eta\right\Vert _{\mathcal{C}^{r}}\leq1$ and
supported on the unit $d_{\s}$-ball centred at the origin. Here, the rescaled test function is given by
$
  \mathcal{S}_{\s,x}^{\delta}\eta(y)
  =\delta^{-|\s|}\eta(\delta^{-\s_{0}}(y_{0}-x_{0}),\ldots,\delta^{-\s_{2}}(y_{2}-x_{2}))
$.

\subsection*{Acknowledgements}

{\small
The authors benefitted from discussions with J\"urg Fr\"ohlich, 
Tom Spencer, Vincent Vargas, and R\'emi Rhodes.
MH's research was funded by the Philip Leverhulme trust through a leadership award, 
by the Royal Society through a research merit award, and by an ERC consolidator grant.
}

\section{Method of proof}
\label{sec:method}

Let $K\colon \R \times \R^2 \to \R$ be a compactly supported function which 
agrees with the heat kernel $\exp(-|x|^2/2t)/(2\pi t)$ in a ball of radius $1$ around the origin, is smooth
everywhere except at the origin, satisfies $K(t,x) = 0$ for $t < 0$, 
and has the property that $\int K(t,x) Q(t,x)\,dt\,dx = 0$
for every polynomial $Q$ of degree $2$. We then define
\begin{equ}[e:defPhi]
\Phi_\eps = K * \xi_\eps\;,
\end{equ}
where ``$*$'' denotes space-time convolution, so that 
$R_\eps \eqdef \d_t \Phi_\eps - \frac{1}{2}\Delta \Phi_\eps -  \xi_\eps$
is a smooth function that converges as $\eps \to 0$ to a smooth limit $R$.
The main reason for considering convolution with $K$ instead of the actual 
heat kernel is that we avoid problems of convergence at infinity. It also
allows us to fit more easily into the framework of \cite[Sec.~5]{Regularity}.

Let now $\Psi_\eps^k$ be defined by
\begin{equ}[e:defPsi]
\Psi_\eps^k = C_\rho \eps^{-\beta^2/4\pi} \exp(i k \beta \Phi_\eps)\;,
\end{equ}
and write $\Psi_\eps^k = \Psi_\eps^{c,k} + i \Psi_\eps^{s,k}$ for its real and imaginary parts.
Since the case $k=1$ is special, we furthermore use the convention that $\Psi_\eps = \Psi_\eps^1$
and similarly for $\Psi_\eps^c$ and  $\Psi_\eps^s$.
Using the same trick as in \cite{MR1941997,MR2016604}, we set $u_\eps = v_\eps +  \Phi_\eps$, 
so that
\begin{equ}
\d_t v_\eps = \frac{1}{2}\Delta v_\eps 
+ \sum_{k=1}^Z 
\zeta_k  \bigl(\sin (k \beta v_\eps + \theta_k)\, \Psi_\eps^{c,k} 
+ \cos(k \beta v_\eps + \theta_k)\, \Psi_\eps^{s,k} \bigr) + R_\eps\;.
\end{equ}
At this stage, we note that the PDE
\begin{equ}[e:general]
\d_t v = \frac{1}{2}\Delta v + f_c(v)\, \Psi^c + f_s(v)\, \Psi^s + R\;,
\end{equ}
is locally well-posed for any continuous space-time function $R$, any continuous initial condition,
any smooth functions $f_c$ and $f_s$,
and any $\Psi^c, \Psi^s \in \CC_\s^{-\gamma}(\R_+ \times \T^2)$
provided that $\gamma < 1$. Furthermore, in this case, the solution $v$ belongs to $\CC_\s^{2-\gamma}$
and it is stable with respect to perturbations of $\Psi^c$, $\Psi^s$ and $R$.
The only potential problem are the products $f_c(v)\, \Psi^c$ and $f_s(v)\, \Psi^s$, but the 
product map turns out to be continuous from $\CC_\s^{2-\gamma}$ times $\CC_\s^{-\gamma}$ 
into $\CC_\s^{-\gamma}$ as soon as $\gamma < 1$.
(See for example \cite[Sec.~2]{Triebel} or \cite[Thm~2.52]{BookChemin}.)

With this in mind, our first result is as follows:

\begin{theorem}\label{theo:conv}
Assume that $\beta^2 \in (0,8\pi)$. Let $\Psi_\eps$ be as in \eref{e:defPsi} and let $\gamma > \beta^2 / (4\pi)$. Then, there exists a 
constant $C_\rho$ and a $\CC_\s^{-\gamma}(\R_+ \times \T^2,\C)$-valued 
random variable 
$\Psi$ independent of $\rho$ such that, for every 
$T> 0$, one has $\Psi_\eps \to \Psi$ and $\Psi_\eps^k \to 0$ for all $k\geq 2$
in probability in $\CC_\s^{-\gamma}([0,T] \times \T^2, \C)$.
\end{theorem}

\begin{remark}
This is essentially a consequence of \cite[Thm~3.1]{RhodesVargas} in the special case 
$\gamma = 0$. (Which is actually the simplest of the cases treated there.) 
Due to a difference in normalisation (compare 
\cite[Eq.~1.2]{RhodesVargas} to 
\eqref{e:decompQ} below) our values of $\beta^2$ differ by a factor $2\pi$, so that 
the boundary $\beta^2 = 8\pi$ appearing here corresponds to $\beta^2 = 4$
in the notations of \cite{RhodesVargas}. This is consistent with the fact 
that parabolic space-time with two space dimensions actually has 
Hausdorff dimension $4$. We will provide a full proof of Theorem~\ref{theo:conv} in 
Section~\ref{sec:linear} for a number of reasons. First, we require a much 
stronger notion of
convergence than that given in \cite{RhodesVargas} and our sequence of approximations
is different than the one given there (in particular it has no martingale 
structure in $\eps$). We also require optimal 
bounds in the parabolic scaling which are not given by that article. Finally,
several ingredients of our proof are reused in later parts of the article. 
\end{remark}

Given the above discussion, Theorem~\ref{theo:main} is an immediate consequence of Theorem~\ref{theo:conv}
for the range $\beta^2 \in (0,4\pi)$,
if the initial data $u^{(0)}$ is equal to $\Phi(0)$ plus a continuous function.
The proof of Theorem~\ref{theo:main} for general initial data $u^{(0)}\in\CC^\eta(\mathbf{T}^2)$ with $\eta\in(-\frac{1}{3},0)$ can be obtained in a way similar
to that of Theorem~\ref{theo:v-equ} below.
At $\beta^2 = 4\pi$ however, this appears to break down completely. Indeed,
it is a fact that the solutions to \eref{e:general} are \textit{unstable} with respect to perturbations
of $\Psi^s$ and $\Psi^c$ in $\CC_\s^{-1}$.
However, it turns out that if we keep track of suitable \textit{higher order information}, then
continuity is restored. More precisely, for each $1\le k \le Z$, 
let $\Psi_\eps^{s,k}$ and $\Psi_\eps^{c,k}$ be two sequences of
continuous space-time functions and let $C_\eps^{(k)}$ be a sequence of real numbers.
Then, for any two space-time points $z = (t,x)$ and $\bar z = (\bar t,\bar x)$
and any two indices $a,b \in \{s,c\}$,
we consider the functions
\begin{equ}[e:Psi-ab]
\Psi_\eps^{ab,kl}(z,\bar z) 
= \Psi_\eps^{a,k}(\bar z) \Bigl(\bigl(K * \Psi_\eps^{b,l})(\bar z) 
	- \bigl(K * \Psi_\eps^{b,l})(z)\bigr)\Bigr) 
- \frac{1}{2}C_\eps^{(k)} \delta_{a,b} \delta_{k,l} \;.
\end{equ}
In the sequel, we consider $\Psi_\eps^{ab,kl}$ as functions of their first argument, taking
values in the space of space-time distributions, corresponding to their second argument. We also use the convention that $\Psi_\eps^{ab}\eqdef\Psi_\eps^{ab,11}$ for simplicity.

 Given a test function $\phi\colon \R \times \R^2 \to \R$, a
point $z$ as before and a value $\lambda > 0$, we write
\begin{equ}[e:scaling]
\phi_z^\lambda(\bar z) = \lambda^{-4} \phi \bigl(\lambda^{-2}(\bar t - t), \lambda^{-1} (\bar x - x)\bigr)\;.
\end{equ}
We then impose the following assumption, which will later be justified in
Theorem~\ref{theo:convModel}.

\begin{assumption} \label{assumption:A}
We assume that there exist distributions $\Psi^a$ and distribution-valued functions
$\Psi^{ab}(z,\cdot)$ such that 
$\Psi_\eps^a \to \Psi^a$ and $\Psi_\eps^{ab} \to \Psi^{ab}$,
as well as
$\Psi_\eps^{a,k} \to 0$ for $k\geq 2$ and $\Psi_\eps^{ab,kl} \to 0$ for $(k,l)\neq (1,1)$,
in the following sense.
For some $\gamma \in (1, {4\over 3})$, one has 
\minilab{e:convModel}
\begin{equs}[2]
\lambda^{\gamma}|\Psi^a(\phi_z^\lambda)| &\lesssim 1 \;,\quad&
\lambda^{\gamma}|\bigl(\Psi_\eps^a - \Psi^a\bigr)(\phi_z^\lambda)| &\to 0 \;,\qquad \label{e:convModel1}\\
\lambda^{2\gamma-2}|\Psi^{ab}(z,\phi_z^\lambda)| &\lesssim 1 \;,\quad&
\lambda^{2\gamma-2}|\bigl(\Psi_\eps^{ab} - \Psi^{ab}\bigr)(z,\phi_z^\lambda)| &\to 0 \;,\label{e:convModel2}\\
\lambda^{\gamma}|\Psi^{a,k}(\phi_z^\lambda)| &\to 0 \;,\quad&
\lambda^{2\gamma-2}|\Psi_\eps^{ab,kl} (z,\phi_z^\lambda)| &\to 0 \;,\label{e:convModel3}
\end{equs}
for all $k\geq 2$ on the left and all $(k,l)\neq (1,1)$ on the right, where the limits on the right (as $\eps \to 0$) and the bounds on the left are 
both assumed to be \textit{uniform} over all
$\lambda \in (0,1]$, all smooth test functions $\phi$ that are supported in the centred ball
of radius $1$ and with their $\CC^2$ norm bounded by $1$, as well as all $z \in [-T,T] \times \T^2$
for any fixed $T> 0$.
\end{assumption}

\begin{remark}\label{rem:topConv}
The structure of \eref{e:general} is essentially the same as that of (PAMg) in 
\cite[Secs~1.5 and 10.4]{Regularity}. To make the link between the bounds~\eref{e:convModel} in Assumption~\ref{assumption:A} and 
\cite[Sec.~10.4]{Regularity} more precise, one could have used the notations of 
\cite[Sec.~8]{Regularity} and introduced $2Z$ abstract symbols $\Xi_c^k$ and $\Xi_s^k$ 
of homogeneity $-\gamma$, as well as an abstract integration operator $\CI$.
One then has the following correspondence with \cite[Sec.~8]{Regularity}:
\begin{equ}[e:model-PAMg]
\Pi_z \Xi_a^k = \Psi^{a,k}\;,\quad \Pi_z \Xi_a^k \CI(\Xi_b^l) = \Psi^{ab,kl}(z,\cdot)\;.
\end{equ}
The fact that the notion of convergence given in 
Assumption~\ref{assumption:A} is equivalent
to the convergence of admissible models of \cite[Sec.~2.3]{Regularity} is an immediate consequence
of \cite[Thm~5.14]{Regularity}, see also  \cite[Thm~10.7]{Regularity}.
\end{remark}

\begin{remark}\label{rem:Kolmogorov}
There exists an analogue to Kolmogorov's continuity test in this context, see 
\cite[Thm~10.7]{Regularity}. In our notations, it states that if there exists $\bar \gamma < \gamma$
such that, for every $p \ge 1$, the bounds
\begin{equ}
\E \lambda^{p\bar \gamma}|\Psi^a(\phi_z^\lambda)|^p \lesssim 1 \;,\quad
\E \lambda^{p\bar \gamma}|\bigl(\Psi_\eps^a - \Psi^a\bigr)(\phi_z^\lambda)|^p \to 0 \;, \quad
\E \lambda^{p\bar \gamma}|\Psi^{a,k\geq 2}(\phi_z^\lambda)|^p \to 0 \;,
\end{equ}
hold uniformly over $\lambda$, $\phi$ and $z$ as before, and similarly for $\Psi_\eps^{ab,kl}$, $\Psi^{ab,kl}$,
then 
Assumption~\ref{assumption:A} holds in probability.
\end{remark}


One then has the following result.
Note that the functions in $v_\eps$ that are multiplied with $\Psi_\eps^{s,k}$, $\Psi_\eps^{c,k}$ can be more general functions, as in the case of (PAMg);
in the statement of the theorem below we allow them to be trigonometric polynomials.
We call a function a trigonometric polynomial if it is a finite linear combination
of $\sin(\nu_k \cdot + \theta_k)$ for some constants $\nu_k$ and $\theta_k$.

\begin{theorem} \label{theo:v-equ}
Assume that the space-time functions $\Psi_\eps^{a,k}$
and $\Psi_\eps^{ab,kl}$ with $a,b\in\{s,c\}$, $1\le k,l\le Z$, 
 are related by \eref{e:Psi-ab}
 and that 
Assumption~\ref{assumption:A} holds for some $\gamma \in (1,{4\over 3})$.
Let $v_\eps$ be the solution to
\begin{equs}
 \d_t v_\eps 
   = \Delta v_\eps 
  &+ \sum_{k=1}^Z  \Big(
	f_{c,k}(v_\eps)\, \Psi^{c,k}_\eps + f_{s,k}(v_\eps)\, \Psi^{s,k}_\eps \Big) \\
  & - \sum_{k=1}^Z 
  	C_\eps^{(k)} \bigl(f_{c,k}(v_\eps) f_{c,k}'(v_\eps) + f_{s,k}(v_\eps) f_{s,k}'(v_\eps)\bigr) + R_\eps\;,\\
v_\eps(0,\cdot) &= v^{(0)}\;, \label{e:renormv}
\end{equs}
where $R_\eps$ is a sequence of continuous functions converging locally uniformly to a limit $R$
and $v^{(0)} \in \CC^\eta(\T^2)$ for some $\eta > -\frac{1}{3}$. 
Assume furthermore that for every $k$, the functions $f_{c,k}$ and $f_{s,k}$ are 
trigonometric polynomials.
Then
the sequence
$v_{\eps}$ converges in probability and locally uniformly as $\eps\to 0$ to a 
limiting process $v$.

More precisely, there exists a stopping time $\tau > 0$ and a random variable 
$v \in \CD'(\R_+\times \T^2)$
such that, for every $\eta\in(-\frac{1}{3},0)$ and every $T' > T > 0$, the natural restriction of $v$ to $\CD'((0,T)\times \T^2)$
belongs to $\CX_{T,\eta} = \CC([0,T], \CC^\eta(\T^2)) \cap \CC((0,T]\times\T^2)$ on the set $\{\tau \ge T'\}$.
Furthermore, on the same set, one has $v_\eps \to v$ in probability in the 
topology of $\CX_{T,\eta}$.
Finally, one has $\lim_{t \to \tau} \left\Vert v(t,\cdot)\right\Vert _{\mathcal{C}^{\eta}(\mathbf{T}^{2})}=\infty$
on the set $\{\tau < \infty\}$. 
The limiting process $v$ depends on $f_{s,1},f_{c,1}$ and $\beta$, but it
depends neither on the choice of mollifier $\rho$, nor on 
the functions $f_{s,k}$, $f_{c,k}$ for $k\geq 2$.
\end{theorem}

\begin{proof}
The theorem would be a straightforward 
consequence of \cite[Thm~7.8]{Regularity} and Remark~\ref{rem:topConv}
if we allowed $v^{(0)}$ to have positive regularity. However, we would like to allow
for negative regularity of the initial condition in order to be able to deduce
Theorem~\ref{theo:main} from this result. The reason why negative regularity of the initial 
condition is a natural requirement in the context of Theorem~\ref{theo:main} 
is that solutions at positive times 
necessarily have negative regularity, whatever the initial condition. 

In order to allow initial
data of negative regularity, we perform the transformation
$v_\eps=Gv^{(0)}+w_\eps$, where $Gv^{(0)}$ denotes the solution to the heat
equation with initial condition $v^{(0)}$. As a consequence of trigonometric 
identities, $w_\eps$ then solves
\begin{equ}[e:equweps]
\d_t w_\eps  =  \Delta w_\eps 
 + \sum_{a,k}   \Big(
	 g_{a,k,v^{(0)}}(w_\eps) \,  \Psi^{a,k}_\eps   
 -   C_\eps^{(k)} \,
g_{a,k,v^{(0)}}(w_\eps)  \, g_{a,k,v^{(0)}}'(w_\eps) \Big)
+ R_\eps \;,
\end{equ}
with initial condition $w_\eps(0,\cdot) = 0$.
Here, the functions $g_{a,k,v^{(0)}}(w_\eps)$ are 
finite linear combinations of terms of the type
\begin{equ}[e:RHSPAM]
\sin(\nu\,Gv^{(0)} + \theta)  \sin(\tilde \nu w_\eps + \tilde \theta)\;,
\end{equ}
for some $\nu,\tilde\nu, \theta, \tilde\theta\in \R$, and $g_{a,k,v^{(0)}}'$
denotes the derivative of $g_{a,k,v^{(0)}}$ with respect to its argument $w_\eps$.

In order to show that $w_\eps \to w$ as $\eps \to 0$, we make use of the 
theory developed in \cite{Regularity}. (We could also equivalently have 
used the theory developed in \cite{PAMPreprint} which requires
very similar assumptions.) We note that \eqref{e:equweps} is of the same
type as the class of equations treated in \cite[Secs~9.1, 9.3]{Regularity},
one difference being that the single noise $\xi_\eps$ is replaced by a collection
of noises $\Psi^{a,k}_\eps$. As a consequence, the relevant algebraic structure in
our context is built in exactly the same way as in \cite[Sec.~8]{Regularity},
but with the single abstract symbol $\Xi$ replaced by a collection of symbols
$\Xi_a^{k}$ representing $\Psi^{a,k}_\eps$, each of them of homogeneity $-\gamma$
with $\gamma$ as in Assumption~\ref{assumption:A}. 
If we denote by $\CP$ the integration operator corresponding to convolution with 
the heat kernel (see \cite[Sec.~5]{Regularity}), \eqref{e:equweps} can be
described by the following fixed point problem:
\begin{equ}[e:FPW]
W = \CP \one_{t > 0} \Bigl(R_\eps + \sum_{a,k} g_{a,k,v^{(0)}}(W_\eps) \,  \Xi_a^{k}\Bigr)\;.
\end{equ}

Indeed, as already noted in Remark~\ref{rem:topConv}, the condition $\gamma < 4/3$ guarantees 
that any model $(\Pi,\Gamma)$ for the corresponding regularity structure
is uniquely determined by the action of $\Pi$ onto the symbols $\Xi_a^k$ 
and $\Xi_a^k \CI(\Xi_b^l)$, and Assumption~\ref{assumption:A} 
precisely states that sequence of models $(\Pi_\eps,\Gamma_\eps)$ given by \eqref{e:model-PAMg} but with $\Psi$ replaced by $\Psi_\eps$ converges to a limiting
model $(\Pi,\Gamma)$. Furthermore, it follows in exactly the same way as
\cite[Prop.~9.4]{Regularity} that if $W$ solves \eqref{e:FPW} for the model
given by \eqref{e:model-PAMg}, but with $\Psi$ replaced by $\Psi_\eps$ and satisfying
the relation \eqref{e:Psi-ab}, then $\CR W$ (where $\CR$ denotes the corresponding
reconstruction operator, which in this case simply discards the higher order
information encoded in $W$) solves \eqref{e:equweps}.

It therefore remains to show that \eqref{e:FPW} admits a unique (local) 
solution for every admissible model, and that this solution depends continuously
on the model in question. In view of \cite[Thm~7.8]{Regularity}, this
is the case if we can show that the map 
\begin{equ}[e:RHS]
W \mapsto \sin(\nu\,Gv^{(0)} + \theta)  \sin(\tilde \nu w_\eps + \tilde \theta) \Xi_a^{k}\;,
\end{equ}
is locally Lipschitz from $\CD^{\mu,0}$ into $\CD^{\mu-\alpha,-\alpha}$
for some $\mu > \alpha$ and some $\alpha \in (0,2)$. (See \cite[Def.~6.2]{Regularity}
for the definition of the spaces $\CD^{\gamma,\eta}$.)

At this stage, we claim that as long as $\mu \in (0, 2]$ and 
$v^{(0)} \in \CC^\eta(\T^2)$ for some $\eta \in (-1/3,0)$, then
$g = \sin(\nu\,Gv^{(0)}+\theta)$ can be interpreted as an element in $\CD^{\mu,2\eta}(\bar T)$ where $\bar T$ is the space of abstract Taylor polynomials (if only
Taylor polynomials are involved, these are just suitably weighted H\"older spaces).
Indeed,
we have the bounds
\begin{equs}
|g(t,x)| & \lesssim 1 
	\lesssim (t \wedge 1)^{\eta/2} \lesssim (t \wedge 1)^{\eta} \;,\\
|\d_x g(t,x)| & \lesssim |\d_x G(v^{(0)})| 
	\lesssim (t \wedge 1)^{(\eta-1)/2}  \lesssim (t \wedge 1)^{\eta-1/2}  \;, \\
|\d_x^2 g(t,x)| + |\d_t g(t,x)| & \lesssim |\d_x G(v^{(0)})|^2 + |\d_x^2G(v^{(0)})|
+ |\d_t G(v^{(0)})| \\
	&\lesssim (t \wedge 1)^{\eta-1} + (t \wedge 1)^{(\eta-2)/2}  
	 \lesssim (t \wedge 1)^{\eta-1}   \;,
\end{equs}
from which the fact that $g \in \CD^{\mu,2\eta}$
follows similarly to \cite[Lemma~7.5]{Regularity}.

It furthermore follows from \cite[Prop.~6.13]{Regularity} that the map $W \mapsto \sin(\tilde \nu W + \tilde \theta)$ is locally Lipschitz from $\CD^{\mu,0}$ into itself.
Combining this with the fact that $g \in \CD^{\mu,2\eta}$ as mentioned above and
that $\Xi_a^{k}$ is of homogeneity $-\gamma$, we conclude from 
\cite[Prop.~6.12]{Regularity} that the map \eqref{e:RHS} is indeed 
locally Lipschitz continuous from $\CD^{\mu,0}$ into
$\CD^{\mu - \gamma,2\eta- \gamma}$.
Since $2\eta-\gamma+2 > -\frac{2}{3}-\frac{4}{3}+2 =0$ and $\mu-\gamma +2 >\mu$,
these exponents do satisfy the required inequalities, thus concluding the proof.
\end{proof}

\begin{remark} \label{rem:regularity}
In fact, the limiting process $v$ belongs to $\CC \left( (0,T],C^{(2-\gamma) \vee 0}(\mathbf{T}^2) \right)$ for every $\gamma>\beta^2/ 4\pi$, as soon as the random variable $\tau$ is strictly greater than $T$.  Since the Gaussian process $\Phi$ belongs to $\CC \left( (0,T],\CC^{-\delta}(\mathbf{T}^2) \right)$ for every $\delta>0$, the solution to the original equation \eref{e:model} is continuous in time for positive times, with values in $\CC^{-\delta}(\mathbf{T}^2)$ for every $\delta>0$.
\end{remark}

\begin{remark}
The condition $\gamma < {4\over 3}$ (corresponding to $\beta^2 < {16\pi \over 3}$ via the
correspondence $\gamma > \beta^2 /4\pi$ which we have seen in Theorem~\ref{theo:conv})
comes from the fact that we restrict ourselves to second-order processes in  Assumption~\ref{assumption:A}.
If we were to consider suitable additional third-order processes as well, this threshold would
increase to $\beta^2 < 6\pi$. In principle, by obtaining convergence of 
the corresponding (suitably renormalised) processes of arbitrarily high order, 
the threshold could be increased all the way up to $\beta^2 < 8\pi$, but this is highly 
non-trivial. At $\beta^2 = 8\pi$, one loses local 
subcriticality in the sense of \cite[Assumption~8.3]{Regularity}
and the theory breaks down.
\end{remark}

At this stage, we note that 
in our specific situation 
$f_{c,k}(v) = \zeta_k \sin(k\beta v + \theta_k)$
and $f_{s,k}(v) = \zeta_k \cos(k\beta v + \theta_k )$, so that one has the identity 
$f_{c,k} f_{c,k}' + f_{s,k} f_{s,k}' = 0$ for each $k$. As a consequence,
the ``renormalised'' equation \eref{e:renormv} is identical to the ``original'' equation \eref{e:general}!
It is now clear that Theorem~\ref{theo:main} follows from the following result, which is the main
technical result of this article.

\begin{theorem}\label{theo:convModel}
Assume that $\beta^2 \in [4\pi, \frac{16\pi}{3})$.
Let $\Psi_\eps^k = \Psi_\eps^{c,k} + i \Psi_\eps^{s,k}$ be defined as in \eref{e:defPsi},
and $\Psi_\eps^{ab,kl}$ for $a,b \in \{s,c\}$
be defined as in \eref{e:Psi-ab}.
Then there exist choices of constants $C_\eps^{(k)}$
depending only on $\beta$ and the mollifier $\rho$, and distributions $\Psi^a$ and distribution-valued functions
$\Psi^{ab}(z,\cdot)$ where $a,b \in \{s,c\}$, which are independent of the mollifier $\rho$,
such that Assumption~\ref{assumption:A} holds.
\end{theorem}

Theorem~\ref{theo:convModel} is proved in Section~\ref{sec:second-order}.
As discussed in Remark~\ref{rem:init-2ndorder},
this theorem is an immediate consequence of Lemma~\ref{lem:Feps}
and Theorem~\ref{theo:second-order}.

At this stage one might wonder if, in view of \cite[Sec.~10]{Regularity} and \cite{KPZJeremy},
there is anything non-trivial left to prove at all. The reason why 
Theorem~\ref{theo:convModel} is not covered by these results is that, in view of \eref{e:defPsi},
the stochastic processes $\Psi_\eps^{a,k}$ are obviously not Gaussian. 
Worse, they do not belong to any Wiener chaos of fixed order. As a consequence, 
we have no automatic way of obtaining equivalence of moments and Wick's formula
does not hold, which is the source of considerable complication.

\section{Convergence of the first-order process}
\label{sec:linear}

In this section, we prove Theorem~\ref{theo:conv} and we will retain the notations
from the introduction. This time however, we define $\Psi_\eps$ somewhat more indirectly by
\begin{equ}[e:defPsi2]
	\Psi_\eps = \Wick{e^{i \beta \Phi_\eps}} 
		\eqdef e^{i \beta \Phi_\eps + {\beta^2 \over 2} \CQ_\eps(0)} \;,
	\qquad
	\Psi_\eps^k 
		\eqdef e^{i k \beta \Phi_\eps + {\beta^2 \over 2} \CQ_\eps(0)} 
	\quad (k\geq 2)\;,
\end{equ}
where $\CQ_\eps$ denotes the covariance function of the Gaussian process $\Phi_\eps$.
Using the definition \eref{e:defPhi}, one has the identity
\begin{equ}[e:defQeps]
\CQ_\eps = (K * \rho_\eps) * \CT(K * \rho_\eps)\;,
\end{equ}
where $\CT$ denotes the reflection operator given by $(\CT F)(z) = F(-z)$.
The link between this definition and \eref{e:defPsi} is given by the following result,
the proof of which is postponed to the end of this section.

\begin{lemma}\label{lem:const}
There exists a constant $\hat C_\rho$ depending only on the mollifier $\rho$ and such that
\begin{equ}
\CQ_\eps(0) = -  {1\over 2\pi} \log \eps  + \hat C_\rho + \CO(\eps^2)\;.
\end{equ}
In particular, if Theorem~\ref{theo:conv} holds for $\Psi_\eps$ defined as in \eref{e:defPsi2},
then it also holds for $\Psi_\eps$ defined as in \eref{e:defPsi}.
\end{lemma}

Our first main result is then the following:

\begin{theorem}\label{theo:convBasic}
Let $0<\beta^2<8\pi$. There exists a stationary random complex 
distribution-valued process $\Psi$ 
independent of the mollifier $\rho$
such that
$\Psi_\eps \to \Psi$ in probability. Furthermore, 
for every $\kappa>0$ sufficiently small, one has
\begin{equ}[e:boundFirstLevel]
	\E |\scal{\phi_x^\lambda, \Psi_\eps}|^p \lesssim  \lambda^{- {\beta^2 p \over 4\pi}}\;,
	\qquad
	\E |\scal{\phi_x^\lambda, \Psi_\eps - \Psi}|^p
		\lesssim \eps^{\kappa} \lambda^{- {\beta^2 p \over 4\pi} - \kappa}\;,
\end{equ}
\begin{equ}[e:bound-kthmode]
\E |\scal{\phi_x^\lambda, \Psi_\eps^k}|^p \lesssim
\eps^{p \kappa}  \lambda^{- {\beta^2 p \over 4\pi} - p \kappa}
\qquad (k\geq 2)\;,
\end{equ}
uniformly over all test functions $\phi$ supported in the unit ball and bounded by $1$,
all $\lambda \in (0,1]$, and locally uniformly over space-time points $x$.
\end{theorem}

\begin{remark}
Throughout this paper we write 
$\scal{\phi_x^\lambda, \Psi} \eqdef \int_{\R^3} \phi_x^\lambda(\bar x) \Psi(\bar x) \,d\bar x$
when $\Psi$ is function of one space-time variable.
In the following, we will also write
$\scal{\phi_x^\lambda, \Psi} \eqdef \int_{\R^3} \phi_x^\lambda(\bar x) \Psi(x, \bar x)\, d\bar x$ if $\Psi$  is function of two space-time variables such as the functions defined in \eref{e:Psi-ab}.
\end{remark}

\begin{proof}
We first obtain the a priori bound stated in the theorem for finite values of $\eps$. 
Denote 
\begin{equ}[e:defJeps]
\J_\eps(z) = \exp(-\beta^2 \CQ_\eps(z))\;,
\end{equ}
where $\CQ_\eps$ was defined in \eref{e:defQeps}. 
We note that, as a consequence of the commutativity of convolution and the fact 
that $\CT(f*g) = \CT f * \CT g$, 
\eref{e:defQeps} can be rewritten as
\begin{equ}
\CQ_\eps = \CQ * (\rho_\eps * \CT \rho_\eps)\;,\qquad \CQ = K * \CT K\;.
\end{equ}
In particular, setting $\bar \rho = \rho * \CT \rho$, one has
\begin{equ}[e:CQeps]
\CQ_\eps = \CQ * \bar \rho_\eps\;,
\end{equ}
and this is the expression that we are going to make use of here.

Then, by Corollary \ref{cor:Jeps} below,
we have the bounds
\begin{equ}[e:divergenceKeps]
 (\|z\|_\s + \eps)^{\beta^2\over 2\pi} \lesssim \J_\eps(z) \lesssim (\|z\|_\s + \eps)^{\beta^2\over 2\pi}\;,
 \qquad \text{for $0 \le \|z\|_\s \le 1$,}
\end{equ}
where the notation $\lesssim$ hides proportionality constants independent of $\eps$.
We will also use the notation $ \J^{-}_\eps(z) = \J^{-1}_\eps(z) = 1/ \J_\eps(z)$.
With this notation at hand, one verifies that
\begin{equ}
\E |\scal{\phi_x^\lambda, \Psi_\eps}|^{2N} = \iint {\prod_{i,j} \phi_x^\lambda(z_i)\phi_x^\lambda(y_j)\J_\eps^{-}(z_i - y_j) \over \prod_{\ell < m} \J_\eps^{-}(z_\ell-z_m)
\prod_{n < o} \J_\eps^{-}(y_n-y_o) } dz\,dy\;,
\end{equ}
where both integrations are performed over $(\R^3)^N$ and each $z_i$ and $y_i$ 
is an element of $\R^3$ (space-time). In a very similar context, a quantity of this
type was already bounded in \cite[Thm~3.4]{Jurg}. However, the proof given there relies on
an exact identity which does not seem to have an obvious analogue in our context.
Furthermore, the construction given in this section will then also be useful when 
bounding the second order processes.

By translation invariance, the above quantity is independent
of $x$. Furthermore, the function $\J_\eps$ is positive, so we can bound this integral
by the ``worst case scenario'' where $\phi$ is the indicator function of the unit
ball. This yields the bound
\begin{equ}[e:easybound]
\E |\scal{\phi_x^\lambda, \Psi_\eps}|^{2N} 
\lesssim \lambda^{-8N} \int_{\Lambda^{2N}} {\prod_{\ell < m} \J_\eps(z_\ell-z_m)
\prod_{n < o} \J_\eps(y_n-y_o) \over \prod_{i,j} \J_\eps(z_i - y_j)}\, dz\,dy\;,
\end{equ}
where $\Lambda$ denotes the parabolic ball of radius $\lambda$.
At this stage, we remark that the integrand of this expression consists of $N(N-1)$
factors in the numerator and $N^2$ factors in the denominator. One would hope that
some cancellations take place, allowing this to be bounded by a similar expression, but
with only $N$ terms, all in the denominator.

This is precisely the case and is the content of Corollary~\ref{cor:pointBound} below with $\J$ chosen to be our function $\J_\eps$ defined in \eref{e:defJeps},
which allows us to obtain the bound
\begin{equ}[e:first-order-bound]
\E |\scal{\phi_x^\lambda, \Psi_\eps}|^{2N}
\lesssim \lambda^{-8N} \Bigl|\int_{\Lambda^2} \J_\eps^{-}(x-y)\,dx\,dy\Bigr|^N\;.
\end{equ}
Taking such a bound for granted for the moment, we see that 
as a consequence of \eref{e:divergenceKeps}, one has the bound
\begin{equ}
\Bigl|\int_{\Lambda^2} \J_\eps^{-}(x-y)\,dx\,dy\Bigr| \le \lambda^{8- {\beta^2 \over 2\pi}}\;,
\end{equ}
uniformly in $\eps \in (0,1]$, provided that $\beta^2 < 8\pi$. (Otherwise 
$\J_\eps^{-}$ is no longer uniformly integrable as $\eps \to 0$.)
It follows immediately that
\begin{equ}
\E |\scal{\phi_x^\lambda, \Psi_\eps}|^{2N} \lesssim \lambda^{- {\beta^2 N \over 2\pi}}\;,
\end{equ}
which is the first of the two claimed bounds in \eref{e:boundFirstLevel}. 

For $k\geq 2$ we proceed analogously.
Indeed, it is straightforward to check that
$\E |\scal{\phi_x^\lambda, \Psi_\eps^k}|^{2N}$
is bounded by the right hand side of \eref{e:easybound} 
multiplied by $\eps^{(k^2-1)\beta^2 N \over 2\pi}$,
and with $\J_\eps$ replaced by $\J_\eps^{k^2}$.
Since $\J_\eps^{k^2}$ still satisfies~\eref{e:boundKeps}, 
Corollary~\ref{cor:pointBound} still applies. 
Therefore, together with~\eref{e:divergenceKeps}, one has
\begin{equ}
\E |\scal{\phi_x^\lambda, \Psi_\eps^k}|^{2N}
\lesssim \lambda^{-8N} \Bigl|\int_{\Lambda^2} 
\eps^{\frac{(k^2-1)\beta^2}{2\pi}} \Big(\eps+\|x-y\|_\s\Big)^{-\frac{k^2\beta^2}{2\pi}}
\,dx\,dy\Bigr|^N\;.
\end{equ}
Since $\beta^2<8\pi$ and $k\geq 2$, one can choose $\kappa>0$ sufficiently small
so that
\begin{equ}
(\eps+\|x-y\|_\s)^{-\frac{k^2\beta}{2\pi}} \lesssim 
\|x-y\|_\s^{-\frac{\beta^2}{2\pi}-2\kappa} \eps^{-\frac{(k^2-1)\beta^2}{2\pi}+2\kappa}\;,
\end{equ}
which is still integrable at short scales. Therefore,
\begin{equ}
\E |\scal{\phi_x^\lambda, \Psi_\eps^k}|^{2N}
\lesssim \eps^{2 \kappa N}\lambda^{-(\frac{\beta^2}{2\pi}+2\kappa)N} \;,
\end{equ}
which is precisely the bound \eref{e:bound-kthmode}.

In order to obtain the second bound of~\eref{e:boundFirstLevel}, we first show that the sequence $\scal{\phi_x^\lambda, \Psi_\eps}$
is Cauchy in $L^2(\Omega)$ for every sufficiently regular test function $\phi$ and every space-time
point $x$. For this, we will also need a notation for
\begin{equ}[e:Qepseps]
\CQ_{\eps,\bar \eps}(z) 
\eqdef \E \Phi_\eps(0) \Phi_{\bar \eps}(z) 
= \bigl(\CQ * (\rho_\eps * \CT \rho_{\bar \eps})\bigr)(z)\;,
\end{equ}
and we set analogously to \eref{e:defJeps} $\CJ_{\eps,\bar \eps} = \exp(-\beta^2 \CQ_{\eps,\bar \eps})$.
Note that 
$\CQ_{\eps,\bar \eps} = \CQ_{\bar \eps,\eps}$.
With this notation, a straightforward calculation yields
\begin{equs}
\E |\scal{\phi_x^\lambda, \Psi_\eps - \Psi_{\bar \eps}}|^2 &= \iint \phi_x^\lambda(y)\phi_x^\lambda(y+z)
\bigl(\CJ_\eps^-(z) + \CJ_{\bar \eps}^-(z) - 2 \CJ_{\eps,\bar \eps}^-(z)\bigr)\,dy\,dz\\
&\lesssim \lambda^{-4}\int_\Lambda \bigl|\CJ_\eps^-(z) + \CJ_{\bar \eps}^-(z) - 2 \CJ_{\eps,\bar \eps}^-(z)\bigr|\,dz\;, \label{e:goodBound}
\end{equs}
where $\Lambda$ now denotes a parabolic ball of radius $2\lambda$ centred around the origin. 
At this stage we note that, thanks to \eref{e:Qepseps}, the function $\CQ_{\eps,\bar \eps}$ also falls within the 
scope of Lemma~\ref{lem:reg_kernel_bounds} since one can write
\begin{equ}
\rho_\eps * \CT \rho_{\bar \eps} = \hat \rho_{\eps \vee \bar \eps}\;,
\end{equ}
for a function $\hat \rho$ which in general depends on $\eps$ and $\bar \eps$ but is bounded
(and has bounded support) independently of them.
As a consequence of Lemma~\ref{lem:reg_kernel_bounds}, we can thus write $\CJ_\eps^- = \CJ^-\, \exp(\CM_\eps)$
and $\CJ_{\eps,\bar \eps}^- = \CJ^-\, \exp(\CM_{\eps,\bar \eps})$ where, assuming without loss of generality that
$\bar \eps \le \eps$, the functions $\CM_\eps$
and $\CM_{\eps,\bar \eps}$ are bounded by
\begin{equ}
|\CM_\eps(z)|  + |\CM_{\eps,\bar \eps}(z)| \lesssim {\eps \over \|z\|_\s}
\end{equ}
for all space-time points $z$ with $\|z\|_\s \ge \eps$. 
Since, as a consequence of the first part of Lemma~\ref{lem:reg_kernel_bounds}, one can furthermore bound 
$\CJ_{\bar \eps}^-(z)$ and $\CJ_{\eps,\bar \eps}^-(z)$ by a suitable multiple of $\|z\|_\s^{-{\beta^2 / 2\pi}}$,
it follows immediately that one has the global bound
\begin{equ}
\bigl|\CJ_\eps^-(z) + \CJ_{\bar \eps}^-(z) - 2 \CJ_{\eps,\bar \eps}^-(z)\bigr| \lesssim \|z\|_\s^{-{\beta^2 \over 2\pi}}\Bigl( {\eps \over \|z\|_\s} \wedge 1\Bigr)\;.
\end{equ}
Inserting this bound into \eref{e:goodBound} and integrating over $\Lambda$ eventually yields
\begin{equ}
\E |\scal{\phi_x^\lambda, \Psi_\eps - \Psi_{\bar \eps}}|^2 \lesssim 
\left\{\begin{array}{cl}
	(\eps \wedge \lambda)^{4-{\beta^2 \over 2\pi}}\lambda^{-4} & \text{for $\beta^2 \in (6\pi,8\pi)$,} \\[.5em]
	(\eps\wedge \lambda)\lambda^{-4}\bigl(1 \vee \log (\lambda/\eps)\bigr)  & \text{for $\beta^2 = 6\pi$,} \\
	(\eps\wedge\lambda) \lambda^{-1-{\beta^2 \over 2\pi}} & \text{for $\beta^2 \in (0,6\pi)$.}
\end{array}\right.
\end{equ}
Since the bound $\bigl(1 \vee \log (\lambda/\eps)\bigr) \lesssim \lambda^\alpha (\eps\wedge \lambda)^{-\alpha}$
holds for every $\alpha > 0$, one can summarise these bounds by
\begin{equ}[e:finalBoundPsi]
\E |\scal{\phi_x^\lambda, \Psi_\eps - \Psi_{\bar \eps}}|^2 \lesssim \eps^{2\kappa} \lambda^{-2\kappa-{\beta^2 \over 2\pi}}\;,
\end{equ}
for some (sufficiently small depending on $\beta$) value of $\kappa$. Note that these bounds are independent of $\bar \eps$ 
as long as $\bar \eps \le \eps$. The existence of limiting random variables $\scal{\phi_x^\lambda,\Psi}$ follows immediately.
The second bound in \eref{e:boundFirstLevel} is then a consequence of the first by combining it with \eref{e:finalBoundPsi}
and using the Cauchy-Schwarz inequality. 

Let $\Psi_\eps^\rho,\Psi_{\eps}^{\bar\rho}$ be the processes defined via two mollifiers $\rho,\bar\rho$ respectively, then one has
\begin{equ}
\E |\scal{\phi_x^\lambda, \Psi_\eps^\rho - \Psi_{\eps}^{\bar\rho}}|^2 
= \iint \phi_x^\lambda(y)\phi_x^\lambda(y+z)
\bigl(\CJ_{\eps,\rho}^-(z) + \CJ_{\eps,\bar\rho}^-(z) - 2 \CJ_{\eps,\rho,\bar\rho}^-(z)\bigr)\,dy\,dz 
\end{equ}
where $\CJ_{\eps,\rho,\bar\rho} = \exp(-\beta^2 \CQ_{\eps,\rho,\bar\rho})$ and
$\CQ_{\eps,\rho,\bar\rho}(z) = \bigl(\CQ * \hat\rho_\eps \bigr)(z)$
with $\hat\rho_\eps = (\rho_\eps * \CT \bar\rho_{\eps})$. 
Then it follows in the same way as above that all the moments of 
$\scal{\phi_x^\lambda, \Psi_\eps^\rho - \Psi_{\eps}^{\bar\rho}}$
converge to zero as $\eps\to 0$.
Therefore the limit process $\Psi$ is independent of the mollifier $\rho$ as claimed.
\end{proof}

As an almost immediate corollary we obtain the

\begin{proof}[of Theorem~\ref{theo:conv}]
It only remains to show that the bounds of Theorem~\ref{theo:convBasic} do imply convergence in
probability in $\CC_\s^{-\gamma}$ for every $\gamma > {\beta^2 \over 4\pi}$. This is an easy consequence of 
the characterisation of the space $\CC_\s^{-\gamma}$ in terms of wavelet coefficients 
(see \cite{MR1228209} in the Euclidean case and 
\cite[Prop.~3.20]{Regularity} for the parabolic case considered here),
combined with the same argument as in the standard proofs of Kolmogorov's continuity
test \cite{RevYor}, see also the proof of \cite[Thm~10.7]{Regularity}.
\end{proof}

The proof of Theorem \ref{theo:convBasic} is completed once we show
that \eref{e:divergenceKeps} holds and that \eref{e:easybound} does indeed imply the bound \eref{e:first-order-bound}.
For this, we consider the following general situation. 
We are given $N$ points $x_i \in \R^d$ as well as corresponding signs $\sigma_i \in \{\pm 1\}$,
so that each point can be interpreted as a ``charge'' (either positive or negative).
We are furthermore given a
``potential'' function $\J \colon \R^d \to \R_+$ with the following property.
For every positive constant $c>0$ there exists
a constant $C>0$ such that the implication
\begin{equ}[e:boundKeps]
\|x\|_\s \le c \|\bar x\|_\s \quad \Rightarrow \quad \J(x) \le C \J(\bar x)\;,
\end{equ}
holds for all $x, \bar x$. Here, the scaling $\s$ of $\R^d$ is fixed throughout. In our
case, one has $d=3$ (space-time) and the scaling is the usual parabolic scaling. As before,
we use the notation $\J^-(x) = 1/\J(x)$. Note that if there exists one point such that $\J(x) \neq 0$
(which is something we will always assume), then one necessarily has $\J(x) \neq 0$ for every $x \neq 0$
as a consequence of \eqref{e:boundKeps}.

We then aim at bounding integrals of the type
\begin{equ}
\CI = \int_\Lambda \cdots \int_\Lambda \prod_{i \neq j=1}^N \J^{\sigma_i \sigma_j}(x_i-x_j)\,dx_1\cdots dx_N\;,
\end{equ}
for some fixed domain $\Lambda \subset \R^d$. This is exactly the situation of the right hand side in
\eref{e:easybound} by taking for the $x_i$ the union of the $y_i$ and the $z_i$ and assigning
one sign to the $y_i$ and the opposite sign to the $z_i$.
Assuming that there are $k$ indices with $\sigma_i = 1$ (and therefore
$N-k$ indices with $\sigma_i = -1$) and assuming without loss of generality that $k \le N/2$ (so that $k \le N-k$), we claim that 
\begin{equ}[e:mainBound]
|\CI| \lesssim \Bigl|\int_{\Lambda^2} \J^{-1}(x-y)\,dx\,dy\Bigr|^{k} \bar \J^{\binom{N-2k}{2}} |\Lambda|^{N-2k}\;,
\end{equ}
where $\bar \J = \sup_{\|x\|_\s \le \diam \Lambda} \J(x)$, $|\Lambda|$ denotes the volume of $\Lambda$, 
and the symbol $\lesssim$ hides a proportionality constant depending only on $N$ and on the constants appearing in \eref{e:boundKeps}. As a matter of fact, we will obtain a stronger pointwise bound on the
integrand of \eref{e:boundKeps} from which \eref{e:mainBound} then follows trivially.
Let us first give a brief reality check of \eref{e:mainBound}. In the case 
when $\CJ = 1$ (which does indeed satisfy \eref{e:boundKeps}), both 
$\CI$ and \eref{e:mainBound} are equal to $|\Lambda|^N$. Furthermore, if we multiply $\CJ$ by
an arbitrary constant $K$ (which does not change the bound \eref{e:boundKeps}), then both $\CI$
and \eref{e:mainBound} are multiplied by $K^q$ with $q = 2k(k-N) + N(N-1)/2$.

In order to obtain the pointwise bound mentioned above, we consider 
any configuration of $N$ distinct points $\{x_1,\ldots,x_N\}$ and 
we associate to it a decreasing sequence $\{\CA_n\}_{n \in \Z}$ of 
partitions of $\{1,\ldots,N\}$ in the following way.\footnote{The sequence is decreasing in the sense that $\CA_{n+1}$ is a coarsening (or equal to) $\CA_n$. For $n$ small enough, $\CA_n$ consists of all singletons and is therefore as fine as possible, while for $n$ large enough $\CA_n$ is the coarsest possible partition consisting only of the whole set.} For $n$ small enough
so that $2^n < \min_{i\neq j} \|x_i - x_j\|_\s$, we
take for $\CA_n$ the partition consisting only of singletons, namely
\begin{equ}[e:defAn]
\CA_n = \bigl\{\{1\},\{2\},\ldots,\{N\}\bigr\}\;.
\end{equ}
For every $n$, we furthermore introduce pairings $\CS_n\colon \CA_n \to \power{\CP(N)}$, where $\CP(N)$ denotes the set
of (unordered) pairs of $N$ elements (we interpret this as the set of subsets of $\{1,\ldots,N\}$ of cardinality $2$) 
and $\power E$ is the power set of $E$, in such a way that 
\begin{claim}
\item If $\{i,j\} \in \CS_n(A)$, then $\{i,j\} \subset A$. In other words, for every $A \in \CA_n$,
$\CS_n(A)$ is a subset of the possible pairs constructed by using only elements contained in $A$.
\item If $\{i,j\} \in \CS_n(A)$, then $\sigma_i \neq \sigma_j$, so only pairings of points with opposite signs occur.
\item If $p_1, p_2 \in \CS_n(A)$, then either $p_1 = p_2$ or $p_1 \cap p_2 = \emptyset$, so only disjoint pairings occur.
\item For any $A \in \CA_n$, if $\{i,j\} \subset A \setminus \bigcup \CS_n(A)$, then $\sigma_i = \sigma_j$. In other words, indices
of $A$ that do not belong to any pairing all correspond to the same sign.
The number of such indices will play an important role in the sequel, so we introduce the notation
$T_n(A) = \bigl|A \setminus \bigcup \CS_n(A)\bigr|$. We furthermore denote by $\Sigma_n(A) \in \{\pm 1\}$ the sign
of those indices in $A$ that do not belong to any pairing. (If all indices belong to some pairing, we can use the
irrelevant convention $\Sigma_n(A) = 1$.)
\end{claim}
For values of $n$ sufficiently small so that $\CA_n$ is given by \eref{e:defAn}, we have no choice
but to set $\CS_n(A) = \emptyset$ for every $A \in \CA_n$.

For larger values of $n$, we then define $\CA_n$ and $\CS_n$ inductively in the following way. Given $\CA_{n-1}$, we define an equivalence relation $\sim_n$ between elements of $\CA_{n-1}$ to be the smallest equivalence relation such that if $A, \bar A \in \CA_{n-1}$ are such that
there exist $x \in A$ and $\bar x \in \bar A$ with $\|x-\bar x\|_\s \le 2^{n}$, then $A \sim_n \bar A$.
The partition $\CA_n$ is then defined by merging all $\sim_n$-equivalence classes of $\CA_{n-1}$. 
In other words, $\CA_n$ is the smallest partition with the property that for any $A, \bar A \in \CA_{n-1}$ with
$A \sim_n \bar A$, there exists $B \in \CA_n$ with $A \cup \bar A \subset B$.

The pairing $\CS_n$ is then defined to be \textit{any} pairing satisfying the above properties that is
furthermore compatible with $\CS_{n-1}$ in the sense that if $A \in \CA_{n-1}$ and $\bar A \in \CA_n$
are such that $A \subset \bar A$, then $\CS_{n-1}(A) \subset \CS_n(\bar A)$. Loosely speaking,
we keep the pairings of $\CA_{n-1}$ and, in any situation where a merger creates a set in our partition
containing both positive and negative indices, we pair up as many of them as possible in an arbitrary way.

\begin{remark}
Our construction is such that there exists $n_0$ such that for $n \ge n_0$ the 
partition $\CA_n$ necessarily consists of a single set. At this stage, the only 
information of the construction that we will actually use is the pairing
$\CS_{n_0}$.
\end{remark}

With this construction in mind, our main result is then the following, recalling that $T_n(A) = \bigl|A \setminus \bigcup \CS_n(A)\bigr|$.

\begin{proposition}\label{prop:hierarchical}
Let $\J$ be as above, let $N>0$, let $\{x_1,\ldots,x_N\}$ be an arbitrary collection of distinct points in $\R^d$
and let $\{\sigma_1,\ldots,\sigma_N\}$ be a collection of signs. Let $\CA_n$ and $\CS_n$ be defined
as above. Then, there exists a constant $C$ depending only on $N$ and on the constants appearing in \eref{e:boundKeps}
such that, for every $n \in \Z$ and every $A \in \CA_n$, one has the pointwise bound
\begin{equ}[e:bound]
\Bigl| \prod_{i\neq j \in A} \J^{\sigma_i\sigma_j}(x_i-x_j)\Bigr|
\le C \Bigl(\prod_{\{i,j\} \in \CS_n(A)} \J^{-1}(x_i-x_j)\Bigr) \bar \J_n^{D_n(A)}\;,
\end{equ}
where we have set $D_n(A) = \binom{T_n(A)}{2}$ 
and $\bar \J_n = \sup_{\|x\|_\s \le 2^{n}} \J(x)$.
\end{proposition}

\begin{proof}
The proof goes by induction on $n$. For $n$ sufficiently small so that \eref{e:defAn} holds, both sides
are empty products so the bound holds trivially. Note first that as a consequence of \eref{e:boundKeps} 
$\bar \J_n$ is essentially increasing in $n$ (in the sense that $\J_m \le C \J_n$ for $m \ge n$, where $C$ is independent
of both $m$ and $n$), so that as long as no merger event takes place, the bound \eref{e:bound}
gets weaker with increasing $n$.
It therefore remains to show that the bound still
holds if two (or more) sets merge when going from some level $n$ to level $n+1$. 
Without loss of generality, we assume that only
two sets $A$ and $\bar A$ merge. We also note that losing optimality by a multiplicative factor possibly depending
on $N$ is harmless since there can altogether be only at most 
a fixed number $N-1$ of  merger events.

Using the inductive hypothesis, we then obtain immediately the bound
\begin{equs}[e:merger]
\Bigl| \prod_{i\neq j \in A \cup \bar A} \J^{\sigma_i\sigma_j}(x_i-x_j)\Bigr|
&\lesssim \Bigl(\prod_{\{i,j\} \in \CS_{n}(A) \cup \CS_{n}(\bar A)} \J^{-1}(x_i-x_j)\Bigr) \bar \J_{n}^{D_{n}(A) + D_{n}(\bar A)} \\
&\qquad \times \prod_{i\in A, j \in \bar A} \J^{\sigma_i\sigma_j}(x_i-x_j)\;.
\end{equs}
At this stage we note that since $A$ and $\bar A$ are distinct sets in $\CA_n$, we necessarily have
$\|x_i - x_j\|_s \ge 2^n$ for $i\in A$ and $j \in \bar A$. On the other hand, since the two sets merged at
level $n+1$, there exists a constant $C$ (possibly depending on $N$) such that one has
$\|x_i - x_j\|_\s \le C 2^n$ for any $i,j \in A \cup \bar A$. As a consequence of this and of \eref{e:boundKeps},
there exists a constant $\bar C$ such that for any $i\in A$ and $j \in \bar A$, one has
\begin{equ}[e:equivK]
\bar C^{-1} \bar \J_{n+1} \le \J(x_i - x_j) \le \bar C \bar \J_{n+1}\;.
\end{equ}
As a consequence, denote $A^p = A \setminus \bigcup \CS_n(A)$ and similarly for $\bar A^p$. Then it follows
from \eref{e:equivK} that
\begin{equ}[e:boundprodK]
\prod_{i\in A, j \in \bar A} \J^{\sigma_i\sigma_j}(x_i-x_j) \lesssim \prod_{i\in A^p, j \in \bar A^p} \J^{\sigma_i\sigma_j}(x_i-x_j)\;,
\end{equ}
where the proportionality constant depends on $\bar C$ and $N$ in general. This is because if $i \in A$
and $j$ belongs to some pair in $\CS_n(\bar A)$, then the two factors coming from the two possible values
of $j$ cancel each other out.
More precisely, if $i \in A$
and $\{j,j'\} \in \CS_n(\bar A)$, then by the triangle inequality,
\begin{equ}
\| x_i - x_j \|_\s 
	\lesssim \| x_i - x_{j'} \|_\s + \| x_{j'} - x_j \|_\s 
	\lesssim \| x_i - x_{j'} \|_\s  \;,
\end{equ}
where the last inequality holds because $\| x_{j'} - x_j \|_\s \lesssim 2^n$ 
and $\| x_i - x_{j'} \|_\s \gtrsim 2^{n+1}$. The same bound holds with $j$ and $j'$ interchanged, thus by \eref{e:boundKeps} and $ \sigma_i\sigma_j = - \sigma_i\sigma_{j'} $, one has the cancellation
\begin{equ}
\J^{\sigma_i\sigma_j}(x_i-x_j) \J^{\sigma_i\sigma_{j'}}(x_i-x_{j'}) \lesssim 1 \;.
\end{equ}
(See Figure \ref{fig:merger} for an illustration about the procedure we are following here.)
\begin{figure}[h] 
\centering
\begin{tikzpicture}
\draw (3,0) circle (1.8cm);
\draw (8,0) circle (1.8cm);
\draw (3,2.2) node {$A$};
\draw (8,2.2) node {$\bar A$};
\node at (3,1) [charge] (1) {$-$};  \node at (3.3,1.3) {a};
\node at (2,0.5) [charge] (2) {$+$}; \node at (1.7,0.2) {b};
\node at (3.5,-0.5) [charge] (3) {$+$}; \node at (3.2,-0.7) {c};
\node at (9,0.5) [charge] (4) {$-$}; \node at (9.3,0.3) {d};
\node at (8,1) [charge] (5) {$+$}; \node at (7.6,1) {e};
\node at (7.5,-1) [charge] (6) {$-$}; \node at (7.9,-1) {f};
\node at (7,0) [charge] (7) {$-$}; \node at (7.4,0) {g};
\draw [kerAlg] (1) to node [above] {$\J^-$} (2);
\draw [kerAlg] (4) to node [below] {$\J^-$} (5);
\draw [kerAlg2] (7) to node [homo,above] {$\J$} (1);
\draw [kerAlg2] (7) to node [homo,below] {$\J^-$} (2);
\draw [kerAlg] (6) to node [below] {$\J^-$} (3);
\end{tikzpicture}
\caption{This illustrates a situation where $A,\bar A\in \CA_n$ 
are merged into $A\cup \bar A \in \CA_{n+1}$, with
$\Sigma_n(A) \neq \Sigma_n(\bar A)$. 
In this case, $\CS_n(A)=\{\{a,b\}\}$ and $\CS_n(\bar A)=\{\{d,e\}\}$.
The factors $\J^-$  drawn on $\{a,b\}$ and $\{d,e\}$ correspond to
the factors $\J^-$  in the first line of the right hand side of \eref{e:merger}.
The two dashed lines correspond to two of the factors in the second line of  \eref{e:merger}, and they ``almost cancel" each other out since $g$
is far away from $\{a,b\}$. There are many other such cancellations which we didn't draw.
The pair $\{c,f\} \in \CS_{n+1}$ is a new pair formed at this step but we could just
as well have chosen to form $\{c,g\}$ instead.
As for the factors $\bar \J$, we have $D_n(A)=0$, $D_n(\bar A)=1$,
and $D_{n+1}(A\cup \bar A)=0$, which is less than $D_n(A)+D_n(\bar A)$
by $1$ due to the newly formed factor $\J^-(x_c - x_f)$.}
\label{fig:merger}
\end{figure}
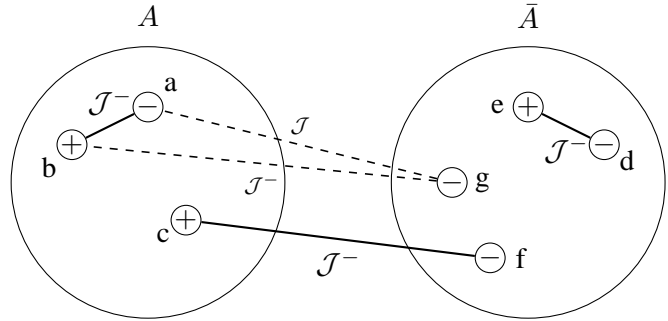

There are now two cases: either one has $\Sigma_n(A) = \Sigma_n(\bar A)$, or 
one has $\Sigma_n(A) \neq \Sigma_n(\bar A)$. We first consider the case $\Sigma_n(A) = \Sigma_n(\bar A)$.
In this case, one necessarily has $\CS_{n+1}(A\cup \bar A) = \CS_n(A)$, so that in view of \eref{e:boundprodK}
we only need to show that
\begin{equ}[e:wanted1]
\prod_{i\in A^p, j \in \bar A^p} \J^{\sigma_i\sigma_j}(x_i-x_j) \lesssim \bar \J_{n}^{D_{n+1}(A\cup \bar A) - D_{n}(A) - D_{n}(\bar A)}\;.
\end{equ}
Since each factor in the product on the left is bounded by some multiple of $\bar \J_n$, this follows at once from the fact that
the number of terms on the left is equal to
\begin{equ}
|A^p| \,|\bar A^p| = |T_n(A)|\, |T_n(\bar A)|\;.
\end{equ}
Writing $a = |T_n(A)|$ and $\bar a = |T_n(\bar A)|$ as a shorthand, the exponent on the right hand side of
\eref{e:wanted1} is equal to
\begin{equ}
{(a+\bar a)(a+\bar a - 1) - a(a-1) - \bar a (\bar a-1)\over 2}  = {2 a \bar a \over 2}\;.
\end{equ}
Since both exponents are the same, the claim follows at once.

We now deal with the case $\Sigma_n(A) \neq \Sigma_n(\bar A)$. Using the same shorthands $a$, $\bar a$ as above,
we note that this time $\CS_n(A \cup \bar A)$ is given by $\CS_{n-1}(A) \cup \CS_{n-1}(\bar A)$, plus $a\wedge \bar a$
of the $a \bar a$ pairs appearing in \eref{e:boundprodK}. Assuming without loss of generality that $a \le \bar a$,
the number of remaining factors is given by $a(\bar a -1)$. This time furthermore each factor contributes one
\textit{negative} power of $\J_n$, so it remains to show that
\begin{equ}
D_{n+1}(A\cup \bar A) - D_{n}(A) - D_{n}(\bar A) = - a(\bar a -1)\;.
\end{equ}
Since this time around
\begin{equ}
D_{n+1}(A\cup \bar A) = {(\bar a - a)(\bar a - a - 1)\over 2}\;,
\end{equ}
this identity follows at once, thus concluding the proof.
\end{proof}

\begin{corollary}\label{cor:pointBound}
The bound \eref{e:mainBound} holds. 
\end{corollary}

\begin{proof}
It suffices to note that as soon as $2^n > \diam \Lambda$, one has $\CA_n = \{\{1,\ldots,N\}\}$ and therefore \eref{e:bound}
implies 
\begin{equ}
\prod_{i\neq j} \J^{\sigma_i\sigma_j}(x_i-x_j)
\le C \sum_\CS \Bigl(\prod_{\{i,j\} \in \CS} \J^{-1}(x_i-x_j)\Bigr) \bar \J^{\binom{N-2k}{2}}\;,
\end{equ}
where the sum runs over \textit{all} possible ways $\CS$ 
of pairing the $k$ indices corresponding to a 
positive sign with $k$ of the indices corresponding to a negative sign. Since there are only finitely many
such pairings, the claim follows by integrating both sides of the inequality.
\end{proof}

We still have to prove that \eref{e:boundKeps} actually holds for our $\J_{\eps}$
defined in \eref{e:defJeps}. 
In order to study the behaviour of such kernels we introduce the
following notation. For continuous function $f,\CQ$ on $\R^{d}\setminus\{0\}$
we write 
\begin{equ}[e:defsim]
\CQ(z) \sim f(z)
\quad \mbox{if} \quad
f(z) + c_{1} \leq \CQ(z) \leq f(z) + c_{2}\;,
\end{equ}
for some constants $c_{1},c_{2}$ and for all $z\in\R^{d}\setminus\{0\}$. We will also
sometimes specify that $\CQ \sim f$ on some domain, in which case it is understood 
that \eref{e:defsim} is only required to hold there.
Given $\CQ$, we define $\CQ_\eps$ as in \eref{e:CQeps} by
\begin{equ}
\CQ_\eps = \CQ * \bar \rho_\eps\;,
\end{equ}
where $\bar \rho$ is a mollifier supported in the ball of radius $1$ and $\eps \in (0,1]$.
The following lemma shows that if 
$\CQ(z) \sim -\frac{1}{2\pi} \log \left\Vert z \right\Vert_{\s}$
then its regularization 
$\CQ_{\eps}$
also satisfies suitable upper and lower bounds.
Note that if $\rho$ integrates to $1$, then so does $\bar\rho = \rho * \CT \rho$.

\begin{lemma}
\label{lem:reg_kernel_bounds}
Assume that $\CQ$ is compactly supported, smooth away from $0$, and such that 
$\CQ(z)\sim-\frac{1}{2\pi}\log\left\Vert z\right\Vert _{\s}$. Assume furthermore that $\bar \rho$ is
any continuous function supported on the unit ball around the origin integrating to $1$,
and that $\CQ_\eps$ is as in \eref{e:CQeps}.
Then, for $\|z\|_\s \le 1$, one has the two-sided bound
\begin{equ}[e:Ceps-bounds]
\CQ_{\eps}(z)\sim-\frac{1}{2\pi}\log(\left\Vert z\right\Vert _{\s}+\eps) \;.
\end{equ}
If furthermore $\CQ$ is of class $\CC^1$ and there 
exists a constant $C$ such that $|\d_i \CQ(z)| \le C / \|z\|_\s^{\s_i}$, 
then there exists a constant $\bar C$ such that
\begin{equ}[e:Ceps-bound2]
|\CQ_{\eps}(z) - \CQ(z)| \le \bar C \Bigl({\eps \over \|z\|_\s} \wedge \Bigl(1 + \Bigl|\log {\eps \over \|z\|_\s}\Bigr|\Bigr)\Bigr) \;,
\end{equ}
for all space-time points $z$.
\end{lemma}

\begin{proof}
We omit the proof since it is a rather straightforward calculation.
\end{proof}

It remains to show that the function $\CQ$ does indeed enjoy the properties we took for
granted in Lemma~\ref{lem:reg_kernel_bounds}. Since these properties are invariant under the addition of a smooth 
compactly supported function (as a matter of fact, it only needs to be $\CC^1$), we will 
use the symbol $R$ to denote a generic such function which can possibly change from
one line to the next. Recall that in a distributional sense one has the identity
\begin{equ}[e:propK]
\d_t K - \frac{1}{2}\Delta K = \delta + R\;,
\end{equ}
and that $\CQ$ is given by
\begin{equ}
\CQ(z) = \int K(z + \bar z) K(\bar z)\,d\bar z\;,
\end{equ}
where $z= (t,x)$ and $\bar z = (\bar t, \bar x)$ are space-time points in $\R^3$.
As a consequence of \eref{e:propK}, we then have the distributional identity
\begin{equ}
{1\over 2}\Delta \CQ(z)  = {1\over 2} \int (\Delta K)(z + \bar z) K(\bar z)\,d\bar z = \int (\d_t K)(z + \bar z) K(\bar z)\,d\bar z - K(-z) + R\;,
\end{equ}
where we used the fact that the convolution of $R$ with $K$ is a new function $R$ with the 
same properties. On the other hand, making the substitution $\bar z \mapsto \bar z - z$ we can write
\begin{equ}
\CQ(z) = \int K(\bar z - z) K(\bar z)\,d\bar z\;,
\end{equ}
so that
\begin{equs}
{1\over 2}\Delta \CQ(z)  &= {1\over 2} \int (\Delta K)(\bar z - z) K(\bar z)\,d\bar z = \int (\d_t K)(\bar z - z) K(\bar z)\,d\bar z - K(z) + R\\
&= \int (\d_t K)(\bar z) K(\bar z + z)\,d\bar z - K(z) + R\;.
\end{equs}
At this stage, we note that
\begin{equ}
(\d_t K)(z + \bar z) K(\bar z) + (\d_t K)(\bar z) K(\bar z + z) = \d_{\bar t} \bigl(K(z+\bar z) K(\bar z)\bigr)\;,
\end{equ}
which integrates to zero. Therefore, summing these two expressions yields the identity
\begin{equ}[e:DeltaQ]
\Delta \CQ(z) = K(z) + K(-z) + R\;,
\end{equ}
for some smooth and compactly supported function $R$.
Let now
\begin{equ}
\hat K(z) = K(z) + K(-z)\;,\qquad G(x) = -{1\over 2\pi} \log |x|\;,
\end{equ}
for $z \in \R^3$ and $x \in \R^2$. Then, one has

\begin{lemma}\label{lem:formCQ}
One has the identity
\begin{equ}[e:formCQ]
\CQ(t,x) = \bigl(\hat K(t,\cdot) * G\bigr)(x) + R\;,
\end{equ}
for some smooth function $R$.
\end{lemma}

\begin{proof}
As an immediate consequence of the definition of $\CQ$, the properties of $K$ 
and, for example, \cite[Lemma~10.14]{Regularity}, we know that, for any
$t \neq 0$, $\CQ(t,\cdot)$ is a smooth compactly supported function. 
This immediately implies that one has the identity
\begin{equ}[e:decompQ]
\CQ(t,x) = \bigl(\Delta \CQ(t,\cdot) * G\bigr)(x)\;,
\end{equ}
and the claim follows at once from \eref{e:DeltaQ}.
\end{proof}

\begin{lemma}\label{lem:C_log}
The kernel $\CQ$ can be decomposed as
\begin{equ}
\CQ(z) = -{1\over 2\pi} \log \|z\|_\s 
	+  \hat \CR\Big({t \over \|z\|_{\s}^2},{x \over \|z\|_{\s}}\Big) + R(z)\;,
\end{equ}
where both $R$ and $\hat \CR$ are smooth functions of $\R^3$ and $z = (t,x)$ as before.
In particular, it satisfies the assumptions of both parts of Lemma~\ref{lem:reg_kernel_bounds}.
\end{lemma}

\begin{proof}
Let $H$ be the heat kernel
$H(t,x) = (4\pi|t|)^{-1} \exp\bigl(-|x|^{2} / (4|t|)\bigr)$,
and set $\hat \CQ(t,x) = \bigl(H(t,\cdot) * G\bigr)(x)$.
Then, as a consequence of Lemma~\ref{lem:formCQ} and the fact that $H$ and $\hat K$ differ 
by a smooth function by definition, $\hat \CQ$ and $\CQ$ only differ by a smooth function, so it is sufficient
to show the result for $\hat \CQ$. For this, note that as a consequence of the scaling relation 
$H(\lambda^2 t, \lambda x) = \lambda^{-2}\, H(t,x)$, one has the identity
\begin{equs}
\hat \CQ(\lambda^2 t, \lambda x) &= -{1\over 2\pi} \int_{\R^2} H(\lambda^2 t, \lambda x - y) \log |y|\,dy \\
&= -{\lambda^2 \over 2\pi} \int_{\R^2} H(\lambda^2 t, \lambda x - \lambda y) \log |\lambda y|\,dy \\
&= -{1 \over 2\pi} \int_{\R^2} H( t, x - y) \,\bigl(\log |y| + \log \lambda\bigr)\,dy 
= \hat \CQ(t,x) -{1 \over 2\pi} \log \lambda \;.
\end{equs}
Here we also used the fact that $H(t,\cdot)$ integrates to $1$ for any fixed $t$.
It follows immediately that if we set
\begin{equ}
\hat \CR(z) = \hat \CQ(z) + {1\over 2\pi} \log \|z\|_\s\;, 
\end{equ}
then $\hat \CR$ is smooth outside the origin and homogeneous of order
$0$ in the sense that $\hat \CR(\lambda^2 t, \lambda x) = \hat \CR(t,x)$.
The claim then follows at once.
\end{proof}

Lemma~\ref{lem:const} is now an immediate corollary of this fact.

\begin{proof}[of Lemma~\ref{lem:const}]
Using the decomposition of Lemma~\ref{lem:C_log}, the identity \eref{e:CQeps}, and the fact that $\bar \rho$ integrates to $1$,
a straightforward calculation shows that we have the identity
\begin{equ}
\CQ_\eps(0) = -{1\over 2\pi} \log \eps + \hat C_\rho + \int_{\R^3} \bar \rho_\eps(z) \bigl(R(z) - R(0)\bigr)\,dz\;,
\end{equ}
with
\begin{equ}
\hat C_\rho \eqdef \int_{\R^3} \bar \rho(z) \Bigl(\hat \CR(z) - {1\over 2\pi} \log \|z\|_\s\Bigr)\,dz + R(0) \;.
\end{equ}
Since $\bar \rho$ is necessarily symmetric under $z \mapsto -z$, it annihilates linear functions so that 
$\int_{\R^3} \bar \rho_\eps(z) \bigl(R(z) - R(0)\bigr)\,dz = \CO(\eps^2)$ as claimed.
\end{proof}

\begin{corollary}\label{cor:Jeps}
The estimates \eref{e:divergenceKeps} and \eref{e:boundKeps} for $\J_{\eps}$ hold
for all $\eps>0$ with proportional constants independent of $\eps$.
\end{corollary}
\begin{proof}
By Lemmas~\ref{lem:reg_kernel_bounds} and \ref{lem:C_log}, 
if $\left\Vert z\right\Vert _{\s}\leq c\left\Vert \bar{z}\right\Vert _{\s}$,
we obtain the bound
\begin{equ}
\J_{\eps}(z) 
  = e^{-\beta^2 \CQ_{\eps}(x,t)}
  \lesssim e^{\frac{\beta^2}{2\pi}
  \log \left( \left\Vert z\right\Vert _{\s} + \eps \right)}
    \lesssim e^{\frac{\beta^2}{2\pi} \log \left( \left\Vert \bar{z}\right\Vert _{\s} + \eps \right)}   \lesssim \J_{\eps}(\bar{z})\;,
\end{equ}
thus concluding the proof of \eref{e:boundKeps}.
The estimate \eref{e:divergenceKeps} is just a rewriting of the first conclusion of Lemma~\ref{lem:reg_kernel_bounds}.
\end{proof}

\section{Second-order process bounds for $k=l$} \label{sec:second-order}

In order to provide a solution theory for \eref{e:model} at or beyond $\beta^2=4\pi$,
we have seen in the introduction that one should construct suitable ``second order'' 
objects \eref{e:Psi-ab}. In this section we consider a closely related second order object
\begin{equ}
\tilde \Psi^{\pm}_\eps (z,\bar z) = \Psi_\eps (\bar z) 
\big( (K*\bar \Psi_\eps) (\bar z)-(K*\bar \Psi_\eps) (z) \big)
 -  \E \bigl(\Psi_\eps (K* \bar \Psi_\eps)\bigr)\;.
\end{equ}
Generally,  we define for $1\leq k,l \leq Z$
\begin{equ}[e:def-tildePsiKbarL]
\tilde \Psi^{k\bar l}_\eps (z,\bar z) \eqdef
 \Psi_\eps^k (\bar z) 
\big( 
	(K*\bar \Psi_\eps^l) (\bar z)
	- (K*\bar\Psi_\eps^l) (z) 
\big)
 - \delta_{k,l} \E \bigl(\Psi_\eps^k (K* \bar\Psi_\eps^l)\bigr)\;,
\end{equ}
where $\delta_{k,l} = 1$ if $k=l $ and equals $0$ otherwise (see \eref{e:Psipm2} below about the definition of a variation of the above objects, written as $\Psi^{k\bar l}_\eps$). We also define
\begin{equ}[e:def-PsiKL]
\Psi^{kl}_\eps (z,\bar z) \eqdef
 \Psi_\eps^k (\bar z) 
\big( (K*\Psi_\eps^l) (\bar z)-(K*\Psi_\eps^l) (z) \big) \;.
\end{equ}
The objects $\tilde \Psi^{k\bar k}_\eps$ are the hardest ones to bound,
so we will first obtain bounds for them. The corresponding bounds on
$\tilde \Psi^{k\bar l}_\eps$ with $k\neq l$ and
on $\Psi^{kl}_\eps$ will then be shown in the very end of this section.


The last  term of \eref{e:def-tildePsiKbarL} is a renormalisation constant which,
 for the case $\tilde \Psi^{\pm}_\eps$, can also be expressed as
\begin{equ}
\E \bigl(\Psi_\eps (K* \bar \Psi_\eps)\bigr) = \int K(x) \J_\eps^-(x)\,dx\;.
\end{equ}
As a consequence of \eref{e:divergenceKeps} and the behaviour
of the heat kernel, this diverges as $\eps \to 0$ as soon as $\beta^2 \ge 4\pi$.
When $\beta^2 = 4\pi$, this divergence is logarithmic, and it behaves like
$\eps^{2-{\beta^2 / 2\pi}}$ for $\beta^2 \in (4\pi, 8\pi)$. 
For general  $\tilde \Psi^{k\bar l}_\eps$ with $(k,l)\neq (1,1)$,
one can verify that
\begin{equs}[e:renConst_kl]
\E \bigl(\Psi_\eps^k (K* \bar \Psi_\eps^l)\bigr) 
	& = e^{-\beta^2 \left(\frac{k^2 + l^2}{2} - 1 \right)Q_\eps(0)} 
		\int K(x) \J_\eps(x)^{-kl}\,dx \\
	& \lesssim \eps^{ \left(\frac{k^2 + l^2}{2} - 1 \right) \frac{\beta^2}{2\pi}}
		\int K(x) \left(\|x\|_\s+\eps \right)^\frac{-kl\beta^2}{2\pi}\,dx \\
	 & \lesssim \eps^\kappa  
	 	\int K(x) \|x\|_\s^{\left(-kl+\frac{k^2 + l^2}{2} -1\right)\frac{\beta^2}{2\pi}-\kappa} \,dx
\end{equs}
for sufficiently small $\kappa>0$, where we used the fact 
\begin{equ}
\left(\|x\|_\s+\eps \right)^\frac{-kl\beta^2}{2\pi} 
 \lesssim  \|x\|_\s^{\left(-kl+\frac{k^2 + l^2}{2} -1\right)\frac{\beta^2}{2\pi}-\kappa}
 	\eps^{ - \left(\frac{k^2 + l^2}{2} - 1 \right) \frac{\beta^2}{2\pi} +\kappa}
\end{equ}
for $(k,l)\neq(1,1)$.  Now we note that if $k\neq l$, this integral is finite for all $\beta^2<8\pi$ as long as $\kappa>0$ is sufficiently small, so that the above expectation converges to zero as $\eps\to 0$. 
On the other hand, if $k=l$, 
	it is easy to check 
	(by the first line of~\eref{e:renConst_kl} 
	and dividing the integration into $\|x\|_\s \le \eps$ and $\|x\|_\s > \eps$) 
	that $\E \bigl(\Psi_\eps^k (K* \bar\Psi_\eps^k)\bigr)$ diverges when $\eps \to 0$ for $\beta^2 \ge 4\pi$, with the same rates as in the case $(k,l)=(1,1)$.
This motivates \eqref{e:def-tildePsiKbarL}, namely that there is only renormalisation 
in the second order object $\tilde\Psi^{k\bar l}$ when  $k=l$. 
For the case of $\Psi^{k l}$, it will be clear in the end of this section that one does not 
need any renormalisation.

Instead of considering $\tilde \Psi^{k\bar l}_\eps$, 
it turns out to be more convenient to consider
	the process $\Psi^{k\bar l}_\eps$ given by
\begin{equs}[e:Psipm2]
\Psi^{k\bar l}_\eps (z,\bar z)
= & \int(K(\bar z-w) - K(z-w)) \\
  & \qquad \times 
    \Bigl(\Psi^k_{\eps}(\bar z)  \bar \Psi^l_{\eps}(w) 
	- \delta_{k,l} \E \big( \Psi^k_{\eps}(\bar z)  \bar\Psi^l_{\eps}(w) \big)
\Bigr)\,dw\;,
\end{equs}
where 
\begin{equ}
\E \left( \Psi^k_{\eps}(\bar z)\bar{\Psi}^k_{\eps}(w) \right)
= e^{-\beta^2 \left(k^2 - 1 \right)
		Q_\eps(0)}\J_\eps(\bar z-w)^{-k^2} \;,
\end{equ}
which is simply equal to $\J_\eps^-(\bar z-w)$ when $k= 1$.
With this notation, one has the identity
\begin{equ}
\tilde\Psi_\eps^{k\bar l} (z,\bar z) = \Psi_\eps^{k \bar l} (z,\bar z) 
- \delta_{k,l} F_\eps^{(k)}(\bar z - z)\;,
\end{equ}
where $F_\eps^{(k)}$ is given by
\begin{equ}
F_\eps^{(k)} \eqdef 
	e^{-\beta^2 \left(k^2 - 1 \right)Q_\eps(0)}  \CT K * \CJ_\eps^{-k^2}\;.
\end{equ}
For $k= 1$, we simply write $F_\eps \eqdef  F_\eps^{(1)} = \CT K * \CJ_\eps^-$.
Regarding the functions $F_\eps^{(k)}$, we have the following lemma.

\begin{lemma} \label{lem:Feps}
Let $\beta^2 \in [4\pi, 8\pi)$ and let $F_\eps^{(k)}$ be defined
as above. 
Then, for every sufficiently small $\kappa > 0$, the bounds
\begin{equ}
|F_\eps(z)| \lesssim \|z\|_\s^{2-{\beta^2 \over 2\pi}-\kappa}\;,\qquad 
|F_\eps(z) - F_{\bar \eps}(z)| \lesssim \eps^\kappa\,\|z\|_\s^{2-{\beta^2 \over 2\pi}-\kappa}\;,
\end{equ}
and, for $k \ge 2$,
\begin{equ}
|F_\eps^{(k)}(z)| \lesssim \eps^\kappa  \|z\|_\s^{2-{\beta^2 \over 2\pi}-\kappa}\;,
\end{equ}
hold uniformly over $z$ and over $0 < \bar \eps < \eps < 1$.
\end{lemma}

\begin{proof}
In view of Corollary \ref{cor:Jeps}, the first bound is an immediate corollary of \cite[Lemma~10.14]{Regularity}. For the second bound, 
as in the proof of Theorem \ref{theo:convBasic}, one has
\begin{equ}
|\J_\eps^- - \J_{\bar \eps}^-| \lesssim \|z\|_\s ^{-{\beta^2 \over 2\pi}}
\Big({\eps \over \|z\|_\s} \wedge 1\Big)\;.
\end{equ}
Since ${\eps \over \|z\|_\s} \wedge 1 \le \eps^\kappa\|z\|^{-\kappa}$
for every sufficiently small $\kappa>0$,  the second bound
follows again by \cite [Lemma~10.14]{Regularity}.

For the cases $k\geq 2$, one has the bound
\begin{equ}
\CJ_\eps(z)^{-k^2} \lesssim 
	\eps^{ -\left(k^2 - 1 \right)\frac{\beta^2}{2\pi} +\kappa}  
	\|z\|_\s^{-\frac{\beta^2}{2\pi}-\kappa}\;,
\end{equ}
and the bound for $F_\eps^{(k)}$ follows immediately again from \cite [Lemma~10.14]{Regularity}.
\end{proof}

\begin{remark} \label{rem:init-2ndorder}
As an immediate corollary, we conclude that if the bounds 
	\eref{e:convModel2} hold for $\Psi_\eps^{k\bar l}$ defined in \eref{e:Psipm2} 
	and for $\Psi_\eps^{k l}$ defined in \eref{e:def-PsiKL},
then they also hold for $\Psi_\eps^{ab,kl}$ defined in \eref{e:Psi-ab},
with 
\begin{equ}
C_\eps^{(k)} = \int K(z) 
\E \left( \Psi^k_{\eps}(0)\bar{\Psi}^k_{\eps}(z) \right) dz \;.
\end{equ}
\end{remark}

The main technical result of this article is as follows, where
we write $\Psi^{\pm}$ as a shorthand for $\Psi^{1\bar 1}$
and $\Psi^{\oplus}$ as a shorthand for $\Psi^{11}$.

\begin{theorem}\label{theo:second-order}
Assume that $\beta^2 \in [4\pi, 6\pi)$.
There exist stationary random complex 
distribution-valued processes $\Psi^\pm$ and $\Psi^\oplus$,
such that
\begin{equ}
	\Psi^\pm_\eps \to \Psi^\pm \;,  \qquad
	\Psi^\oplus_\eps \to \Psi^\oplus  \;,  \qquad
	\Psi^{k\bar l}_\eps \to 0  \;,  \qquad
	\Psi^{kl}_\eps \to 0
\end{equ}
 for all $(k,l)\neq (1,1)$  in probability. Furthermore, 
for every $\delta, \kappa>0$ sufficiently small and $m\in\Z_+$, one has
\begin{equ}
	\E |\scal{\phi_z^\lambda, \Psi^{a}_\eps}|^m
		\lesssim  \lambda^{(2-\frac{\beta^2}{2\pi}-\delta)m}\;,
	\qquad
	\E |\scal{\phi_z^\lambda, \Psi^{a}_\eps - \Psi^{a}}|^m
		\lesssim \eps^{\kappa} 
		\lambda^{(2-\frac{\beta^2}{2\pi}-\delta)m-\kappa}\;,
\end{equ}
where $a\in \{\pm,\oplus\}$, and
\begin{equ} [e:boundkbark]
\E |\scal{\phi_z^\lambda, \Psi^{k\bar l}_\eps}|^m
		\lesssim \eps^{\kappa}  \lambda^{(2-\frac{\beta^2}{2\pi}-\delta)m-\kappa}\;,
\quad
\E |\scal{\phi_z^\lambda, \Psi^{k l}_\eps}|^m
		\lesssim \eps^{\kappa}  \lambda^{(2-\frac{\beta^2}{2\pi}-\delta)m-\kappa}\;,
\end{equ}
for $(k,l)\neq (1,1)$,
uniformly over all test functions $\phi$ supported in the unit ball and bounded by $1$, all $\lambda\in(0,1]$, and all space-time points $z$.
\end{theorem}
We remark that the complex conjugates of these processes of course also have the corresponding bounds and convergence results. The remainder of this article
is devoted to the proof of Theorem~\ref{theo:second-order}. We will treat separately
the cases $kk$, $k\bar k$, $k l$, and $k\bar l$ for $k \neq l$. The first two
cases are all that is required for the treatment of \eqref{e:model}, and these form the
remainder of this section. The last section is then devoted to the proof of
the above bounds for the last two cases.

\begin{remark}
We actually expect that the above bounds hold for all $\beta^2 \in [4\pi,8\pi)$.
The second order object could in principle be constructed below $8\pi$,
and $6\pi$ would just be a threshold where it becomes necessary 
to construct even higher order objects in order to study our equation.
The reason that we assume $\beta^2 < 6\pi$ here is because
the analysis in the following will be not as sharp as possible, see Remark 
\ref{rem:not-sharp} below.
We choose to do so for simplicity since we are here only interested in
solving the equation for $\beta^2 < 16\pi/3 < 6\pi$ anyway.
\end{remark}

As a corollary (see Remark \ref{rem:Kolmogorov}), the bounds 
\eref{e:convModel2} hold for $\Psi_\eps^\pm$, $\Psi_\eps^\oplus$,
and the bounds \eref{e:convModel3} hold for $\Psi_\eps^{k\bar l}$, $\Psi_\eps^{k l}$ 
for $(k,l)\neq (1,1)$, therefore Assumption~\ref{assumption:A} is justified.
The rest of this section devoted to the proof of Theorem \ref{theo:second-order}. 
By translation invariance, we only need to show the above bounds for $z=0$.

\subsection{Moments of $\Psi_\eps^{k\bar k}$: renormalisations}
\label{subsec:PsiKBK-ren}

Let us start from the most important case: the moments for $\Psi_\eps^{k\bar l}$ with $k=l$.
By definition in \eref{e:Psipm2}, the $m$-th moment with $m = 2N$ an even integer
can be expressed as
\begin{equs}  \label{e:first-formula}
\E |\scal{\phi_0^\lambda, \Psi^{k\bar k}_\eps}|^m
& =
\E \,\Big[ \Big| \iint \phi_0^\lambda(x) \,  (K(x-y)-K(-y)) \\
	&\qquad
	\times \Big( \Psi^k_\eps(x) \bar\Psi^k_\eps(y)
	- \E\big[\Psi^k_\eps(x) \bar\Psi^k_\eps(y)\big] 
		\Big) dx dy  \Big|^{2N}  \Big]  \;.
\end{equs}
We will rewrite this expression as 
an integral over $4N$ variables. Observe that half of these $4N$ variables
will be arguments of $\Psi_\eps^k$, and the other half will be arguments of $\bar\Psi_\eps^k$.
Also, these $4N$ variables appear as arguments of $K(x-y)-K(-y)$ by pairs, in such a way that the
roles played by $x$ and $y$ are {\it not} symmetric. Based on these simple observations we introduce 
the following terminologies and notations.
 
\begin{claim}
\item
Fix two integers $1\le k,l \le Z$. We will say that we are given 
	a finite number $2m = 4N$ of charges (where $N\in\Z$),
by which we mean points in $\R^3$
	endowed with a sign, as well as an index $h\in\{k,l\}$. 
We impose that exactly $2N$ of these charges have a positive sign
(corresponding to the arguments of $\Psi_\eps^k$),
and the other $2N$ charges have a negative sign
(corresponding to the arguments of $\bar\Psi_\eps^k$).
\item
We denote by $M$ a set of labels with cardinality $2m$ and, 
given $j \in M$, we write $x_j \in \R^3$ for the location
of the corresponding charge, $\sigma_j$ for its sign,
and $h_j$ for its index.
In this section, we will only consider the case $k=l$, 
namely all the charges have the same index $k$. 
We therefore do not make any reference to this index anymore until Section~\ref{sec:kneql}.
\item
These $4N$ charges are furthermore partitioned into $2N$ 
distinct \textit{oriented} pairs with each pair consisting of one positive and one negative charge.
Here, an \textit{oriented} pair consists of two charges, 
with one of them called the outgoing point and the other one called the incoming point.
Given two charges $i$ and $j$, we write $i\to j$ for the oriented pair with outgoing point $i$ and incoming point $j$.
We denote this set of oriented pairs by $\CR$ and we impose that $\CR$ is such that exactly $N$ pairs
are oriented from the positive to the negative charge and $N$ pairs are oriented the other way
around. \footnote{This is a reflection of the fact that, in \eref{e:first-formula}, half of the factors involve the complex conjugate.}
\item
Generally, given an arbitrary oriented pair of charges $e$, we say that it is \textit{renormalised} if $e \in \CR$.
Given a pair $e \in \CR$, we write 
$e^+$ (resp.\ $e^-$) for the point in $e$ with positive (resp.\ negative) charge,
and $e_\uparrow$ (reps.\ $e_\downarrow)$ for the outgoing point (resp.\ incoming point) of $e$, in other words $e = e_\uparrow \to e_\downarrow$.
\end{claim}

\begin{remark}
In order to shorten our expressions, we will sometimes identify a charge $i$ with
the corresponding coordinate $x_i \in \R^3$. For example, if $\J$ is a function defined on $\R^3$
and we write $\J(e^+ - e^-)$, this is a shorthand for $\J(x_{e^+} - x_{e^-})$.
\end{remark}

For any oriented pair $e \in \CR$, we use the shorthand notation
\begin{equs}
	  K(e)  & \eqdef   K(e_\downarrow-e_\uparrow) - K(-e_\uparrow) \;.
\end{equs} 
Then, using the notations introduced above,  one can rewrite the right hand side of \eref{e:first-formula} as
\begin{equ}  [e:over-4N]
\int_{(\R^3)^M} 
\E \,\Big[  \prod_{e\in\CR}
	\Big( \Psi^k_\eps(e^+) \bar\Psi^k_\eps(e^-)
	- \E\big[\Psi^k_\eps(e^+) \bar\Psi^k_\eps(e^-)\big] 
		\Big) \Big]
\prod_{e \in \CR} \Big(
	\phi_0^\lambda(e_\downarrow) K(e)
\Big)
\,dx\;.
\end{equ}

We will now expand the first product over $e\in \CR$, which amounts to assignment of a subset $P\subset \CR$ to the first term and $\CR\backslash P$ to the second term. This motivates us to further introduce the following notations.
For any subset $P \subset \CR$, we write $P' = \bigcup P$ for the set of
all charges appearing in the pairs in $P$.
Given subsets $A \subset M$, we write $\CE(A)$ for the set of all pairs 
$\{i,j\}$ with $i,j \in A$. Here, the pairs are not oriented, and the two charges
in any such pair are not necessarily of opposite signs.
Finally, for any pair $e = \{i,j\}$ and a symmetric function $\J:\R^3\to \R_+$, we write 
\begin{equ}[e:J_e-kbark]
\J_e \eqdef
		\J (x_i - x_j) \;,\qquad 
\hat\J_e \eqdef
		\J (x_i - x_j)^{\sigma_i \sigma_j} \;.
\end{equ}
Note that for this particular notation it does not matter whether an orientation 
is specified for $e$ since $\J$ is symmetric.

Given again a function $\J$ as above and any subset $P \subset \CR$, we then define the quantity
\begin{equ} [eq:calH]
\CH_P(x;\J)
= \Bigl(\prod_{e\in P} \hat\J_{e}\Bigr)
\Bigl(\prod_{f\in \CE(M \setminus P')}\hat\J_{f}\Bigr)\;,
\end{equ}
where $x\in(\R^3)^M$. 
Note that in the first product above, every $e$ is a pair of opposite charges
$\{e_+,e_-\}$, so all the factors $\J(e_+ - e_-)$  are powered by $-1$; in the second product, the factors $\J(x_i - x_j)$ for $f=\{i,j\}$ could appear in either the numerator or the denominator, depending on whether $i,j$ having the same sign or not.
We also write
\begin{equ} 
\CH(x;\J)
= \sum_{P \subset \CR} (-1)^{|P|} \CH_P(x;\J)\;,
\end{equ}

With all these notations at hand,  the first product over $e\in\CR$
in the expression \eref{e:over-4N} is then written as
\begin{equs} 
& \sum_{P \subset \CR} (-1)^{|P|}  
\Big( \prod_{e\in P}
\E\big[\Psi^k_\eps(e_+) \bar\Psi^k_\eps(e_-)\big] \Big)  \,
\E \Big[ \prod_{f\in\CE(M \setminus P')}
	\!\!\!\! \Psi^k_\eps(f_+) \bar\Psi^k_\eps(f_-)
		 \Big]  \\
& = e^{ - \beta^2 m \left(k^2 - 1 \right)Q_\eps(0)}     \,
	\CH \big(x;    \J_\eps^{k^2} \big)		  \,.
\end{equs}
Therefore, for $m=2N$
and the function $\J_\eps$ defined in~\eref{e:defJeps}, we have
the identity
\begin{equ} [e:mth-moment]
\E |\scal{\phi_0^\lambda,  \Psi^{k\bar k}_\eps}|^{m}
=  e^{- \beta^2 m \left(k^2 - 1 \right)Q_\eps(0)}
\int_{(\R^3)^M}       
	\CH \big(x;   \J_\eps^{k^2} \big) 
\Big(
\prod_{e \in \CR}
	\phi_0^\lambda(e_\downarrow) K(e)
\Big)
\,dx\;.
\end{equ}

Similarly to before, we aim to obtain suitable bounds on the function $\CH$ that are uniform
over $\eps \in (0,1]$ and such that we can bound the small-$\lambda$
behaviour of this integral.
The most important case would be $k=1$ for which $\CH = \CH (x;\J_\eps)$,
since only $\Psi^{1\bar 1}_\eps$ converges to a nontrivial limit.

Given $\J$ as above, we define for any two ``dipoles'' $e\neq f \in \CR$ the quantity  
\begin{equ}[e:Delta-ef]
	\Delta_{e}^{f} (\J)
	=\frac{\J_{e^{+}f^{+}}\J_{e^{-}f^{-}}}{\J_{e^{+}f^{-}}\J_{e^{-}f^{+}}}-1\;.
\end{equ}
This quantity plays an important role in this section because it is small if
either $e^+ \approx e^-$ or $f^+ \approx f^-$. 
As a consequence, being able
to extract sufficiently many factors of this type from $\CH(x;\J_\eps)$ will
enable us to compensate enough of the divergence of the kernel $K$ in the 
expression for $\E |\scal{\phi_0^\lambda,  \Psi^{k\bar k}_\eps}|^{m}$.

We also define for $A\subset\CR$ and $e\notin A$ the quantity
\begin{equ}
	\Delta_{e}^{A} (\J) =\prod_{f\in A}\Delta_e^f (\J) \;.
\end{equ}
Finally, 
suppose that we are given a subset $A\subset\CR$
as well as a map $\mathcal{B} \colon A \to \power \mathcal R\setminus\{\emptyset\}$
\footnote{Given a set $E$, we write $\power E$ for its power set.}
associating to each pair $e \in A$ a non-empty subset $\CB_e$ of $\CR$.
Then, provided that $A\neq \emptyset$, we define the quantity
\begin{equ}
	\Delta_{A}^{\mathcal{B}} (\J) = \prod_{e\in A}\Delta_e^{\CB_e} (\J) \;.
\end{equ}
In the special case $A=\emptyset$ so that the above product is empty, 
we use the usual convention that $\Delta_{A}^{\mathcal{B}} (\J) =1$.
This definition also makes sense if $\CB$ is defined on a larger set containing $A$.
We also write
\begin{equ}
	U_{A}^{\mathcal{B}}=A\cup \bigcup_{e\in A} \CB_e\;.
\end{equ}
The following identity, which can easily be proved by induction, will be used:
\begin{equ}
\label{eq:identity}
\Bigl(\prod_{i=1}^n a_i\Bigr) -1
= \sum_{\emptyset \neq P \subseteq \{1,\ldots,n\}} 
\prod_{i\in P} (a_i -1)\;.
\end{equ}

In order to rewrite $\CH$ in a way that makes some of the cancellations more explicit, we
will make use of the following notations. Assume that we are given an ordering of $\CR$ 
so that $\CR = \{e_{1},\ldots,e_{m}\}$, as well as a subset $A \subset \CR$. 
We set $\CR_0 = \emptyset$ and $\CR_\ell \eqdef \{e_1,\ldots,e_\ell\}$ for $0 < \ell \leq m$, as well as $A_\ell=A\cap \CR_\ell$.
For any $\ell \in \{0,\ldots,m\}$ and $A \subset \CR_\ell$, 
we then write $\CM_\ell(A)$ for the set of all maps
$\CB \colon A \to \power \mathcal R\setminus\{\emptyset\}$ which furthermore satisfy the following
two properties:
\begin{claim}
\item If $e_k \in A$, then $\CB_{e_k} \subset \CR \setminus A_k$.
\item For every $k \le \ell$, one has $e_k \in A$ if and only if $e_k \not\in U_{A_{k-1}}^\CB$.
\end{claim}
We have the following very useful recursive characterisation of the
sets $\CM_\ell(A)$:

\begin{lemma}\label{lem:charMl}
Let $\ell \ge 1$ and $A \subset \CR_\ell$. Then $\CB \in \CM_\ell(A)$
if and only if $\CB$ restricted to $A_{\ell-1}$ belongs to $\CM_{\ell-1}(A_{\ell-1})$ 
and exactly one of the following two statements holds:
\begin{claim}
\item One has $e_\ell \in A$, $e_\ell \not\in U_{A_{\ell-1}}^\CB$, and $\CB_{e_\ell}\subset \CR \setminus A$.
\item One has $e_\ell \not\in A$ and $e_\ell \in U_{A_{\ell-1}}^\CB$.
\end{claim}
\end{lemma}

\begin{proof}
This follows immediately from the definitions.
\end{proof}

\begin{remark}
For $\ell > 0$, the second of these properties cannot be satisfied 
unless $e_1 \in A$. In particular,
this shows that $\CM_\ell(\emptyset) = \emptyset$.
For $\ell = 0$ however, both properties are empty so that $\CM_0(\emptyset)$
consists of one element, which is the trivial map.
\end{remark}

\begin{proposition}
\label{prop:cancellations}
Fix an arbitrary ordering of $\CR$ as above.
Then, for any given function $\J$, 
and for every $0\leq \ell \leq m$, one has the identity
\begin{equ}[e:bound-atscale]
 \CH(x;\J)
 =
 \sum_{A\subseteq \CR_\ell}
  \sum_{\mathcal{B}\in \CM_\ell(A)}
  \Delta_{A}^{\mathcal{B}}(\J) \,\CH(A,\mathcal{B};x;\J)
\end{equ}
where we made use of the notation
\begin{equ} [e:HAB]
\CH(A,\mathcal{B};x;\J)
\eqdef 
\sum_{P\subseteq\mathcal{R}\setminus U_{A}^{\mathcal{B}}}
	(-1)^{|P|}   \CH_{A\cup P}(x;\J)\;,
\end{equ}
for any set $A\subseteq \mathcal R$. Here, $\CH_{A\cup P}(x;\J)$
is as in \eqref{eq:calH}.
\end{proposition}

\begin{remark}
The factor $\Delta_{A}^{\mathcal{B}}$ appearing in this
expression does of course also depend on the specific configuration 
$x$ of the charges. We drop this 
dependence in the notations in order not to overburden them.
Figure~\ref{fig:picture} provides a pictorial representation of an example of term
$\CH(A,\CB)$ appearing in the statement,
the reader is encouraged to read the proof with this example in mind.
\end{remark}

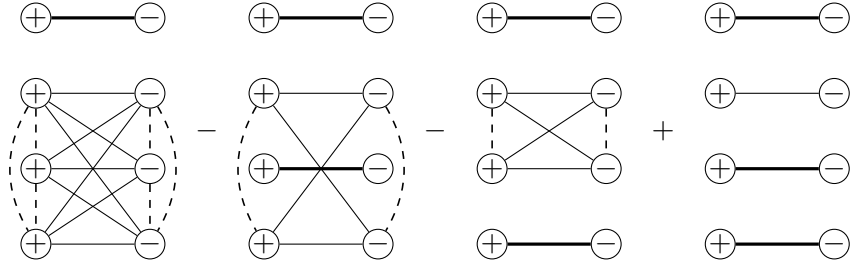
\begin{figure}
\centering
\begin{tikzpicture}
\node at (0,2) [charge] (1) {$+$};  \node at (1.5,2) [charge] (2) {$-$};
\node at (0,1) [charge] (3) {$+$};  \node at (1.5,1) [charge] (4) {$-$};
\node at (0,0) [charge] (5) {$+$};  \node at (1.5,0) [charge] (6) {$-$};
\node at (0,-1) [charge] (7) {$+$};  \node at (1.5,-1) [charge] (8) {$-$};
\draw [kerP] (1) to (2);  
\draw (3) to (4); 
\draw (3) to (6); \draw (3) to (8);
\draw (4) to (5); \draw (4) to (7);
\draw  (5) to (6); 
\draw (5) to (8);
\draw  (6) to (7);
\draw  (7) to (8);
\draw [kerAlg2] (3) to (5);   \draw [kerAlg2] (5) to (7);
\draw [kerAlg2] (3) to [bend right=30] (7); 
\draw [kerAlg2] (4) to (6);  \draw [kerAlg2] (6) to (8);
 \draw [kerAlg2] (4) to [bend left=30] (8); 
\node at (2.25,0.5) {$-$};
\node at (3,2) [charge] (1) {$+$};  \node at (4.5,2) [charge] (2) {$-$};
\node at (3,1) [charge] (3) {$+$};  \node at (4.5,1) [charge] (4) {$-$};
\node at (3,0) [charge] (5) {$+$};  \node at (4.5,0) [charge] (6) {$-$};
\node at (3,-1) [charge] (7) {$+$};  \node at (4.5,-1) [charge] (8) {$-$};
\draw [kerP] (1) to (2);  
\draw (3) to (4); 
\draw  (3) to (8);
\draw (4) to (7);
\draw [kerP] (5) to (6); 
\draw (7) to (8);
\draw [kerAlg2] (3) to [bend right=30] (7); 
 \draw [kerAlg2] (4) to [bend left=30] (8); 
\node at (5.25,0.5) {$-$};
\node at (6,2) [charge] (1) {$+$};  \node at (7.5,2) [charge] (2) {$-$};
\node at (6,1) [charge] (3) {$+$};  \node at (7.5,1) [charge] (4) {$-$};
\node at (6,0) [charge] (5) {$+$};  \node at (7.5,0) [charge] (6) {$-$};
\node at (6,-1) [charge] (7) {$+$};  \node at (7.5,-1) [charge] (8) {$-$};
\draw [kerP] (1) to (2);  
\draw (3) to (4); 
\draw (3) to (6); 
\draw (4) to (5); 
\draw  (5) to (6); 
\draw [kerP]  (7) to (8);
\draw [kerAlg2] (3) to (5);   
\draw [kerAlg2] (4) to (6);  
\node at (8.25,0.5) {$+$};
\node at (9,2) [charge] (1) {$+$};  \node at (10.5,2) [charge] (2) {$-$};
\node at (9,1) [charge] (3) {$+$};  \node at (10.5,1) [charge] (4) {$-$};
\node at (9,0) [charge] (5) {$+$};  \node at (10.5,0) [charge] (6) {$-$};
\node at (9,-1) [charge] (7) {$+$};  \node at (10.5,-1) [charge] (8) {$-$};
\draw [kerP] (1) to (2);  
\draw  (3) to (4); 
\draw [kerP] (5) to (6); 
\draw [kerP]  (7) to (8);
\end{tikzpicture}
\caption{Pictorial representation of $\CH(A,\CB)$ in the case
$A=\{e_1\} \subseteq \CR_1$ and $\CB_{e_1}=\{e_2\}$ with $m=4$.
The $4$ horizontal pairs are, from top to bottom, $e_1,e_2,e_3$ and $e_4$.
By definition, $U_A^\CB = \{e_1,e_2\}$, and therefore $P$ in \eref{e:HAB} runs over all subsets
of $\{e_3,e_4\}$: the $4$ terms above correspond to $\CH_{A\cup P}$  with $P$ being $\emptyset$,
$\{e_3\}$, $\{e_4\}$, and $\{e_3,e_4\}$, respectively.
Each $\CH_{A\cup P}$ is a product of $\J$'s (drawn in dashed lines) 
or $\J^-$'s (drawn in solid lines).
Pairs in $A\cup P$ are a bit thicker since they stand for the $\J^-$'s 
in the first product of \eref{eq:calH}.
}\label{fig:picture}
\end{figure}
\begin{proof}
In this proof, we hide the argument $\J$ in all the $\CH$ functions  for simplicity of notations.
The proof of the result now goes by induction over $\ell$. 
For $\ell=0$, $A=\emptyset$, $\CM_0(\emptyset)$ consists of one element which is the trivial 
map for which, by convention, $\Delta_{A}^{\mathcal{B}}=1$
so that the statement is precisely the definition of $\CH$. 
Assuming
that the statement holds for $\ell-1$, we now show that it also
holds for $\ell$.
We rewrite the induction hypothesis as
\begin{equ}
[e:oneplusone]
  \CH
  =   \sum_{A\subseteq \CR_{\ell-1}} \sum_{\mathcal{B}\in\CM_{\ell-1}(A)}
   \bigl(\one_{e_\ell \in U_{A}^{\mathcal{B}}} + \one_{e_\ell\notin U_{A}^{\mathcal{B}}}\bigr) 
  \Delta_{A}^{\mathcal{B}}\,\mathcal{H}(A,\mathcal{B})\;,
\end{equ}
and we consider the resulting two terms separately.

Consider first the case $e_\ell\notin U_{A}^{\mathcal{B}}$.
Writing $\bar{A}=A\cup\{e_\ell\}$, one can then rewrite $\CH(A,\CB)$ as
\begin{equs}
	\mathcal{H}(A,\mathcal{B}) 
		& =\sum_{P\subseteq\mathcal{R}\setminus(U_{A}^{\mathcal{B}}\cup\{e_\ell\})}(-1)^{|P|}
		\Big(\mathcal H_{A\cup P}
		-\mathcal H_{\bar{A}\cup P} \Big)\\
 	& =\sum_{P\subseteq\mathcal{R}\setminus(U_{A}^{\mathcal{B}}\cup\{e_\ell\})}(-1)^{|P|}
	 	\mathcal H_{\bar{A}\cup P}
 		\bigg(\prod_{f\in\mathcal{R}\setminus(\bar{A}\cup P)}
		\frac{\J_{f^{+}e_{\ell}^{+}}\J_{f^{-}e_{\ell}^{-}}}
		{\J_{f^{+}e_{\ell}^{-}}\J_{f^{-}e_{\ell}^{+}}}-1\bigg) \;.
\end{equs}
Using the identity \eref{eq:identity}, this can be rewritten as
\begin{equs}
	\mathcal{H}(A,\mathcal{B}) 
 		& =\sum_{P\subseteq\mathcal{R}\setminus(U_{A}^{\mathcal{B}}\cup\{e_\ell\})}(-1)^{|P|}
		 \mathcal H_{\bar{A}\cup P}
 		\sum_{\emptyset\neq Q\subseteq\mathcal{R}\setminus(\bar{A}\cup P)}\Delta_{e_\ell}^{Q}\\
 	& =\sum_{\emptyset\neq Q\subseteq\mathcal{R}\setminus\bar{A}}\Delta_{e_\ell}^{Q}
		\sum_{P\subseteq\mathcal{R}\setminus(U_{A}^{\mathcal{B}}\cup Q\cup\{e_\ell\})}(-1)^{|P|}
 		\mathcal H_{\bar{A}\cup P}\;.
\end{equs}
Given $\mathcal B \in \CM_{\ell-1}(A)$ and a non-empty set $Q \subset \CR\setminus \bar A$ as above,
we then define a map  $\bar{\mathcal B} \in \CM_\ell(\bar A)$
by $\bar{\mathcal B} (e)=\mathcal B (e)$ for all $e\in A$
and $\bar{\mathcal B} (e_\ell)=Q$. 
As a consequence of Lemma~\ref{lem:charMl}, we see that 
all maps in $\CM_\ell(\bar A)$ arise in this way.
One then has 
$U_{A}^{\mathcal{B}}\cup Q\cup\{e_l\}=U_{\bar{A}}^{\bar{\mathcal{B}}}$
and thus
\begin{equ}
	\sum_{P\subseteq\mathcal{R}\setminus(U_{A}^{\mathcal{B}}\cup Q\cup\{e_\ell\})}(-1)^{|P|}
 	\mathcal H_{\bar{A}\cup P}
	=\mathcal{H}(\bar A,\bar{\mathcal{B}})\;.
\end{equ}
Making use of the identity
$\Delta_{A}^{\mathcal{B}} \Delta_{e_\ell}^Q= \Delta_{\bar A}^{\bar{\mathcal{B}}}$,
we conclude that
\begin{equ}
  \sum_{A\subseteq \CR_{\ell-1}} 
   \sum_{\mathcal{B}\in \CM_{\ell-1}(A)}
  \one_{e_\ell\notin U_{A}^{\mathcal{B}}}
  \Delta_{A}^{\mathcal{B}}\mathcal{H}(A,\mathcal{B}) 
=  \sum_{\bar A\subseteq \CR_\ell \atop e_\ell\in \bar A} 
 \sum_{\bar{\mathcal{B}}\in \CM_\ell(\bar A)}
  \Delta_{\bar A}^{\bar{\mathcal{B}}}
  \mathcal{H}(\bar A,\bar{\mathcal{B}})\;.
\end{equ}

Concerning the term in \eref{e:oneplusone} with $e_\ell \in U_A^{\mathcal B}$, we use the fact
that, again by Lemma~\ref{lem:charMl}, if $e_\ell \not\in A$ but $e_\ell \in U_A^{\mathcal B}$, then $\CM_{\ell-1}(A) = \CM_{\ell}(A)$, so that
\begin{equ}
  \sum_{A\subseteq \CR_{\ell-1}} 
 \sum_{\mathcal{B}\in \CM_{\ell-1}(A)}
  \one_{e_\ell \in U_{A}^{\mathcal{B}}}
  \Delta_{A}^{\mathcal{B}}\mathcal{H}(A,\mathcal{B})
  = \sum_{\bar A\subseteq \CR_\ell \atop e_\ell\notin \bar A} 
  \sum_{\bar{\mathcal{B}}\in \CM_\ell(\bar A)}
  \Delta_{\bar A}^{\bar{\mathcal{B}}}
  \mathcal{H}(\bar A,\bar{\mathcal{B}})\;.
\end{equ}
Adding both identities concludes the proof 
of \eref{e:bound-atscale}.
\end{proof}


The most important point of Proposition~\ref{prop:cancellations} is that
at an arbitrary step $\ell$, all the renormalised pairs in $\CR_\ell$
are contained in the set $U_{A}^{\mathcal{B}}$, 
as imposed by the definition of $\CM_\ell(A)$.
The quantity $\Delta_{A}^{\mathcal{B}}$
then generates factors of {\it positive} homogeneities for the pairs 
in $U_{A}^{\mathcal{B}}$ (see the remarks after \eref{e:Delta-ef}).
In order to make this statement more precise, it is convenient to introduce the quantity
\begin{equ}[e:defAij]
\mathcal{A}_{ef}  \eqdef
\Vert x_{e_+}-x_{e_-}\Vert_{\s}\Vert x_{f_+}-x_{f_-}\Vert_{\s}\Vert x_{e_+}-x_{f_-}\Vert_{\s}^{-1}\Vert x_{f_+}-x_{e_-}\Vert_{\s}^{-1}\;,
\end{equ}
for any two dipoles $e$ and $f$.

\begin{lemma} \label{lem:smaller-product}
Suppose that $e^+,e^-,f^+,f^- \in\R^3$ are four distinct points,
and that
\begin{equ}[e:conditionAsmall]
\Vert e^+ - e^-\Vert_{\s} \wedge
\Vert f^+ - f^-\Vert_{\s}
\le
\Vert e^+ - f^-\Vert_{\s} \wedge
\Vert e^- - f^+ \Vert_{\s}\;,
\end{equ}
Then, one has the inequality
\begin{equ} [eq:Aboundedby1]
\Vert e^+ - e^-\Vert_{\s}
\Vert f^+ - f^-\Vert_{\s} 
\lesssim
\Vert e^+ - f^-\Vert_{\s}
\Vert e^- - f^+ \Vert_{\s}\;.
\end{equ}
In terms of the quantity $\mathcal{A}_{ef}$ defined in \eref{e:defAij},
we can also write this more succinctly as $\mathcal{A}_{ef} \lesssim 1$.
\end{lemma}

\begin{proof}
Since the statement is symmetric under $e \leftrightarrow f$, we can assume without loss of generality that one has
\begin{equ}
\Vert e^+ - e^-\Vert_{\s}
\lesssim
\Vert e^+ - f^-\Vert_{\s} \wedge
\Vert e^- - f^+ \Vert_{\s}\;.
\end{equ}
Then, by the triangular
inequality,
\begin{equs}
  \Vert f^{+}- f^{-}\Vert_{\s}
  &  \lesssim
  \Vert e^+ - e^-\Vert_{\s}+
\Vert e^+ - f^-\Vert_{\s}+
\Vert e^- - f^+ \Vert_{\s} \\
 & \lesssim
\Vert e^+ - f^-\Vert_{\s} \vee
\Vert e^- - f^+ \Vert_{\s} \;.
\end{equs}
The bound \eref{eq:Aboundedby1} follows by combining these two 
inequalities.
\end{proof}

The following lemma will be used in the proof of Proposition~\ref{prop:bound--H1} below.

\begin{lemma} \label{lem:max}
The final pairing $\mathcal{S}$ selected by the procedure in Section \ref{sec:linear} maximizes
(up to a multiplicative constant depending only on $m$ but not on the specific configuration of points $x$) the quantity
$\Pi_{\mathcal{S}} \eqdef \prod_{\{i,j\}\in\mathcal{S}}\J_{ij}^{-}$.
\end{lemma}

\begin{proof}
Let $\mathcal{S}$ be the pairing selected by the procedure in Section \ref{sec:linear} 
and let $\bar{\mathcal{S}} \neq \mathcal{S}$ be a different pairing.
Without loss of generality, we assume that in the procedure to construct
$\mathcal{A}_{n}$ and $\mathcal{S}_{n}$ in Section \ref{sec:linear}, at each step
$n$ only two sets merge together. Let $n$ be the smallest number
such that there exist $A\in\mathcal{A}_{n}$ with $e\in\mathcal{S}_{n}(A)$
but $e\notin\bar{\mathcal{S}}$. Then, there exist a set $B\in\mathcal{A}_{n-1}$
containing $e_{+}$ and $\bar{B}\in\mathcal{A}_{n-1}$ containing
$e_{-}$ and $B\neq\bar{B}$. 

Suppose that $\{e_{+},f_{-}\},\{e_{-},f_{+}\}\in\bar{\mathcal{S}}$.
One has $f_{-}\notin B$ (otherwise there would be already a pair
with both charges in $B$ which belongs to $\mathcal{S}$ but not
$\bar{\mathcal{S}}$, thus contradicting the minimality assumption on $n$), 
and $f_{+}\notin\bar{B}$, so
\begin{equ}
\Vert x_{e_{+}}-x_{f_{-}}\Vert_{\s}\gtrsim2^{n}\;,
\qquad\Vert x_{e_{-}}-x_{f_{+}}\Vert_{\s}\gtrsim2^{n}\;.
\end{equ}
By $\Vert x_{e_{+}}-x_{e_{-}}\Vert_{\s}\lesssim2^{n}$ 
and Lemma \ref{lem:smaller-product}, one has
\begin{equ}
  \Vert x_{e_{+}}-x_{e_{-}}\Vert_{\s}\Vert x_{f_{+}}-x_{f_{-}}\Vert_{\s}
  \lesssim\Vert x_{e_{+}}-x_{f-}\Vert_{\s}\Vert x_{f_{+}}-x_{e_{-}}\Vert_{\s}\;.
\end{equ}
If we define $\tilde{\mathcal{S}}$ by keeping all the parings in
$\bar{\mathcal{S}}$ except that $\{e_{+},f_{-}\},\{e_{-},f_{+}\}$
are replaced by $\{e_{+},e_{-}\},\{f_{+},f_{-}\}$, we have $\Pi_{\bar{\mathcal{S}}}\lesssim\Pi_{\tilde{\mathcal{S}}}$.
Note that $e\in\tilde{\mathcal{S}}$. We can then iterate the above
procedure to consequently increase $\Pi$ until we obtain the pairing
$\mathcal{S}$. 
\end{proof}


\begin{lemma}
\label{lem:Delta-by-A}
Assume that we are given a function $\CQ(z)\sim-\frac{1}{2\pi}\log\left\Vert z\right\Vert _{\s}$, that $\bar \rho$ is
any continuous function supported on the unit ball around the origin integrating to $1$,
and that $\CQ_\eps$ is as in \eref{e:CQeps}.
Given $\alpha \geq 2$, let $\J(z) = e^{-2\pi\alpha \CQ(z)}$ and $\J_\eps(z) = e^{-2\pi\alpha \CQ_\eps(z)}$.
Assume furthermore that $\CQ$ is of class $\CC^2$ and there 
exists a constant $C$ such that $|\nabla^k \CQ(z)| \le C / \|z\|_\s^{|k|_\s}$ for $|k|\leq 2$.
Let  $e^+,e^-,f^+,f^- \in\R^3$ be four distinct points such that
\begin{equ}[e:condef]
\Vert e^+ - e^-\Vert_{\s} \wedge
\Vert f^+ - f^-\Vert_{\s}
\le
\Vert e^+ - f^-\Vert_{\s} \wedge
\Vert e^- - f^+ \Vert_{\s}\;,
\end{equ}
and write $e=\{e^+,e^-\},f=\{f^+,f^-\}$.
One then has the bound
\begin{equ} [e:Delta-by-A]
\bigg|\frac{\J_\eps(e^+-f^+)\J_\eps(e^--f^-)}{\J_\eps(e^+-f^-)\J_\eps(e^--f^+)} -1 \bigg|
\lesssim 
\mathcal{A}_{ef}\;,
\end{equ}
uniformly for all $\eps\geq0$, where $\mathcal{A}_{ef}$ is as in \eqref{e:defAij}.

In particular, for $k\geq 1$, the function $\Delta_{e}^f (\J_\eps^{k^2})$ 
with $\CJ_\eps$ defined in \eref{e:defJeps} satisfies \eref{e:Delta-by-A}
on the set \eqref{e:condef} by choosing $\alpha = k^2 \beta^2/(2\pi)$.
\end{lemma}

\begin{proof}
As a consequence of the symmetries $e \leftrightarrow f$ and $(e^+,f^+) \leftrightarrow (e^-,f^-)$
we can assume without loss of generality that
\minilab{e:WLG-smaller}
\begin{equs}
\Vert e^+ - e^-\Vert_{\s} & \le \Vert f^+ - f^- \Vert_{\s} \;, \label{e:WLG-smaller1}
\\
\Vert f^+ - e^-\Vert_{\s} & \le \Vert e^+ - f^- \Vert_{\s}  \label{e:WLG-smaller2}
\;.
\end{equs}

First of all, we consider the ``easier" case, that is
\begin{equ}
	\Vert e^+ - f^+\Vert_{\s}  \le 5\Vert e^+ - e^- \Vert_{\s} \;.
\end{equ}
In this case, by the triangular inequality one has 
$\Vert f^+ - e^-\Vert_{\s}  \le 6\Vert e^+ - e^- \Vert_{\s}$
and, using the triangle inequality together with \eref{e:WLG-smaller1}, one also has
$\Vert e^+ - f^-\Vert_{\s}  \le 6\Vert f^+ - f^- \Vert_{\s}$,
so that $\CA_{ef} \ge 1/36$.
Furthermore, by the triangle inequality, \eref{e:condef},
and \eref{e:WLG-smaller1}, one has
$\Vert e^+ - f^+\Vert_{\s}  \lesssim \Vert f^+ - e^- \Vert_{\s}$
and $\Vert e^- - f^-\Vert_{\s}  \lesssim \Vert e^+ - f^- \Vert_{\s}$.
Therefore the left hand side of \eref{e:Delta-by-A} is bounded by
some constant independent of $\eps$, and \eref{e:Delta-by-A} follows.

If
$\Vert e^- - f^-\Vert_{\s}  \le 5\Vert e^+ - e^- \Vert_{\s} $
then the bound \eref{e:Delta-by-A} holds in a similar way.
Therefore it remains to consider the situation where
\begin{equ}
	\Vert e^- - f^-\Vert_{\s}  \wedge
	\Vert e^+ - f^+\Vert_{\s}
	 \geq 5\Vert e^+ - e^- \Vert_{\s}  \;,
\end{equ}
which in particular implies that
\begin{equ}[e:additional-assum]
4\Vert e^+ - e^- \Vert_{\s} \le	\Vert e^\pm - f^\pm\Vert_{\s} \;,
\end{equ}
for any of the four choices of signs that can appear on the right hand side.
Define a function $F$ depending on $e^+,f^+$ and $\eps$ by
\begin{equ}
F(z_{e},z_{f}) 
\eqdef \J_\eps(e^+ - f^+) \, \J_\eps(z_{e}-z_{f}) 
 -\J_\eps(e^+ - z_f) \, \J_\eps(f^+ - z_e) \;.
\end{equ}
Since $\CJ_\eps$ is assumed to be symmetric, one has 
$F(z_{e},z_{f})= 0$ whenever $z_{e}=e^+$ or $z_{f}=f^+$. In
particular $\partial_{z_{e}}^{k}F(z_{e},f^+)=\partial_{z_{f}}^{k}F(e^+,z_{f})=0$
for all $k\geq0$. We will show that 
under the assumptions of the lemma, one has the bound
\begin{equ}[e:quadrupole]
|F(e^-,f^-)|\lesssim
\J_\eps(e^+ - f^-) \, \J_\eps(f^+ - e^-) \, \mathcal{A}_{ef}\;,
\end{equ}
which then immediately concludes the proof of the lemma.

To show \eref{e:quadrupole}, let 
$\gamma_{e}:[0,1]\to \R^3 $ 
be the  piecewise linear path from $e^+$ to $e^-$ which is made up from three pieces,
each of them parallel to one of the coordinate axes.
Then, since $F(e^+,f^-)=0$, one has
\begin{equ}[e:first-line]
| F (e^-,f^-) |
 \lesssim 
 \sum_{i=0}^2    |\gamma_e^{(i)}| \sup_{z_e \in\gamma_e} 
 	|\nabla_{z_{e}}^{(i)} F(z_e,f^-)| \;,
\end{equ}
where $\nabla_z^{(i)}$ denotes the derivative with respect to the $i$th component
of the variable $z$ and $|\gamma_e^{(i)}|$ denotes the total (Euclidean!) 
length of the pieces of
the path $\gamma_e$ that are parallel to the $i$th coordinate axis.
Note that one has $|\gamma_e^{(i)}| \le \|e^+ - e^-\|_\s^{\s_i}$.

Similarly, let $\gamma_{f}: [0,1]\to \R^3 $ 
be a piecewise linear path from $f^+$ to $f^-$, again with each piece parallel to
one of the coordinate axes, but this time with possibly more than three pieces.
We claim that one can furthermore choose $\gamma_{f}$ in such a way that each of its 
pieces has parabolic length at most $\|f^+ - f^-\|_\s$ and such that the bounds
\begin{equ}[e:choice-of-curve]
\|f^+ - e^-\|_\s
\lesssim
\| z_e - z_f\|_\s
\lesssim
\|e^+ - f^-\|_\s \;,
\end{equ}
hold uniformly over $ z_e \in \gamma_e$ and $ z_f \in \gamma_f$. 
Here, the upper bound is a simple consequence
of the triangle inequality and the fact that \eqref{e:WLG-smaller} and 
\eqref{e:additional-assum} imply that 
$\|e^+  - f^+\|_\s + \|e^+  - e^-\|_\s + \|f^-  - f^+\|_\s \lesssim \|e^+  - f^-\|_\s$.

In order to enforce the lower bound, more care has to be taken. 
Define an ``exclusion zone'' $Z \subset \R^3$ as the convex hull of
$\{z\,:\, \|z-\bar z\|_\s \ge \|f^+ - e^-\|_\s/4\;\forall \bar z \in \gamma_e\}$.
It then follows from \eqref{e:additional-assum} that both $f^\pm$ are located 
outside of $Z$, so it suffices to choose $\gamma_f$ in such a way that it does not 
intersect $Z$. A typical situation with $Z$ draw in light grey is depicted 
in Figure~\ref{fig:paths}.
\begin{figure}
\centering
\begin{tikzpicture} [scale=0.8]
\fill [black!10] (-0.5,-1.6) rectangle (2,1);

\node [dot] at (0.5,0) (e+) [label=above:{$e^+$}] {};  
\node [dot] at (1,-0.6) (e-)[label=right:{$e^-$}] {};
 \node [dot] at (-2,-2) (f+) [label=left:{$f^+$}] {};
 \node [dot] at (5,2) (f-) [label=right:{$f^-$}] {};
 
\draw (e+) -- (1,0) -- (e-);  
\draw (f+) -- (-2,2) -- (f-);
\node at (1.3,0) {$\gamma_e$};
\node at (-2.4,1.3) {$\gamma_f$};
  

\end{tikzpicture}
\caption[Construction of the $\gamma_i$.]{Construction of the paths $\gamma_e$ and $\gamma_f$. The ``exclusion zone'' $Z$ is shaded in light grey.}
\label{fig:paths}
\end{figure}
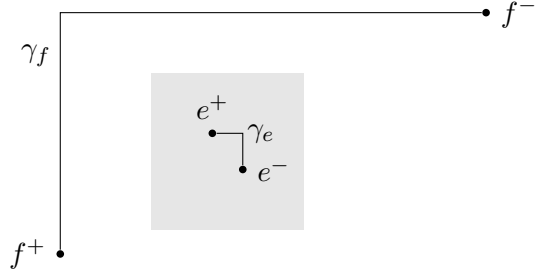

Since  $F(z_e,f^+) = 0$ for every $z_e$ and thus
$\nabla_{z_e}  F(z_e,f^+)=0$, we
can apply the gradient theorem to 
$|\nabla_{z_{e}}^{(i)} F(z_e,f^-)|$, yielding
\begin{equ}[e:second-line]
| F (e^-,f^-) |
 \lesssim 
 \sum_{i,j=0}^2  
 |\gamma_e^{(i)}| |\gamma_{f}^{(j)}| 
 \sup_{z_e \in\gamma_e} 
 \sup_{z_f\in\gamma_f}
 	|\nabla_{z_{e}}^{(i)} 
	\nabla_{z_f}^{(j)}
	F(z_e,z_f)|  \;.
\end{equ}

Write now $\|z\|_{\s,\eps} \eqdef \|z\|_\s + \eps$ so that $\CJ_\eps(z)$ is bounded from
above and below by some fixed multiple of $\|z\|_{\s,\eps}^{\alpha}$ and
note that $|\nabla^k\CQ_\eps(z)| \lesssim \|z\|_{\s,\eps}^{-|k|_\s}$ as a consequence of
our assumptions and of  \cite[Lemma~10.17]{Regularity}.
As a consequence, one has the bound 
$|\nabla^k \CJ_\eps (z)| \lesssim \|z\|_{\s,\eps}^{\alpha -|k|_\s}$
so that one has 
\begin{equs}  [e:bound-DDF]
 |\nabla_{z_e}^{(i)}  \nabla_{z_f}^{(j)}  F (z_e, z_f)| 
& \lesssim 
\|e^+ - f^+\|_{\s,\eps}^{\alpha}  \cdot
\| z_e - z_f \|_{\s,\eps}^{\, \alpha-\s_i-\s_j} \\
& \qquad +
\|e^+ - z_f\|_{\s,\eps}^{\, \alpha-\s_j} \cdot
\|f^+ - z_e \|_{\s,\eps}^{\, \alpha-\s_i} \;.
\end{equs}
Combining the  triangle inequality with \eref{e:condef}, \eref{e:WLG-smaller1}, 
and \eref{e:choice-of-curve}, the factors appearing in the right hand side 
of \eref{e:bound-DDF} are bounded
as follows (here we use the fact that $\alpha \geq 2$ so that 
$\alpha - \s_j$ is guaranteed to be positive):
\begin{equs}
	\|e^+ - z_f\|_{\s,\eps}^{\alpha-\s_j} 
	&\lesssim
	\|e^+ - f^-\|_{\s,\eps}^{\alpha-\s_j}\;, \\
	\|f^+ - z_e \|_{\s,\eps}^{\alpha-\s_i} 
	&\lesssim
	\|f^+ - e^- \|_{\s,\eps}^{\alpha-\s_i} \;,\\
	\| z_e - z_f \|_{\s,\eps}^{\alpha- s_i - \s_j} 
	&\lesssim
	\|e^+ - f^- \|_{\s,\eps}^{\alpha-\s_j }  
	\|f^+ - e^-\|_{\s,\eps}^{-\s_i } \;,\\
	\|e^+ - f^+\|_{\s,\eps}^{\alpha}
	&\lesssim
	\|f^+ - e^-\|_{\s,\eps}^{\alpha -\s_i} \|f^+ - e^-\|_{\s,\eps}^{\s_i}  \;.
\end{equs}
Inserting these bounds into \eqref{e:bound-DDF}, we conclude that
\begin{equ}
 |\nabla_{z_e}^{(i)}  \nabla_{z_f}^{(j)}  F (z_e, z_f)| 
\lesssim \J_\eps(e^+ - f^-) \, \J_\eps(f^+ - e^-) \|e^+ - f^- \|_{\s,\eps}^{-\s_j}\|f^+ - e^- \|_{\s,\eps}^{-\s_i}\;,
\end{equ}
uniformly for $z_{e,f} \in \gamma_{e,f}$.
Finally, we observe that
\begin{equ}
|w_1^{(i)}|  \lesssim \|e^+ - e^-\|_\s \, \|f^+ - e^-\|_\s^{\s_i -1} \;,
\qquad
|w_2^{(j)}|  \lesssim \|f^+ - f^-\|_\s \, \|e^+ - f^- \|_\s^{\s_j -1} \;,
\end{equ}
so that the claim \eref{e:quadrupole} follows from \eqref{e:second-line}.
\end{proof}

\begin{remark} \label{rem:not-sharp}
The analysis we are following here is not as sharp as possible. 
In the above proof, we essentially performed a first order Taylor expansion
of $\Delta_e^f$ (viewed as a function of $e^-$) around $e^+$,
which allowed us to gain a factor $\|e^+ - e^-\|_\s$. 
This factor, when multiplied by $K(e) \J_e^-$, is integrable 
as long as $\beta^2 < 6\pi$. 
However, the linear term $e^- - e^+$ is an odd function, while all other functions
($K,\J$ etc.) are even in their spatial coordinates, so that 
the integration over $e^-$ in a neighborhood of $e^+$
essentially vanishes. As a consequence, we believe that it should be possible to 
gain a factor $\|e^+ - e^-\|_\s^2$, thus allowing to control the second-order objects 
for all $\beta^2 < 8\pi$.  This would however require us to change our strategy, which
is to obtain bounds on the absolute value of $\CH$ that are sufficiently sharp to guarantee
that \eqref{e:mth-moment} has the correct order of magnitude.
\end{remark}

Before we proceed, we introduce the following definition, where $\CS$ is assumed to be 
a pairing constructed as in Section~\ref{sec:linear}, while $\CR$ is a fixed set of renormalised pairs
as before.

\begin{definition}
We say that $e\in \CR\cap\CS$ is a {\it bad pair} if there exists
an $f\in \CR$ such that the condition
\eref{e:conditionAsmall} of Lemma~\ref{lem:smaller-product}
is not satisfied, namely
\begin{equ}[e:bad]
\Vert e^+ - f^-\Vert_{\s} \wedge
\Vert e^- - f^+ \Vert_{\s}
\geq
\Vert e^+ - e^-\Vert_{\s} \wedge
\Vert f^+ - f^-\Vert_{\s}\;.
\end{equ}
If such an $f$ exists, one must have $f\in\CR\setminus\CS$ since the construction of $\CS$
guarantees that any two pairs
$e,f \in \CS$ do satisfy \eqref{e:conditionAsmall}.
We say that $e\in \CR\cap\CS$ is a {\it good pair} if it is not a bad pair.
We denote by $\CD\subseteq\CR\cap\CS$ the set of good pairs and by 
$\CD^c\eqdef (\CR\cap\CS)\setminus \CD$ the set of bad pairs.
\end{definition}


If we were to do the expansion \eref{eq:identity} for a ``bad pair'', 
the conditions of Lemma~\ref{lem:Delta-by-A} and Lemma~\ref{lem:smaller-product}
would be violated. 
Therefore we will only do the expansion for ``good pairs''.
The next important proposition shows that
one can bound $\CH(x;\J)$ by $\prod_{\{i,j\}\in \mathcal S} \J_{ij}^-$
with a paring $\CS$, multiplied by some factors $\CA_{ef}$, in such a way
that there appear additional factors $\|e^+-e^-\|_\s$ 
taming the non-integrable divergence of the function $K(e)\J^-(e)$,
for every $e\in\CD$. Furthermore, these factors $\CA_{ef}$ are all bounded, so that we will have freedom to erase some of them for convenience of the integrations over all the space-time points in Subsection~\ref{subsec:PsiKBK-int}.
One may worry that $K(e)\J^-(e)$ would still be non-integrable for a bad pair $e$,
but in fact we will see later that one can just insert a factor $\CA_{ef}$ for 
these pairs ``for free'' (see the proof of Proposition~\ref{prop:boundH8} below).

\begin{proposition}
\label{prop:bound--H1}
Assume that $\J\in\{\J_\eps^{k^2}, \J^{k^2}\}$.
Let $\mathcal{S}$ be the pairing selected by the procedure in Section~\ref{sec:linear} and
let $\CD$ be the set of good pairs. One has the bound
\begin{equ}
 \CH(x;\J)
 \lesssim
 \Big(\prod_{\{i,j\}\in \mathcal S} \J_{ij}^-\Big)
 {\sum_{\mathbf{P}\subset \CR^2}}^{\!\!\prime}\prod_{(e,f)\in\mathbf{P}}\mathcal{A}_{ef} \;,
\label{eq:Bound--H}
\end{equ}
where the sum $\sum'$
is restricted to those sets $\mathbf{P}$ such that
\begin{claim}
\item for every $e\in \CD$, 
there exists at least one $f \in \CR, f\neq e$ with $(e,f) \in \mathbf{P}$,
\item for every $(e,f) \in \mathbf{P}$, one has $e\in\CD$ or $f\in\CD$,
\item  for every $(e,f) \in \mathbf{P}$, one has $\CA_{ef}\lesssim 1$.
\end{claim}
\end{proposition}

\begin{proof}
The following bound will turn out to be useful for our calculation.
For any set $A \subset \CR$, one has 
\begin{equ}
\Big( \prod_{e\in A}\hat\J_{e} \Big)\Big(
 \prod_{f\in\mathcal{E}(M\setminus A^{\prime})} \hat\J_{f}\Big)
 \lesssim
 \prod_{\{i,j\}\in\mathcal{S}}\J_{ij}^{-}\;.
\label{eq:bound-each}
\end{equ}
Recall here that $A'$ denotes the set of all charges covered by $A$, i.e.\ $A' = \bigcup A$.
In order to show this, we apply Proposition \ref{prop:hierarchical} to the 
collection of points $M\setminus A^{\prime}$,
which allows us to bound the left hand side of \eref{eq:bound-each} 
by the expression $\prod_{\{i,j\}\in\bar{\mathcal{S}}}\J_{ij}^{-}$
for some pairing $\bar{\mathcal{S}}$. By Lemma \ref{lem:max}, the latter is bounded
by the same expression with $\bar{\mathcal{S}}$ replaced by $\mathcal{S}$.

Combining this bound with Proposition~\ref{prop:cancellations}, 
where we choose the ordering of $\CR$ in such a way that $\CD = \{e_1,\ldots,e_\ell\}$
for some $\ell \ge 0$,
we obtain
\begin{equs}  [e:boundCH]
\CH(x;\J)
& \le
 \sum_{A\subseteq \CD}
\sum_{\CB \in \CM_\ell(A)}
|\Delta_{A}^\CB (\J)| 
\; \CH(A,\CB; x; \J)\\
&\lesssim \Big( \sum_{A\subseteq \CD}
\sum_{\CB \in \CM_\ell(A)}
|\Delta_{A}^\CB(\J)| \Big) 
\Big(
\prod_{\{i,j\}\in\CS}\J_{ij}^{-}\Big)\;, 
\end{equs}
where we used the fact that $\CH(A,\CB)$ is nothing but an alternating
sum of terms of the type appearing in the left hand side of \eref{eq:bound-each}, but with
different choices of $A$. We recall that $\Delta_{A}^{\mathcal{B}}$ is, by
definition, given by
\begin{equ}[e:defDeltaAB]
\Delta_{A}^{\mathcal{B}} = \prod_{e \in A} \Bigl(\prod_{f \in \CB_e} \Delta_{e}^f\Bigr)\;.
\end{equ}

At this stage, we observe that
since $A\subseteq\CD$, every $e\in A$ is a good pair and therefore
 for each of the quadrupoles $(e,f)$ appearing in \eref{e:defDeltaAB}, the shortest distance between 
$\|x_{e_+} - x_{e_-}\|_\s$, $\|x_{f_+} - x_{f_-}\|_\s$, $\|x_{f_+} - x_{e_-}\|_\s$
and $\|x_{e_+} - x_{f_-}\|_\s$ is always one of the first two, at least
up to a constant multiple depending only on $m$.
Then, we can apply Lemma \ref{lem:Delta-by-A} to obtain the following bound
\begin{equ} [e:applied-DeltabyA]
\CH(x;\J)
\lesssim \Big( \sum_{A\subseteq \CD}
\sum_{\mathcal{B} \in \CM_\ell(A)}
\prod_{e \in A} \prod_{f \in \CB_e}
\CA_{ef} \Big) 
\Big(
\prod_{\{i,j\}\in\mathcal{S}}\J_{ij}^{-}\Big)\;.
\end{equ}
By the definition of $\CM_\ell(A)$, for every $A,\CB$ in the above summation, one has $\CD\subseteq U_A^{\CB}$, in other words
for every $e\in \CD$, 
there exists at least one $f\neq e$ such that the factor $\CA_{ef}$ (or possibly $\CA_{fe}$, but
these are identical) appears 
in \eref{e:applied-DeltabyA}. Since $A \subset \CD$ it is also the case that for 
every factor $\CA_{ef}$ appearing
in \eref{e:applied-DeltabyA}, one has either $e\in\CD$ or $f\in\CD$.
Finally, the definition of ``good pairs'' implies that for every factor $\CA_{ef}$ appearing
in \eref{e:applied-DeltabyA} the bound \eqref{e:conditionAsmall} holds and thus $\CA_{ef}\lesssim 1$.
Therefore we can indeed bound the right hand side of \eref{e:applied-DeltabyA} by
a multiple of
\begin{equ}
 \Big(\prod_{\{i,j\}\in \mathcal S} \J_{ij}^-\Big)
 {\sum_{\mathbf{P}\subset \CR^2}}^{\!\!\prime}
 \prod_{(e,f)\in\mathbf{P}}\mathcal{A}_{ef} \;,
\end{equ}
where the sum $\sum'$
is restricted to those sets $\mathbf{P}$ satisfying all the conditions described in the statement.
\end{proof}



\subsection{Moments of $\Psi_\eps^{k\bar k}$: integrations} \label{subsec:PsiKBK-int}

The bound given in Proposition~\ref{prop:bound--H1} turns out not to be 
very convenient to use
when one tries to actually perform the final integration over the positions of the charges,
so we will first derive a slightly weaker bound which has a ``nicer'' form.
We start with some definitions.
Suppose that we are given a graph $\CG = (\CV,\CE)$ with vertices $\CV$
and edges $\CE$. For a subset of vertices $\CV' \subseteq \CV$, we then define 
a subgraph $\CG_{\CV'} = (\CV', \CE')$,
with $\CE'$ consisting of the edges in $\CE$ with both ends in $\CV'$.

We can also define
a graph $\CG^{\CV'}$ 
by identifying all the vertices in $\CV'$
as one vertex called $v$, so that the set of vertices of $\CG^{\CV'}$ is 
given by $(\CV\setminus\CV') \sqcup \{v\}$.
Regarding the edges of $\CG^{\CV'}$, we postulate that 
$(x,y)$ is an edge of $\CG^{\CV'}$ if and only if
either $v\notin \{x,y\}$ and $(x,y)\in\CE$, or
 $x=v$, $y\neq v$, and there exists $z\in \CV'$ such that $(z,y)\in\CE$.
The set of edges of $\CG^{\CV'}$
can be identified canonically with the set $\CE \setminus \{\mbox{edges with both ends in }\CV'\}$.
If the original graph $\CG$ is directed, both $\CG^{\CV'}$ and $\CG_{\CV'}$ inherit
its direction in the obvious way.

Given a vertex set $\CV$, 
we define the set of \textit{admissible graphs} $\mathbf{G}_\CV$
to be the set of all directed graphs over $\CV$
such that every vertex has degree at least  $1$ and
there exists a partitioning $\CV = \CV_L \sqcup \CV_T$ of $\CV$ with the following
properties:
\begin{claim}
\item Each connected component of $\CG^{\CV_L}$ is a tree. The 
connected component containing the distinguished vertex $v$ is considered as a rooted tree
with root $v$ and all other connected components should contain at least two vertices.
\item  Each connected component of $\CG_{\CV_L}$ is a directed loop.
\end{claim}
Here, a  {\it directed loop} is a connected graph with at least two vertices such that  every 
vertex is of degree $2$ and has exactly one directed edge going into it and the other directed edge going out of it.
A {\it tree} is a non-empty connected graph without loops.
Given an admissible graph $\CG$, we furthermore write $\CE_L$ for the edges
connecting two vertices in $\CV_L$ and $\CE_T$ for the remaining edges.
We also remark that if $\CG$ is admissible, then the decomposition $\CV = \CV_L \sqcup \CV_T$ is
unique. See Figure~\ref{fig:adm-graph} for a generic admissible graph.

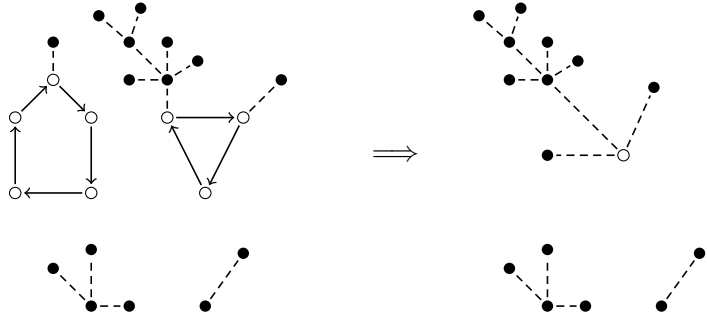
\begin{figure}
\centering
\begin{tikzpicture}[scale=0.5]
\node at (0,0) [loopnode] (1) {};
\node at (0,2) [loopnode] (2) {};
\node at (1,3) [loopnode] (3) {};
\node at (2,2) [loopnode] (4) {};
\node at (2,0) [loopnode] (5) {};
\node at (1,4) [treenode] (6) {};
\draw[loopline]  (1) to (2);
\draw[loopline] (2) to (3);
\draw[loopline] (3) to (4);
\draw[loopline] (4) to (5);
\draw[loopline] (5) to (1);
\draw[treeline] (3) to (6);
\node at (4,2) [loopnode] (7) {};
\node at (6,2) [loopnode] (8) {};
\node at (5,0) [loopnode] (9) {};
\draw[loopline] (7) to (8);
\draw[loopline] (8) to (9);
\draw[loopline] (9) to (7);
\node at (4,3) [treenode] (10) {};
\node at (4.8, 3.5) [treenode] (11) {};
\node at (4,4) [treenode] (12) {};
\node at (3,4) [treenode] (13) {};
\node at (3,3) [treenode] (14) {};
\node at (2.2, 4.7) [treenode] (15) {};
\node at (3.3, 4.9) [treenode] (16) {};
\node at (7, 3) [treenode] (17) {};
\draw[treeline] (7) to (10);
\draw[treeline] (10) to (11);
\draw[treeline] (10) to (12);
\draw[treeline] (10) to (13);
\draw[treeline] (10) to (14);
\draw[treeline] (13) to (15);
\draw[treeline] (13) to (16);
\draw[treeline] (8) to (17);
\node at (2, -3) [treenode] (18) {};
\node at (3, -3) [treenode] (19) {};
\node at (2, -1.5) [treenode] (20) {};
\node at (1, -2) [treenode] (21) {};
\draw[treeline] (18) to (19);
\draw[treeline] (18) to (20);
\draw[treeline] (18) to (21);
\node at (5, -3) [treenode] (22) {};
\node at (6, -1.6) [treenode] (23) {};
\draw[treeline] (22) to (23);
\node at (10,1)  {$\Longrightarrow$};
\node at (16,1) [loopnode] (31) {};
\node at (14,1) [treenode] (32) {};
\draw[treeline] (31) to (32);
\node at (14,3) [treenode] (100) {};
\node at (14.8, 3.5) [treenode] (110) {};
\node at (14,4) [treenode] (120) {};
\node at (13,4) [treenode] (130) {};
\node at (13,3) [treenode] (140) {};
\node at (12.2, 4.7) [treenode] (150) {};
\node at (13.3, 4.9) [treenode] (160) {};
\node at (16.8, 2.8) [treenode] (170) {};
\draw[treeline] (31) to (100);
\draw[treeline] (100) to (110);
\draw[treeline] (100) to (120);
\draw[treeline] (100) to (130);
\draw[treeline] (100) to (140);
\draw[treeline] (130) to (150);
\draw[treeline] (130) to (160);
\draw[treeline] (31) to (170);
\node at (14, -3) [treenode] (180) {};
\node at (15, -3) [treenode] (190) {};
\node at (14, -1.5) [treenode] (200) {};
\node at (13, -2) [treenode] (210) {};
\draw[treeline] (180) to (190);
\draw[treeline] (180) to (200);
\draw[treeline] (180) to (210);
\node at (17, -3) [treenode] (220) {};
\node at (18, -1.6) [treenode] (230) {};
\draw[treeline] (220) to (230);
\end{tikzpicture}
\caption[An admissible graph $\CG$ (left) and the corresponding graph $\CG^{\CV_L}$ (right).]{An admissible graph $\CG$ (left) and the corresponding graph $\CG^{\CV_L}$ (right).
The admissible graph consists of four connected components.
Elements of $\CV_L$ are shown as little circles \tikz[baseline=-3, auto] {
\node at (0,0) [loopnode] {}; },
 elements of $\CV_T$ are shown as black dots \tikz[baseline=-3, auto] {
\node at (0,0) [treenode] {}; },
 edges in $\CE_L$ are shown as solid lines with arrows,
  while edges in $\CE_T$ are shown as dashed lines.
}
\label{fig:adm-graph}
\end{figure}

Now let $\CR$ be the set of renormalised pairs as above, which is going to be
our vertex set. 
To every $\CG \in \mathbf{G}_\CR$,
we associate a pairing $\CS_\CG$ of the $2m$ charges
as follows.
If $e\in\CV_T$, then $\{e^+,e^-\}\in\CS_\CG$.
If $e\in\CV_L$, and therefore there exist $f,g \in \CV_L$ such that $(f,e)$ is an edge pointing 
from $f$ to $e$ and $(e,g)$ is an edge pointing 
from $e$ to $g$, then $\{f^+,e^-\}\in\CS_\CG$ and $\{e^+,g^-\}\in\CS_\CG$. In particular,
we only care about the orientation of the edges connecting vertices in $\CV_L$.
With this notation at hand, we can reformulate our bound on $\CH(x,\CJ)$ as follows.

\begin{proposition}
\label{prop:boundH8}
Assume that $\J\in\{\J_\eps^{k^2}, \J^{k^2}\}$.
One has the bound
\begin{equ}[e:boundH8]
 \mathcal{H}(x;\J)
 \lesssim
 \sum_{\CG\in\mathbf{G}_\CR}
 \Big(\prod_{\{i,j\}\in \CS_\CG} \J_{ij}^-\Big)
 \Big(\prod_{(e,f)\in \CE_T}\CA_{ef}\Big) \;.
\end{equ}
\end{proposition}

\begin{proof}
Given a pairing $\CS$ (in practice we take the specific pairing selected 
in Section~\ref{sec:linear})
and a set $\mathbf{P} \subset \CR^2$ satisfying the conditions
listed in Proposition~\ref{prop:bound--H1},
we construct an admissible graph 
$\CG \in \mathbf{G}_\CR$ as follows. 

First, we define a set of oriented edges $\CE_L$ by setting
\begin{equ}
\CE_L = \{(e,f) \, :\, \{e^+,f^-\}\in\CS\setminus\CR\}\;.
\end{equ}
This set of edges has the property that if 
$(e,f) \in \CE_L$, then we necessarily have $\{e,f\} \subset \CR\setminus\CS$.
Furthermore, one can see that the graph $(\CR,\CE_L)$ consists of loops of length
at least two, as well as of singletons, with the singletons consisting of $\CR \cap \CS$, 
see Figure~\ref{fig:pairings}. 
(If we had only imposed that $\{e^+,f^-\}\in\CS$, the graph would consist
of loops with every vertex belonging to exactly one loop, but some loops could consist of
only one vertex.)
We therefore \textit{define} at this stage $\CV_L = \CR\setminus\CS$.

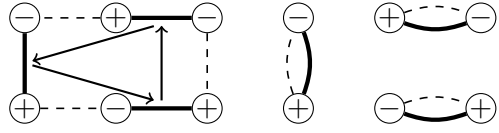
\begin{figure}
\centering
\begin{tikzpicture}[scale=0.8]
\node at (3,0) [charge] (4) {$+$};  \node at (3,1.5) [charge] (2) {$-$};
\node at (0,1.5) [charge] (3) {$-$};  \node at (1.5,1.5) [charge] (1) {$+$};
\node at (0,0) [charge] (5) {$+$};  \node at (1.5,0) [charge] (6) {$-$};

\node at (4.5,0) [charge] (11) {$+$};  \node at (6,0) [charge] (21) {$-$};
\node at (4.5,1.5) [charge] (31) {$-$};  \node at (6,1.5) [charge] (51) {$+$};
\node at (7.5,0) [charge] (41) {$+$};  \node at (7.5,1.5) [charge] (61) {$-$};

\draw [arrow] (2.25,1.4) to (0,0.75);  
\draw [arrow] (0,0.75) to (2.25,0.1);  
\draw [arrow] (2.25,0) to (2.25,1.5);  

\draw [renorm] (1) to (2);  
\draw [kerAlg2] (1) to (3); 
\draw [kerAlg2] (2) to (4); 
\draw [renorm] (3) to (5);  
\draw [renorm] (4) to (6);  
\draw [kerAlg2] (5) to (6); 

\draw [kerAlg2] (11) to [bend left=20] (31); 
\draw [renorm] (11) to [bend right=20] (31); 
\draw [kerAlg2] (21) to [bend left=20] (41); 
\draw [renorm] (21) to [bend right=20] (41); 
\draw [kerAlg2] (51) to [bend left=20] (61); 
\draw [renorm] (51) to [bend right=20] (61); 
\end{tikzpicture}
\caption{Generic situation for the construction of $\CE_L$: pairs in $\CR$ are drawn as 
thick lines and pairs in $\CS$ are drawn as dotted lines. The arrows show the edges 
belonging to $\CE_L$.
}\label{fig:pairings}
\end{figure}

We now consider the graph $\CG_\P = (\CR,\P)$. Here, we note that by the constraints on 
$\P$ given in Proposition~\ref{prop:bound--H1}, every edge in $\P$ contains at least one element of
$\CR \cap \CS$ (the ``singletons'' of the first step) so that the reduced graph
$\CG_\P^{\CV_L}$ contains the same edges as $\CG_\P$.
We then select an arbitrary spanning forest $\CE_T^{(1)} \subset \P$ for $\CG_\P^{\CV_L}$. 
In other words, $\CE_T$ is such that 
the connected components of $(\CR,\CT)^{\CV_L}$ are the same as those of $\CG_\P^{\CV_L}$, 
but each such component is a tree. (Here, the orientation of these edges is irrelevant.)

Finally, let $\CR^{(0)} \subset \CR$ be those vertices in $\CR \cap \CS$ that are not in
$\P'$. Because of the first condition on $\mathbf{P}$, any $e \in \CR^{(0)}$ necessarily
belongs to $\CD^c$, i.e.\ it is a ``bad pair''. 
Therefore, for every such $e$, there exists 
an $f_e\in\CR\setminus\CS$ such that \eref{e:bad} holds.
We then define $\CE_T^{(0)} = \{(e,f_e)\,:\,e \in \CR^{(0)}\}$.

With all of these definitions at hand, we now set $\CG = (\CR, \CE_L \cup \CE_T)$
with $\CE_T = \CE_T^{(0)} \cup \CE_T^{(1)}$,
which is indeed an admissible graph with decomposition
$\CV_L = \CR \setminus \CS$ and $\CV_T = \CR \cap \CS$. 
Furthermore, our construction and the definition of $\CS_\CG$ guarantee that 
one actually has $\CS_\CG=\CS$.
Finally, we claim that one has 
\begin{equ}[e:wantedBound]
\prod_{(e,f)\in\mathbf{P}}\mathcal{A}_{ef} 
\lesssim  \prod_{(e,f)\in \CE_T}\mathcal{A}_{ef}\;.
\end{equ}
Indeed, since $\CE_T^{(1)} \subset \P$ by construction and, for every 
$(e,f)\in\mathbf{P}\setminus\CE_T$,
one has $\CA_{ef}\lesssim 1$ by the last condition on $\mathbf{P}$,
this bound holds with $\CE_T$ replaced by $\CE_T^{(1)}$.
On the other hand, for every $(e,f)\in\CE_T^{(0)}$,
one has $\CA_{ef}\gtrsim1$ by combining the definition of a bad pair 
with Lemma~\ref{lem:smaller-product}, so that \eqref{e:wantedBound} does hold.
The claim now follows by applying the above inequality to the right hand side of
\eref{eq:Bound--H} and then bounding it by the sum over all 
possible admissible graphs.
\end{proof}

The bound \eqref{e:boundH8} has two major advantages: first, it does not make any reference to
the special pairing $\CS$ anymore, so that we now have one single bound which holds 
for any configuration of charges $x$. Second, the tree structure given by the notion
of ``admissible graph'' will make it possible to bound the integral \eqref{e:mth-moment} by inductively
integrating over the variables corresponding to the ``leaves'' until we are only left
with the ``loops'' which can then be handled separately.
Now we have all the elements in place to give the proof to  Theorem~\ref{theo:second-order}. For a simpler notation, from now on we will write
\begin{equ}
\bar \beta \; \eqdef \; {\beta^2 \over 2\pi} \;.
\end{equ}
We now have everything in place for the proof of Theorem~\ref{theo:second-order}.
We first give the proof of the bounds and convergence statements for $\Psi^{k\bar k}_\eps$.
In Section~\ref{sec:Psikk} below, we then bound the objects $\Psi^{k k}_\eps$, while the bounds
on $\Psi^{k k}_\eps$ and $\Psi^{k k}_\eps$ with $k \neq \ell$ are postponed to Section~\ref{sec:kneql}.

\subsection{Moments of $\Psi_\eps^{k\bar k}$}

\begin{proof}[of Theorem~\ref{theo:second-order} for $\Psi^{k\bar k}_\eps$]
We first prove the statements for $\Psi^{\pm}_\eps$ which is the harder case. The modifications
required to obtain the analogous statements for $\Psi^{k\bar k}_\eps$ with $k>1$ 
will be indicated at the end of the proof. Recall from \eqref{e:mth-moment}
that for $m=2N$, one has
\begin{equ}
\E |\scal{\phi_0^\lambda,  \Psi^{\pm}_\eps}|^{m}
=
\int \CH(x;\J_\eps) 
\prod_{e\in\CR}
\Big(
	\phi_0^\lambda(e_\downarrow)   K(e)
\Big)
\,dx\;.
\end{equ}
As a consequence of Proposition \ref{prop:boundH8}, we can bound this expression by
\begin{equ}
\E |\scal{\phi_0^\lambda,  \Psi^{\pm}_\eps}|^{m}
 \lesssim 
\sum_{\CG\in\mathbf{G}_\CR}
\int 
 \Big( \prod_{\{i,j\}\in \CS_\CG} \J_{ij,\eps}^-
 \prod_{(e,f)\in \CE_T}\CA_{ef} \Big) 
\prod_{e\in\CR}  \Big(
|\phi_0^\lambda(e_\downarrow)| \, | K(e) |
\Big)
\,dx.
\end{equ}

The proof of Theorem~\ref{theo:second-order} now goes by induction over $m$. 
Suppose that
for every  $\bar m< m$ and every $\delta > 0$, the bound
\begin{equ}[e:fullExprInductive]
\int 
 \Big( \prod_{\{i,j\}\in \CS_{\bar\CG}} \J_{ij,\eps}^-
 \prod_{(e,f)\in \bar\CE_T}\CA_{ef} \Big) 
\Big(
\prod_{e\in\bar\CR}
|\phi_0^\lambda(e_\downarrow)|\,    | K(e) |
\Big)
\,dx
\lesssim  \lambda^{(2-\bar\beta-\delta)\bar m}\;,
\end{equ}
holds uniformly over $\lambda \in (0,1]$
for every admissible graph $\bar \CG \in \mathbf{G}_{\bar\CR}$
over the set $\bar\CR$ of cardinality $\bar m$.
Here we are not assuming any more that
exactly half of the pairs in $\bar\CR$ are oriented from the positive to the negative
and the other half of the pairs the other way.
We aim to prove that in this case the analogous bound also holds
for every admissible graph $\CG \in \mathbf{G}_\CR$
over $\CR$ of cardinality $m$. 
(Again, not assuming that 

To make the calculations more clear and visual, we introduce some graphical notation.
A line 
\tikz[baseline=-3] {
\node [dot] at (0,0)  (0) {}; 
\node [dot] at (1.5,0)  (1) {};
\node [xshift=-8] at (0) {$x$};
\node [xshift=8] at (1) {$y$};
\draw [kernel] (0) to node[homoge] {$\alpha$} (1); } 
with a label $\alpha$ represents the function 
$\|x-y\|_s^\alpha$. 
A dashed line with an arrow
\tikz[baseline=-3] {
\node [dot] at (0,0)  (0) {}; 
\node [dot] at (1.5,0)  (1) {};
\node [xshift=-8] at (0) {$x$};
\node [xshift=8] at (1) {$y$};
\draw [KK,->] (0) to (1); } 
represents the function $K(x\to y) = K(x-y)-K(x)$. 
A gray dot 
\tikz[baseline=-3] {
\node [graydot] at (0,0)  {}; 
}
means that the point is integrated out, while
a black dot 
\tikz[baseline=-3] {
\node [dot] at (0,0)  {}; 
}
simply stands for a point without integration.
We then distinguish between the following two cases.

\noindent\textbf{Case~1.} In this case, we assume that $\CV_T\neq\emptyset$, where $\CV_T$ is
associated to the admissible graph $\CG$ as above. 
Since $\CG^{\CV_L}$ is a union of disjoint trees, one can always find a vertex $e$ such that
the degree of $e$ is one (namely, $e$ is a leaf.) 
Let $f$ be the unique pair in $\CR$ such that $(e,f)\in\CE_T$.
There are then two possible situations. The first situation is that
$f\in\CV_T$ and $f$ is also a leaf (see the bottom-right connected component of the graph in Figure \ref{fig:adm-graph}). In this situation, the integration over $e^{\pm}$
and $f^\pm$ factors out and is either of the form
\begin{equ}[e:form1]
I_1 \eqdef \int 
|\varphi_0^{\lambda}(e^+)|
|\varphi_0^{\lambda}(f^-)|
\, | K(e^- \to e^+) K(f^+ \to f^-) |   \; \hat\J_e \hat\J_f   \CA_{ef}
\,de^\pm df^\pm \;,
\end{equ}
where the integration is over $(\R^3)^4$, or of the form
\begin{equ}[e:form2]
I_2 \eqdef \int 
|\varphi_0^{\lambda}(e^+)|
|\varphi_0^{\lambda}(f^+)|
\, | K(e^- \to e^+) K(f^- \to f^+) |  \; \hat\J_e \hat\J_f   \CA_{ef}
\,de^\pm df^\pm \;.
\end{equ}
Of course, it could also be \eref{e:form1} or \eref{e:form2}
with all the signs flipped, but these can be reduced to the above two cases by symmetry. 
Leaving aside the test functions $\phi_0^\lambda$, 
the integrands in $I_1$ and $I_2$, are depicted by 
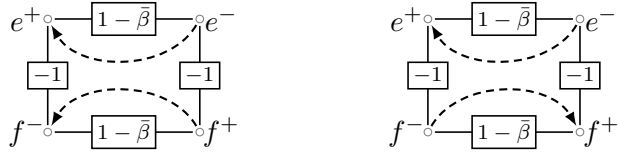
\begin{center}
\begin{tikzpicture} 
\node at (0,0) [graydot] (1) {};
\node at (2,0) [graydot] (0) {};
\node at (0,1.5) [graydot] (2) {};
\node at (2, 1.5) [graydot] (3) {};
\node [xshift=8] at (0) {$f^+$};
\node [xshift=-8]  at (1) {$f^-$};
\node [xshift=-8]  at (2) {$e^+$};
\node [xshift=8]  at (3) {$e^-$};
\draw [kernel] (0) to node [homoge] {$1-\bar \beta$} (1) ;
\draw [kernel] (2) to node[homoge] {$1-\bar\beta$} (3);
\draw [kernel] (0) to node[homoge] {$-1$} (3);
\draw [kernel] (1) to node[homoge] {$-1$} (2);
\draw[KK,->, bend right = 60] (0) to (1);
\draw[KK,->,bend left = 60] (3) to (2);
\node at (5,0) [graydot] (11) {};
\node at (7,0) [graydot] (10) {};
\node at (5,1.5) [graydot] (12) {};
\node at (7, 1.5) [graydot] (13) {};
\node [xshift=8] at (10) {$f^+$};
\node [xshift=-8]  at (11) {$f^-$};
\node [xshift=-8]  at (12) {$e^+$};
\node [xshift=8]  at (13) {$e^-$};
\draw [kernel] (10) to node[homoge] {$1-\bar \beta$} (11);
\draw [kernel] (12) to node[homoge] {$1-\bar\beta$} (13);
\draw [kernel] (10) to node[homoge] {$-1$} (13);
\draw [kernel] (11) to node[homoge] {$-1$} (12);
\draw[KK,->, bend left = 60] (11) to (10);
\draw[KK, ->,bend left = 60] (13) to (12);
\end{tikzpicture}
\end{center}
respectively. By Lemma \ref{lem:int-12} we have
\begin{equ} [e:bound12]
|I_1| + |I_2| \lesssim \lambda^{2\,(2-\bar\beta-\delta)}\;,
\end{equ}
for every sufficiently small $\delta>0$,
so that the statement follows from the induction hypothesis with
$\bar m = m-2$ and $\bar\CR=\CR\setminus\{e,f\}$. 

The second situation is that when either $f\notin\CV_T$ or $f\in\CV_T$ but has 
degree greater than one. In this situation, the integration over $e^{\pm}$ again factors out and
 has the form
\begin{equ}[e:form3]
I_3 \eqdef 
\int 
\varphi_0^{\lambda}(e^+)
| K(e^-\to e^+) | \; \hat\J_e   \CA_{ef}
\;de^+de^- \;,
\end{equ}
or the same expression with all the signs flipped. 
Graphically, ignoring again the test function, the integrand is given by
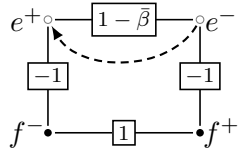
\begin{center}
\begin{tikzpicture} 
\node at (2,0) [dot] (f+) {};
\node at (0,0) [dot] (f-) {};
\node at (0,1.5) [graydot] (e+) {};
\node at (2, 1.5) [graydot] (e-) {};
\node [xshift=8] at (f+) {$f^+$};
\node [xshift=-8]  at (f-) {$f^-$};
\node [xshift=-8]  at (e+) {$e^+$};
\node [xshift=8]  at (e-) {$e^-$};
\draw [kernel] (f+) to node[homoge] {$1$} (f-);
\draw [kernel] (e+) to node[homoge] {$1-\bar\beta$} (e-);
\draw [kernel] (f+) to node[homoge] {$-1$} (e-);
\draw [kernel] (f-) to node[homoge] {$-1$} (e+);
\draw[KK, ->,bend left = 60] (e-) to (e+);
\end{tikzpicture}
\end{center}
Note now that the integrand in the full expression \eqref{e:fullExprInductive}
necessarily contains either a factor
$|\varphi_0^{\lambda}(f^+)|$ or a factor $|\varphi_0^{\lambda}(f^-)|$. Therefore,
we can restrict the integral to those configurations for which $\|f^+\|_\s \wedge \|f^-\|_\s \le \lambda$, which allows us to apply  Lemma~\ref{lem:int-3} below, thus yielding the bound
$|I_3| \lesssim \lambda^{2-\bar\beta-\delta}$,
for every $\delta>0$.
The required bound now follows by using the induction assumption with
$\bar m = m-1$, $\bar\CR=\CR\setminus\{e\}$.

\noindent\textbf{Case~2.} We now turn to the case when $\CV_T = \emptyset$ (which in particular implies that $\CE_T = \emptyset$), and therefore $\CE_L\neq \emptyset$. In this case, the integral
\eqref{e:fullExprInductive} factors according to the connected components of the graph $\CG$, which consist of loops.
The integral for a loop of size $n$ linking vertices $\{e_i\}_{i=1}^n \subset \CR$ is given by
\begin{equ}[e:formL]
I_L = \int 
\prod_{i=1}^n
\Big(
\varphi_0^{\lambda}(e_{i,\downarrow})
\, | K(e_i) | \,
\J_{e^-_i e^+_{i+1}}^-
\Big)  
\; 
\prod_{i=1}^n 
\Big(  de^+_i de^-_i  \Big)\;,
\end{equ}
where we made the identification $e_{n+1}\eqdef e_1$. 
Furthermore, each $e_i \in \CR$ comes with an arbitrary orientation which appears in the
definition of $K(e_i)$. Integrals of this type are bounded in Lemma~\ref{lem:int-L} below,
which yields
\begin{equ} [e:boundL]
|I_L| \lesssim \lambda^{(2-\bar\beta-\delta)n} \;,
\end{equ}
thus again allowing us to invoke the induction hypothesis with $\bar m = m - n$
and $\bar \CR = \CR \setminus \{e_1,\ldots,e_n\}$.

We now turn to prove the convergence statement of the theorem.
Define for $e=(e^+,e^-),f=(f^+,f^-)\in\R^3\times\R^3$
\begin{equ}
\CH_{\eps,\bar\eps}(e,f)  
\eqdef 
\frac{\J_{\eps,\bar\eps}(e^+ - f^+)\J_{\eps,\bar\eps}(e^- - f^-)}
{\J_{\eps}(e^+ - e^-)\J_{\bar\eps}(f^+ - f^-)\J_{\eps,\bar\eps}(e^+ - f^-)\J_{\eps,\bar\eps}(e^- - f^+)}
- \hat\J_{e,\eps} \hat\J_{f,\bar\eps}
\end{equ} 
where $\J_{\eps,\bar\eps}$ is defined in the proof of Theorem \ref{theo:convBasic}.
Then one has
\begin{equs}
\E |\scal{\phi_0^\lambda, \Psi^{\pm}_\eps & - \Psi^{\pm}_{\bar\eps}}|^2
=
\int 
(\CH_\eps (e,f)   + \CH_{\bar\eps} (e,f)   - 2 \CH_{\eps,\bar\eps}(e,f) ) \\
& \times \Big(
\phi_0^\lambda(e^+) \phi_0^\lambda(f^-)  
K(e^- \to e^+)  K(f^+ \to f^-)
\Big)
\,de^\pm df^\pm\;.
\end{equs}
Assume without loss of generality that $\bar\eps\le\eps$.
As in the proof of Theorem \ref{theo:convBasic}, one has $\CJ_\eps^- = \CJ^-\, \exp(\CM_\eps)$
and $\CJ_{\eps,\bar \eps}^- = \CJ^-\, \exp(\CM_{\eps,\bar \eps})$ where the functions $\CM_\eps$
and $\CM_{\eps,\bar \eps}$ are bounded by
$
|\CM_\eps(z)|  + |\CM_{\eps,\bar \eps}(z)| \lesssim \eps /\|z\|_\s
$
for all space-time points $z$ with $\|z\|_\s \ge \eps$,
and, the function $\CJ_{\eps,\bar \eps}^-$ also falls within the scope of Lemma  \ref{lem:Delta-by-A}
and therefore satisfies the bound \eref{e:Delta-by-A}.

For our collection of four charges $e^\pm,f^\pm$,
there are only two possible admissible graphs:
the first one is $\CV_L = \{e,f\}$ (i.e. $e$ and $f$ form a loop),
and the second one is $\CV_T = \{e,f\}$ (i.e. $e$ and $f$ form a tree).
It is then straightforward to show that
\begin{equs}
| \CH_\eps & (e,f)    + \CH_{\bar\eps} (e,f)   - 2 \CH_{\eps,\bar\eps}(e,f) | \\
& \lesssim \|e^+ -  f^-\|_\s^{-\bar\beta} \|e^- -  f^+\|_\s^{-\bar\beta}
 \sum_{x\neq y\in\{e^\pm,f^\pm\}} \Big(\frac{\eps}{\|x-y\|_\s} \wedge 1\Big) \\
& \quad + \|e^+ -  e^-\|_\s^{-\bar\beta} \|f^+ -  f^-\|_\s^{-\bar\beta}
\CA_{ef}
 \Big(\frac{\eps}{\|e^+-e^-\|_\s} \wedge 1+ \frac{\eps}{\|f^+-f^-\|_\s} \wedge 1 \Big)
\end{equs}
for all $e^\pm,f^\pm\in\R^3$.

Now to perform the integrations over $e^\pm, f^\pm$, 
one needs the following fact: suppose that 
$|D^k K_1 (x)| \lesssim (\frac{\eps}{\|x\|_\s} \wedge 1) \|x\|_\s^{\zeta_1 - |k|_\s}$,
and $K_2$ is of order $\zeta_2$, then
\begin{equs}[e:generalise-conv]
|D^k (K_1*K_2) (x)| \lesssim \eps^\kappa \|x\|_\s^{\bar\zeta - |k|_\s -\kappa}
\end{equs}
for sufficiently small $\kappa>0$ where $\bar\zeta = \zeta_1 + \zeta_2 -|\s| \notin \N$ and $k$   
is such that $\bar\zeta - |k|_\s <0$.
To prove \eref{e:generalise-conv},
observe that if $\|x\|_\s < 2\eps$, then one just bounds
$|D^k K_1 (x)|$ by  $\|x\|_\s^{\zeta_1 - |k|_\s}$ and then
uses $\|x\|_\s^\kappa \lesssim \eps^\kappa$ to obtain the desired bound.

If $\|x\|_\s \geq 2\eps$ on the other hand, writing 
$(K_1*K_2) (x)$ as $\int K_1 (y) K_2(x-y) dy$,
we distinguish three cases as in the proof
of \cite[Lemma~10.14]{Regularity}. The first case is that $\|y\|_\s < \|x\|_\s /2$:
we bound $\|x-y\|_\s$ by $\|x\|_\s$, and integrate $K_1 (y)$
following the steps above \eref{e:finalBoundPsi}.
The second case is that $\|y-x\|_\s < \|x\|_\s /2$ and therefore 
$\eps / \|y\|_\s  < 1$,
we can bound $K_1 (y)$ by $\eps^\kappa \|x\|_\s^{\zeta_1 -\kappa}$.
The third case is the complement of the above two regions,
where one still has $\eps / \|y\|_\s < 1$,
following the same arguments as in the proof
of \cite[Lemma~10.14]{Regularity} one obtains the desired bound.


We can then integrate over $e^\pm, f^\pm$ analogously as in the proof of
\ref{lem:int-3} and \ref{lem:int-L}.
One has
\begin{equ}
\E |\scal{\phi_0^\lambda, \Psi_\eps - \Psi_{\bar \eps}}|^2 \lesssim \eps^{2\kappa} \lambda^{-2\kappa+2(2-\bar\beta-\delta)}\;,
\end{equ}
so the second bound stated by the theorem follows by Cauchy-Schwarz inequality.

We now prove the bounds for $\Psi_\eps^{k\bar k}$ with $k>1$. One has
\begin{equ}
\E |\scal{\phi_0^\lambda,  \Psi^{k\bar k}_\eps}|^{m}
=  e^{-\beta^2 m \left(k^2 - 1 \right)Q_\eps(0)} 
\int \CH(x;\J_\eps^{k^2}) 
\prod_{e\in\CR}
\Big(
	\phi_0^\lambda(e)   K(e)
\Big)
\,dx\;.
\end{equ}
By Proposition \ref{prop:boundH8}, one has
\begin{equs}
\E |\scal{\phi_0^\lambda,  \Psi^{k\bar k}_\eps}|^{m}
 \lesssim  \eps^{{\beta^2\over 2\pi} m \left(k^2 - 1 \right)} 
\sum_{\CG\in\mathbf{G}_\CR}
& \int   
 \Big( \prod_{\{i,j\}\in \CS_\CG} \J_{ij,\eps}^{-k^2}
 \prod_{(e,f)\in \CE_T}\CA_{ef} \Big)   \\
&  \times 
\prod_{e\in\CR}  \Big(
\phi_0^\lambda(e)  | K(e) |
\Big)
\,dx\;.
\end{equs}
In the above expression, there are $m$ of  factors $\J_\eps^{-k^2}$.
In fact, for some sufficiently small parameter $\kappa>0$, one has
\begin{equ}
\eps^{{\beta^2\over 2\pi} \left(k^2 - 1 \right)}
 \J_{\eps} (x_i - x_j)^{-k^2}
\lesssim 
\eps^\kappa
\|x_i - x_j\|_\s^{{-\beta^2\over 2\pi} - \kappa } \;.
\end{equ}
Then, the required bounds follow in the same way 
as the case of $\Psi_\eps^\pm$,
except that $\eps^\kappa \to 0$ as $\eps\to 0$.
\end{proof}

Now we proceed to prove the bounds for all the integrals in the proof of the previous theorem. Notice that the entire integral comes with a test function $\varphi^\lambda_0(f^+)$ or $\varphi^\lambda_0(f^-)$, which justifies the assumption of the following lemma.

\begin{lemma} \label{lem:int-3}
Let $I_3$ be given by \eqref{e:form3}. Then the bound $|I_3| \lesssim \lambda^{2-\bar\beta-\delta}$ holds uniformly over all $f^\pm$ such that $\|f^+\|_\s \wedge \|f^-\|_\s \le \lambda$ and $\|f^+\|_\s \vee \|f^-\|_\s \lesssim 1$.
\end{lemma}

\begin{proof}
By the gradient theorem, one has
\begin{equs}
| K(e^- & \to e^+) |  \lesssim \left|K(e^{+}-e^{-})-K(e^{-})\right|   \\
& \lesssim \|e^+\|_\s^{3-\bar\beta-\delta} 
\Big(  \|e^-\|_\s^{\bar\beta-5+\delta} + \|e^+ - e^- \|_\s^{\bar\beta-5+\delta}  \Big)
\end{equs}
for every  $\delta>0$ sufficiently small, where we used the fact $3-\bar\beta \in (0,1]$.
By our definitions, the left hand side of \eref{e:form3} is bounded by
$\lambda^{-4} \|f^{-}-f^{+}\|_{\s} \,(T_1 + T_2)$, where
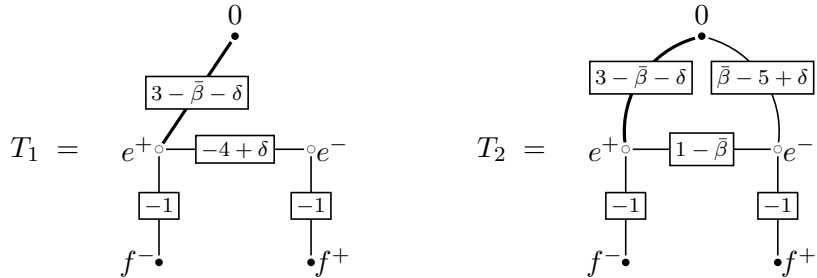
\begin{center}
\begin{tikzpicture} 
\node at (-1.5,1.5) {$T_1 \; =$} ;
\node at (1,3) [dot] (0) {};
\node at (2,0) [dot] (f+) {};
\node at (0,0) [dot] (f-) {};
\node at (0,1.5) [graydot] (e+) {};
\node at (2, 1.5) [graydot] (e-) {};
\node [yshift=8] at (0) {$0$};
\node [xshift=8] at (f+) {$f^+$};
\node [xshift=-8]  at (f-) {$f^-$};
\node [xshift=-8]  at (e+) {$e^+$};
\node [xshift=8]  at (e-) {$e^-$};
%
\draw [kernel] (e+) to node[homoge] {$-4+\delta$} (e-);
\draw [kernel] (f+) to node[homoge] {$-1$} (e-);
\draw [kernel] (f-) to node[homoge] {$-1$} (e+);
\draw [kernel,very thick] (e+) to node[homoge,semithick] {$3-\bar\beta-\delta$} (0);
%
\node at (3.8, 1.2) {$\,$};
\end{tikzpicture}
\begin{tikzpicture} 
\node at (-1.5,1.5) {$T_2 \; =$} ;
\node at (1,3) [dot] (0) {};
\node at (2,0) [dot] (f+) {};
\node at (0,0) [dot] (f-) {};
\node at (0,1.5) [graydot] (e+) {};
\node at (2, 1.5) [graydot] (e-) {};
\node [yshift=8] at (0) {$0$};
\node [xshift=8] at (f+) {$f^+$};
\node [xshift=-8]  at (f-) {$f^-$};
\node [xshift=-8]  at (e+) {$e^+$}; 
\node [xshift=8]  at (e-) {$e^-$};
%
\draw [kernel] (e+) to node[homoge] {$1-\bar\beta$} (e-);
\draw [kernel] (f+) to node[homoge] {$-1$} (e-);
\draw [kernel] (f-) to node[homoge] {$-1$} (e+);
\draw [kernel,very thick,bend left=40] (e+) to node[homoge,semithick] {$3-\bar\beta-\delta$} (0);
\draw [kernel,bend right=40] (e-) to node[homoge] {$\bar\beta-5+\delta$} (0);
\end{tikzpicture}
\end{center}
and the thick lines indicate that the corresponding (parabolic) distance
is restricted to taking values less than $\lambda$.

We bound the first term $T_1$. Integrating $e^-$ 
using \cite[Lemma~10.14]{Regularity}, one has
\begin{equ}[e:termT_1]
T_{1} 
 \lesssim
 \int_\Lambda \|e^{+}-f^{-}\|_{\s}^{-1} \, 
 	\|e^{+}-f^{+}\|_{\s}^{-1+\delta} \,
	\|e^{+}\|_{\s}^{3-\bar\beta-\delta} \,
	 de^{+}  \;,
\end{equ}
where $\Lambda$ denotes the ball of radius $\lambda$.
We now distinguish two cases. If $\|f^{+}\|_{\s}\geq2\lambda$, then one must have $\|f^{-}\|_\s \leq\lambda$ which together with $\|e^+\|_\s\le \lambda$ implies
$\|e^{+}-f^{+}\|_\s^{-1+\delta} \lesssim \|f^{-}-f^{+}\|_\s^{-1+\delta}$. 
Inserting this bound into \eref{e:termT_1} and integrating over $e^+$, one obtains
\begin{equ} [e:bound-T1]
T_{1} 
\lesssim
\lambda^{6-\bar\beta-\delta}  \, \|f^{-}-f^{+}\|_{\s}^{-1+\delta}  \;.
\end{equ}
If on the other hand $\|f^{+}\|_\s \leq2\lambda$, then one has $\|e^{+}-f^{+}\|_{\s}\lesssim\lambda$ and therefore 
$\|e^{+}-f^{+}\|_{\s}^{-1+\delta}  \lesssim  \lambda^{3}\|e^{+}-f^{+}\|_{\s}^{-4+\delta}$,
so that
\begin{equ}
T_{1}  \lesssim
 \lambda^{6-\bar \beta - \delta}\int_{\R^3} \|e^{+}-f^{-}\|_{\s}^{-1} \, 
 	\|e^{+}-f^{+}\|_{\s}^{-4+\delta} \,
	 de^{+}  \;.
\end{equ}
Integrating over $e^+$, one again obtains \eref{e:bound-T1}, which yields the required bound
on this term.

Next, we bound the term $T_2$. Define the quantity
\begin{equ}
S(e^{+},f^{+})\eqdef
\int\|e^{-}-f^{+}\|_{\s}^{-1}  \,
	\|e^{+}-e^{-}\|_{\s}^{-\bar\beta+1}  \,
	\|e^{-}\|_\s^{\bar\beta-5+\delta}\,de^{-}\;,
\end{equ} 
so that $T_2$ can be rewritten as
\begin{equ}[e:termT_2]
T_{2} 
= \int_\Lambda \|e^{+}-f^{-}\|_{\s}^{-1} \,
	\|e^{+}\|_{\s}^{3-\bar\beta-\delta} \,
	S(e^{+},f^{+}) \,de^{+}\;.
\end{equ}
We estimate $S(e^{+},f^{+})$ using Holder's inequality
\begin{equs}
S(e^{+},f^{+}) & =
\int
\Big(\|e^{-}\|_\s^{\bar\beta-\frac{7}{2}+\frac{\delta}{2}}
	\|e^{+}-e^{-}\|_\s^{-\bar\beta+2-\frac{\delta}{2}}\Big)
\Big(\|e^{-}\|_\s^{-\frac{3}{2}+\frac{\delta}{2}}
	\|f^{+}-e^{-}\|_\s^{-\frac{\delta}{2}}\Big)\\
&\qquad \times
\Big(\|e^{+}-e^{-}\|_\s^{-1+\frac{\delta}{2}}
	\|f^{+}-e^{-}\|_\s^{-1+\frac{\delta}{2}}\Big)\;de^{-}\\
& \lesssim
\Big\Vert  \|e^{-}\|_\s^{\bar\beta-\frac{7}{2}+\frac{\delta}{2}}\|e^{+}-e^{-}\|_{\s}^{-\bar\beta+2-\frac{\delta}{2}}\Big\Vert_{L^{\frac{8}{3}}}
\Big\Vert \|e^{-}\|_\s^{-\frac{3}{2}+\frac{\delta}{2}} \|f^{+}-e^{-}\|_{\s}^{-\frac{\delta}{2}}\Big\Vert_{L^{\frac{8}{3}}}\\
& \qquad \times
\Big\Vert \|e^{+}-e^{-}\|_{\s}^{-1+\frac{\delta}{2}} \|f^{+}-e^{-}\|_{\s}^{-1+\frac{\delta}{2}}\Big\Vert_{L^{4}}\\
&\lesssim
\|e^{+}-f^{+}\|_{\s}^{-1+\delta}
\end{equs}
where all the $L^p$ norms are defined on functions in the variable $e^-$. 
In fact the two $L^{\frac{8}{3}}$ norms are both bounded by constants.
We are now back to the same situation as \eqref{e:termT_1} for $T_1$, so that the claim follows. 
\end{proof}

In the sequel, we will make repeated use of the inequality
\begin{equ}  [e:boundz1]
\|z\|_\s^{-\alpha} \|\bar z\|_\s^{-\beta} \lesssim \|z\|_\s^{-\alpha-\beta} + \|\bar z\|_\s^{-\alpha-\beta}\;,
\end{equ}
which holds for every $z$, $\bar z$ in $\R^4$ and any two exponents $\alpha, \beta > 0$. 

\begin{lemma} \label{lem:int-12}
The bound \eref{e:bound12} holds for $I_1$ and $I_2$ given by \eqref{e:form1}
and \eqref{e:form2}.
\end{lemma}

\begin{proof}
We first show the bound for $I_1$.
Define a function (which also depends on $e^+, f^-$)
\begin{equs}
F(z,w)& \eqdef
\int
\|e^+ - e^-\|_\s^{-\bar\beta+1} K(e^- -z)
\|e^- - f^+\|_\s^{-1} \\
& \qquad\times
\|f^+ - f^-\|_\s^{-\bar\beta+1} K(f^+ - w)
\;de^-df^+
\end{equs}
for every $z,w\in\R^3$. Then
\begin{equs}
|I_1|  \lesssim &
\int 
 \Big| F(e^+,f^-) - F(0,f^-) 
  - F(e^+,0) + F(0,0)\Big| \\
& \qquad \times
\varphi_0^{\lambda}(e^+)
\varphi_0^{\lambda}(f^-)
 \|e^+ - f^-\|_\s^{-1} 
\;de^+df^- \;.
\end{equs}
Since $3-\bar\beta \in (0,1]$,
applying gradient theorem to $K(e^- \to e^+)$  as  in the proof of Lemma \ref{lem:int-3}, one has
\begin{equ}
 | F(e^+,f^-) - F(0,f^-) |  \lesssim H_1 + H_2  \lesssim H_1 + H_3 
\end{equ}
where $H_i$ are defined as follows, and in the last inequality \eref{e:boundz1} has been applied to $\|e^-\|_\s^{\bar\beta-5+\delta}\cdot \|e^+-e^-\|_\s^{1-\bar\beta}$.
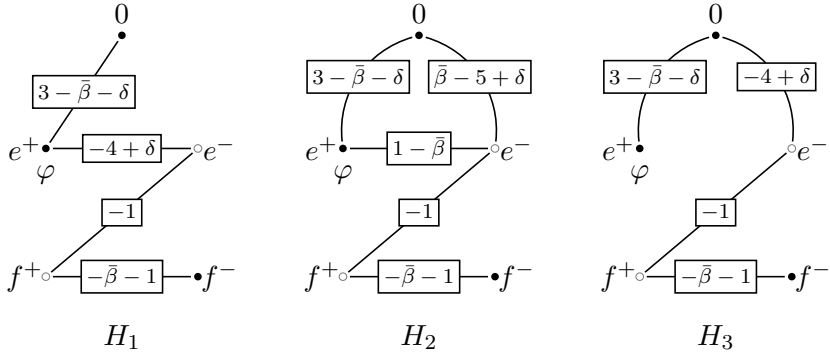
\begin{center}
\begin{tikzpicture} 
\node at (1,0) {$H_1$} ;
\node at (1,4) [dot] (0) {};
\node at (0,0.8) [graydot] (f+) {};
\node at (2,0.8) [dot] (f-) {};
\node at (0,2.5) [dot] (e+) {};
\node at (2, 2.5) [graydot] (e-) {};
\node [yshift=8] at (0) {$0$};
\node [xshift=-8] at (f+) {$f^+$};
\node [xshift=8]  at (f-) {$f^-$};
\node [xshift=-8]  at (e+) {$e^+$};  \node [yshift=-10]  at (e+) {$\varphi$};
\node [xshift=8]  at (e-) {$e^-$};
\draw [kernel] (f+) to node[homoge] {$-\bar\beta-1$} (f-);
\draw [kernel] (e+) to node[homoge] {$-4+\delta$} (e-);
\draw [kernel] (f+) to node[homoge] {$-1$} (e-);
\draw [kernel] (e+) to node[homoge] {$3-\bar\beta-\delta$} (0);
%
\node at (3, 1.2) {$\,$};
\end{tikzpicture}
\begin{tikzpicture} 
\node at (1,0) {$H_2$} ;
\node at (1,4) [dot] (0) {};
\node at (0,0.8) [graydot] (f+) {};
\node at (2,0.8) [dot] (f-) {};
\node at (0,2.5) [dot] (e+) {};
\node at (2, 2.5) [graydot] (e-) {};
\node [yshift=8] at (0) {$0$};
\node [xshift=-8] at (f+) {$f^+$};
\node [xshift=8]  at (f-) {$f^-$};
\node [xshift=-8]  at (e+) {$e^+$};  \node [yshift=-10]  at (e+) {$\varphi$};
\node [xshift=8]  at (e-) {$e^-$};
\draw [kernel] (f+) to node[homoge] {$-\bar\beta-1$} (f-);
\draw [kernel] (e+) to node[homoge] {$1-\bar\beta$} (e-);
\draw [kernel] (f+) to node[homoge] {$-1$} (e-);
\draw [kernel,bend left=40] (e+) to node[homoge] {$3-\bar\beta-\delta$} (0);
\draw [kernel,bend right=40] (e-) to node[homoge] {$\bar\beta-5+\delta$} (0);
\node at (3, 1.2) {$\,$};
\end{tikzpicture}
\begin{tikzpicture} 
\node at (1,0) {$H_3$} ;
\node at (1,4) [dot] (0) {};
\node at (0,0.8) [graydot] (f+) {};
\node at (2,0.8) [dot] (f-) {};
\node at (0,2.5) [dot] (e+) {};
\node at (2, 2.5) [graydot] (e-) {};
\node [yshift=8] at (0) {$0$};
\node [xshift=-8] at (f+) {$f^+$};
\node [xshift=8]  at (f-) {$f^-$};
\node [xshift=-8]  at (e+) {$e^+$};  \node [yshift=-10]  at (e+) {$\varphi$};
\node [xshift=8]  at (e-) {$e^-$};
\draw [kernel] (f+) to node[homoge] {$-\bar\beta-1$} (f-);
\draw [kernel] (f+) to node[homoge] {$-1$} (e-);
\draw [kernel,bend left=40] (e+) to node[homoge] {$3-\bar\beta-\delta$} (0);
\draw [kernel,bend right=40] (e-) to node[homoge] {$-4+\delta$} (0);
\end{tikzpicture}
\end{center}
Performing the convolutions in $e^-$,
and then bounding $\|f^+\|_\s^{-1+\delta}$ by $\|f^+\|_\s^{-1-\delta}$
for $H_3$ and 
bounding $\|e^- - f^+\|_\s^{-1+\delta}$ by $\|e^- - f^+\|_\s^{-1-\delta}$
for $H_1$, and finally integrating $f^+$,
 we obtain
\begin{equ}
| F(e^+,f^-) - F(0,f^-) |
 \lesssim 
\|e^+\|_\s^{3-\bar\beta-\delta} 
\Big(
\|f^-\|_\s^{2-\bar\beta-\delta}
 + \|f^- - e^+\|_\s^{2-\bar\beta-\delta}
\Big)
\end{equ}

In the similar way, applying gradient theorem to $K(e^- \to e^+)$ again as above,
one obtains that $| F(e^+,0) - F(0,0) |$ is bounded by the sum of the following two terms
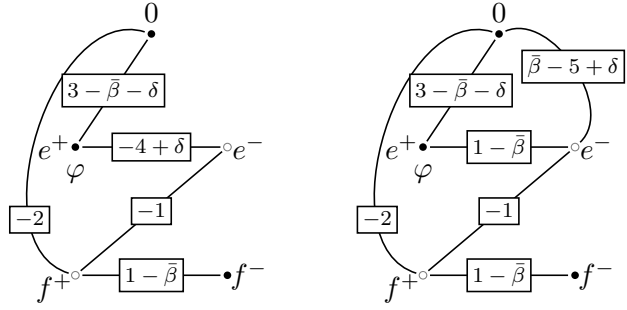
\begin{center}
\begin{tikzpicture} 
\node at (1,4) [dot] (0) {};
\node at (0,0.8) [graydot] (f+) {};
\node at (2,0.8) [dot] (f-) {};
\node at (0,2.5) [dot] (e+) {};
\node at (2, 2.5) [graydot] (e-) {};
\node [yshift=8] at (0) {$0$};
\node [xshift=-8,yshift=-5] at (f+) {$f^+$};
\node [xshift=8]  at (f-) {$f^-$};
\node [xshift=-8]  at (e+) {$e^+$};  \node [yshift=-10]  at (e+) {$\varphi$};
\node [xshift=8]  at (e-) {$e^-$};
\draw [kernel] (f+) to node[homoge] {$1-\bar\beta$} (f-);
\draw [kernel,bend left=90] (f+) to node[homoge,near start] {$-2$} (0);
\draw [kernel] (e+) to node[homoge] {$-4+\delta$} (e-);
\draw [kernel] (f+) to node[homoge] {$-1$} (e-);
\draw [kernel] (e+) to node[homoge] {$3-\bar\beta-\delta$} (0);
%
\node at (3, 1.2) {$\,$};
\end{tikzpicture}
\begin{tikzpicture} 
\node at (1,4) [dot] (0) {};
\node at (0,0.8) [graydot] (f+) {};
\node at (2,0.8) [dot] (f-) {};
\node at (0,2.5) [dot] (e+) {};
\node at (2, 2.5) [graydot] (e-) {};
\node [yshift=8] at (0) {$0$};
\node [xshift=-8,yshift=-5] at (f+) {$f^+$};
\node [xshift=8]  at (f-) {$f^-$};
\node [xshift=-8]  at (e+) {$e^+$};  \node [yshift=-10]  at (e+) {$\varphi$};
\node [xshift=8]  at (e-) {$e^-$};
\draw [kernel] (f+) to node[homoge] {$1-\bar\beta$} (f-);
\draw [kernel,bend left=90] (f+) to node[homoge,near start] {$-2$} (0);
\draw [kernel] (e+) to node[homoge] {$1-\bar\beta$} (e-);
\draw [kernel] (f+) to node[homoge] {$-1$} (e-);
\draw [kernel] (e+) to node[homoge] {$3-\bar\beta-\delta$} (0);
\draw [kernel,bend right=80] (e-) to node[homoge] {$\bar\beta-5+\delta$} (0);
\node at (3, 1.2) {$\,$};
\end{tikzpicture}
\end{center}
Applying \eref{e:boundz1} to $\|e^-\|_\s^{\bar\beta-5+\delta}\cdot \|e^+-e^-\|_\s^{1-\bar\beta}$, and to $\|f^+ - f^-\|_\s^{-\bar\beta+1} \cdot \|f^+ \|_\s^{-2} $,
it is then straightforward to obtain the bound

%
\begin{equ}
 | F(e^+,0) - F(0,0) |
 \lesssim 
\|e^+\|_\s^{3-\bar\beta-\delta} 
\Big(
\|f^-\|_\s^{2-\bar\beta+\delta}
 + \|e^+\|_\s^{2-\bar\beta+\delta}
  + \|f^- - e^+\|_\s^{2-\bar\beta+\delta} 
\Big)
\end{equ}
Then the integrations over $e^+, f^-$ are straightforward; 
this concludes the proof for the desired bound on $I_1$.

The bound for $I_2$ follows simply by expanding
$K_+(e) K_+(f) $ into four terms according to the definitions and then bounding the integral with each term
separately, using \eref{e:boundz1}.
\end{proof}

The bound \eref{e:boundL} holds as a consequence of the following result for integrating general ``cycles" or ``chains".
Before stating the result we introduce a notation.

We denote by $K_i^{\varphi}(x,y)$ 
 functions that are given by either $\varphi^\lambda_0(y) K(x\to y)$ or  $\varphi^\lambda_0(x) K(y \to x)$. 
Given real numbers $\{\alpha_1,...,\alpha_n,\alpha',\bar\alpha'\}$,
we aim to bound the integration of the following functions
\begin{equ}
F_\CL = F_\CL \Big(\{x\}_{i=1}^n, \{y\}_{i=1}^n \Big)
\eqdef \prod_{i=1}^n \Big( \left| K_i^\varphi(x_i, y_i) \right| \cdot
	\|y_i - x_{i+1}\|_\s^{\alpha_i} \Big)  
\end{equ}
 (with $n\ge 2$) where
$x_{n+1}$ is identified with $x_{1}$; and 
$F_\CC=F_\CC  \big (\{x\}_{i=1}^n, \{y\}_{i=1}^n ,z,\bar z \big)$ with
\begin{equ}
F_\CC  
 \eqdef  g(x_1,z) \, \bar g(y_{n},\bar z) 
 \prod_{i=1}^{n-1} 
\Big( \left| K_i^\varphi(x_i, y_i) \right| \cdot
	\|y_i - x_{i+1}\|_\s^{\alpha_i} \Big)
K_n^\varphi(x_{n}, y_{n})
\end{equ}
(with $n\ge 1$) where 
\begin{equ}
g(x_1,z) \eqdef \|x_1-z\|_\s^{\alpha'} \,\varphi^\lambda_0(z)
\end{equ}
and $\bar g$ is defined in the same way 
with change of roles $x_1 \leftrightarrow y_n$, $z \leftrightarrow \bar z$ and $\alpha' \leftrightarrow \bar\alpha'$.

\begin{remark}
By inspection, one can realise that 
$F_\CL$ corresponds to a cycle shaped graph $\CL$:
$(x_1 - y_1 - ... - x_n - y_n - x_1)$,
and, $F_\CC$ corresponds to a chain shaped graph $\CC$:
$(z - x_1 - y_1 - ... - x_n - y_n - \bar z)$.
All the variables $x_i,y_i$ and $z,\bar z$ will be integrated.
For the case of $F_\CC$, we will allow $\alpha'$ (the same discussion applies to $\bar\alpha'$) to be zero,
which means $g(x_1,z)=2 \varphi^\lambda_0(z)$ will be factored out
and the integral of it over $z$ gives a constant; in other words 
one can simply think of the chain as ending with the function $K(x_1,y_1)$.
Our notation is just in order to treat the chain in a unified way
no matter it ends with a function $K$ or $g$.
\end{remark}

\begin{lemma} \label{lem:int-L}
In the setting above, suppose that $n \ge 2$ and that $\alpha_i \in (-4,-2]$ for $i \in \{1,\ldots,n\}$.
Then  one has
\begin{equ} [e:int-loops]
	\int F_\CL \,dx\,dy
	\lesssim \lambda^{\mathfrak{h}(\CL)-\delta}  \;,
\end{equ}
for any $\delta>0$ arbitrarily small, where 
$\mathfrak{h}(\CL) \eqdef 2n+\sum_{i=1}^n \alpha_i$,
and the integration is over $x = \{x_i\}_{i=1}^n \in (\R^3)^n$,
$y = \{y_i\}_{i=1}^n  \in (\R^3)^n$.

Suppose additionally that $\alpha',\bar\alpha'\in(-4,-2]\cup\{0\}$,
and if $\alpha'=0$ (resp. $\bar\alpha'=0$) then
$x_1$ (resp. $y_n$) is an incoming point.
Then, for every $n \ge 1$, one has
\begin{equ} [e:int-chains]
	\int F_\CC \,dx\,dy\,dz\,d\bar z
	\lesssim \lambda^{\mathfrak{h}(\CC)-\delta}  \;,
\end{equ}
for any $\delta>0$ arbitrarily small, where 
$\mathfrak{h}(\CC)\eqdef 2n+\sum_{i=1}^{n-1} \alpha_i
+\alpha'+\bar\alpha'$ and
the integration is over $x = \{x_i\}_{i=1}^n \in (\R^3)^n$,
$y = \{y_i\}_{i=1}^n  \in (\R^3)^n$, and $z,\bar z\in\R^3$.
\end{lemma}

\begin{proof}
The integrand $F_\CL$, ignoring the test functions, can be depicted graphically by
the left picture below
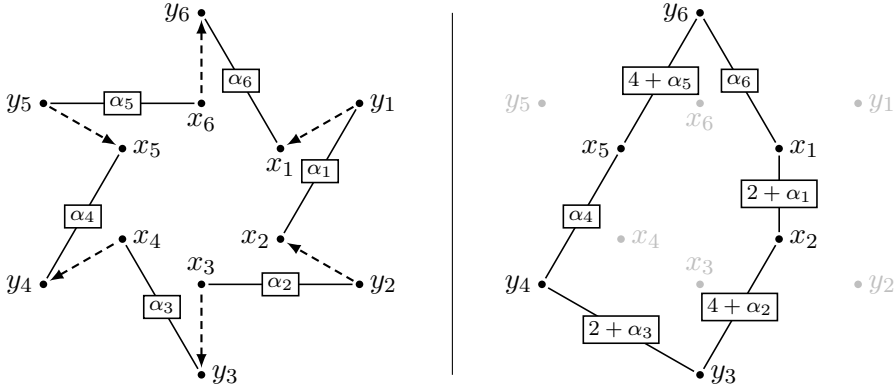
\begin{center}
\begin{tikzpicture}[scale=1.2]
%
\node at (90:1) [dot] (0a) {}; \node [below] at (0a) {$x_6$}; 
\node at (90:2) [dot] (0b) {}; \node [left] at (0b) {$y_6$}; 
\node at (30:1) [dot] (1a) {}; \node [below] at (1a) {$x_1$}; 
\node at (30:2) [dot] (1b) {};\node [right] at (1b) {$y_1$}; 
\node at (-30:1) [dot] (2a) {}; \node [left] at (2a) {$x_2$}; 
\node at (-30:2) [dot] (2b) {};\node [right] at (2b) {$y_2$}; 
\node at (-90:1) [dot] (3a) {}; \node [above] at (3a) {$x_3$}; 
\node at (-90:2) [dot] (3b) {}; \node [right] at (3b) {$y_3$}; 
\node at (-150:1) [dot] (4a) {}; \node [right] at (4a) {$x_4$}; 
\node at (-150:2) [dot] (4b) {};\node [left] at (4b) {$y_4$}; 
\node at (150:1) [dot] (5a) {};\node [right] at (5a) {$x_5$};  
 \node at (150:2) [dot] (5b) {};\node [left] at (5b) {$y_5$}; 
\draw [kernel] (0b) to node[homoge] {$\alpha_6$} (1a);
\draw [kernel] (1b) to node[homoge] {$\alpha_1$} (2a);
\draw [kernel] (2b) to node[homoge] {$\alpha_2$} (3a);
\draw [kernel] (3b) to node[homoge] {$\alpha_3$} (4a);
\draw [kernel] (4b) to node[homoge] {$\alpha_4$} (5a);
\draw [kernel] (5b) to node[homoge] {$\alpha_5$} (0a);
\draw [KK,<-] (0b) to (0a);
\draw [KK,->] (1b) to (1a);
\draw [KK,->] (2b) to (2a);
\draw [KK,<-] (3b) to (3a);
\draw [KK,<-] (4b) to (4a);
\draw [KK,->] (5b) to (5a);

\draw (2.75,2) -- (2.75,-2);

\node [right=6.5cm] at (0b) [dot] (0b) {}; \node [left] at (0b) {$y_6$}; 
\node [right=6.5cm] at (1a) [dot] (1a) {}; \node [right] at (1a) {$x_1$}; 
\node [right=6.5cm] at (2a) [dot] (2a) {}; \node [right] at (2a) {$x_2$}; 
 \node [right=6.5cm] at (3b) [dot] (3b) {};\node [right] at (3b) {$y_3$}; 
 \node [right=6.5cm] at (4b) [dot] (4b) {};\node [left] at (4b) {$y_4$}; 
\node [right=6.5cm]  at (5a) [dot] (5a) {}; \node [left] at (5a) {$x_5$}; 
\node [right=6.5cm] at (4a) [dot,lightgray] (4a) {}; \node [right,lightgray] at (4a) {$x_4$}; 
\node [right=6.5cm] at (3a) [dot,lightgray] (3a) {}; \node [above,lightgray] at (3a) {$x_3$}; 
\node [right=6.5cm] at (0a) [dot,lightgray] (0a) {}; \node [below,lightgray] at (0a) {$x_6$}; 
 \node [right=6.5cm] at (5b) [dot,lightgray] (5b) {};\node [left,lightgray] at (5b) {$y_5$}; 
\node [right=6.5cm] at (1b) [dot,lightgray] (1b) {};\node [right,lightgray] at (1b) {$y_1$}; 
\node [right=6.5cm] at (2b) [dot,lightgray] (2b) {};\node [right,lightgray] at (2b) {$y_2$}; 

\draw [kernel] (0b) to node[homoge] {$\alpha_6$} (1a);
\draw [kernel] (1a) to node[homoge] {$2+\alpha_1$} (2a);
\draw [kernel] (2a) to node[homoge] {$4+\alpha_2$} (3b);
\draw [kernel] (3b) to node[homoge] {$2+\alpha_3$} (4b);
\draw [kernel] (4b) to node[homoge] {$\alpha_4$} (5a);
\draw [kernel] (5a) to node[homoge] {$4+\alpha_5$} (0b);
\end{tikzpicture}
\end{center}

The picture illustrates the generic situation (with $n=6$) showing that
the orientations of $\{x_i,y_i\}$ are arbitrary.
We will first integrate out all the outgoing points (see the definitions of outgoing / incoming points of oriented pairs in the beginning of Subsection \ref{subsec:PsiKBK-ren}).
 We claim that after these integrations, 
 one has the bound
\begin{equ} [e:claimLoop]
\Big| \int F_\CL \,dxdy\Big| \lesssim \int 
\prod_{i=1}^n
\Big(
	\varphi_0^{\lambda}(z_i) \; G_{l_i}(z_i,z_{i+1})  \Big)
\;  dz
\end{equ}
where the integration is over $z = \{z_i\}_{i=1}^n \in (\R^3)^n$, 
and $z_{n+1}=z_1$, $l_i\in (-4,2)$, 
$\sum_{i=1}^n l_i =\mathfrak{h}(\CL)-\delta$ and
\begin{equ} [e:def-G-cases]
G_{l_i}(x,y)=
\begin{cases}
\|x-y\|_\s^{l_i}
& \quad l_i\in(-4,-2] \;, \\
\|x-y\|_\s^{l_i} +  \|x\|_\s^{l_i} + \|y\|_\s^{l_i}
& \quad l_i\in(-2,0] \;, \\
\|x\|_\s^{l_i}
+ \|y\|_\s^{l_i}
& \quad l_i \in(0,2] \;.
\end{cases}
\end{equ}

The integrand of \eref{e:claimLoop}, ignoring the test functions, is drawn as the right picture above 
(where only the subscripts of $G$ are drawn; and the dummy $z$-variables are still written as $x$ or $y$-variables
to make a  clearer comparison with the left picture and the variables that have been 
integrated out are still indicated in light gray).
We substitute the definition of $G_k$ into \eref{e:claimLoop}
and expand, and obtain a sum where each summand falls into the scope of
\eref{e:test-loop} of Lemma~\ref{lem:int-test} below
(in fact $\alpha_i>-4$ implies $2n+\sum_{i=1}^n \alpha_i>-2n
\ge -4(n-1)$ for $n\ge 2$, so the assumption of \eref{e:test-loop} of Lemma~\ref{lem:int-test} is satisfied).
Therefore the bound \eref{e:int-loops}  follows.

To show the claim  \eref{e:claimLoop},
one needs to show the following bounds.
\begin{claim}
\item The case of integrating out a single point 
when its two neighboring points are both incoming points
(e.g. the $y_1$ in the picture):
for any $\alpha\in(-4,-2]$,
\begin{equ}
 \int 
| K(y \to x) | \cdot
\|y-z\|_\s^{\alpha} 
\;dy
\lesssim
\|z-x\|_\s^{2+\alpha-\delta} + \|z\|_\s^{2+\alpha-\delta}   \;;
\end{equ}
\item The case of two incoming points are adjacent so ``nothing has to be integrated" (e.g. in the picture, 
the successive charges $y_4$ and $x_5$ are both incoming points, so neither of them need to be integrated now); 
\item The case of integrating out two adjacent outgoing points  (e.g. in the picture, 
the successive charges $y_2$ and $x_3$ are both outgoing points, so both of them have to be integrated now):
for any $\alpha\in(-4,-2]$, 
\begin{equ} [e:int-two-charges]
 \iint 
| K(y \to x') | \cdot
\|y-x\|_\s^{\alpha}  \cdot
| K(x \to y') |
\;dxdy
\lesssim
G_{4+\alpha-\delta}(x'-y') 
 \;.
\end{equ}
\end{claim}
The first bound follows from  \cite[Lemma~10.14]{Regularity}.
The second case is trivial. 
To show the last bound, one writes
\begin{equ}
Q (x'-y')\eqdef \int  K(x' - y) \cdot
\|x-y\|_\s^{\alpha}  \cdot
K(y' -y) \;dxdy  \;.
\end{equ}
Then, the left hand side of \eref{e:int-two-charges} is given by
\begin{equs}
| Q(x'-y')  & - Q(x') - Q(y') + Q(0,0) | 
=| \hat Q(x'-y') - \hat Q(x') - \hat Q(y' ) |
\end{equs}
where
\begin{equ}
\hat Q (x)\eqdef  Q(x) - Q(0) - x\cdot\nabla Q(0)\;.
\end{equ}
It is then straightforward to show that 
$|\hat Q(x)|\lesssim \|x\|_\s^{4+\alpha-\delta}$.
Therefore \eref{e:int-two-charges} follows and we obtain \eref{e:claimLoop}.
This completes the proof of the bound for integration of $F_\CL$.

The integration of $F_\CC$ can be bounded analogously. 
Note first that $F_\CC$ can not simply be a function $K$, since if $n=1$
by assumption of the lemma 
one has $\alpha' \wedge \bar\alpha'<0$.
In fact,
there exist 
$l_i\in (-4,2)$ for $0\le i \le n+1$, 
and $\sum_{i=0}^{n+1} l_i =\mathfrak{h}(\CC)-\delta$,
 such that one has
\begin{equ} [e:claimChain]
\Big| \int F_\CC \,dxdydzd\bar z\Big| \lesssim \int 
\prod_{i=0}^{n+1}
	\varphi_0^{\lambda}(z_i) \; 
\prod_{i=0}^{n} 
	G_{l_i}(z_i,z_{i+1})
\;  dz
\end{equ}
where $G_{l_i}$ are defined in \eref{e:def-G-cases}  
and the integration is over $z_0,...,z_{n+1}$. 
The integration variables $z_0, z_{n+1}$ correspond to $z,\bar z$ respectively,
and $z_1,...,z_{n}$ correspond to the incoming points, i.e. the points that
have not been integrated.
To show \eref{e:claimChain}, we integrate out all the outgoing points
in the same way as above, except that we only need to treat the two ends 
of the chain separately.
Since the chain is symmetric under reflection we only consider the end at the function $g$.
If $\alpha'=0$, by assumption $x_1$ is an incoming point,
so we simply take $l_0=0$;
the factored function $\varphi^\lambda_0(z_0)$ can be simply 
integrated out over $z_0$ which gives a constant.
If  $\alpha'<0$, then arguing as above we can have the bound \eref{e:claimChain} 
with  $l_0=\alpha'\in(-4,-2]$ if $y_1\to x_1$, or $l_0=\alpha'+2\in(-2,0]$ if $x_1\to y_1$.

As before we can then expand the right hand side 
and obtain a sum in which each summand  falls into the scope  of 
\eref{e:test-chain} of Lemma~\ref{lem:int-test} below.
\end{proof}


\begin{lemma} \label{lem:int-test}
Given $n$ real numbers $\{\alpha_i\}_{i=1}^n$, 
let $g_{\alpha_i}(x,y)$ be one of the three functions:
$\|x-y\|_\s^{\alpha_i}$ with $\alpha_i>-4$, or, $\|x\|_\s^{\alpha_i}$ or $\|y\|_\s^{\alpha_i}$ with $\alpha_i>-2$.
Let $\bar \alpha=\sum_{i=1}^n \alpha_i$.
Then:

1) Assuming $n\geq 2$ and $\bar\alpha>-4(n-1)$, with $z_{n+1}$ identified with $z_1$, the following bound hold
\begin{equ} \label{e:test-loop}
\int_{(\R^3)^n} \prod_{i=1}^n
	g_{\alpha_i}(z_i,z_{i+1})  
\prod_{i=1}^n \Big( \varphi_0^{\lambda}(z_i) \;  dz_i  \Big)	
\lesssim \lambda^{\bar \alpha}  \;.
\end{equ}

2) Assuming $n \ge 1$, the following bound hold
\begin{equ} \label{e:test-chain}
\int_{(\R^3)^{n+1}}
\prod_{i=1}^n	
g_{\alpha_i}(z_i,z_{i+1})
 \prod_{i=1}^{n+1}
\Big( 	\varphi_0^{\lambda}(z_i)
\;  dz_i  \Big)
\lesssim \lambda^{\bar\alpha} \;.
\end{equ}
\end{lemma}
\begin{proof}
First of all we bound all the functions $\varphi^\lambda_0 (z)$ 
by $\lambda^{-4}$ times the characteristic function for the set $\Lambda = \{z\,:\, \|z\|_\s \le \lambda\}$.
We can therefore bound every $g_{\alpha_i}(z_i,z_{i+1})$ with positive $\alpha_i$ 
by $\lambda^{\alpha_i}$ and restrict the integration of all the $z_k$ over $\Lambda$.
The integral then factorises into integrations of the form
\begin{equ}
\int \prod_{i=k}^{\ell-1}
	\|z_i  -  z_{i+1}\|_\s^{\alpha_i}  \, dz
\end{equ}
where $1\le k \le \ell \le n$ with $\alpha_k, ..., \alpha_{\ell-1} <0$, and 
the integration is one of the following cases
\begin{claim}
\item an integration over $z_k,...,z_{\ell}$;
\item an integration over $z_k,...,z_{\ell-1}$ with $z_\ell \equiv 0$;
or its ``symmetric" case: an integration over $z_{k+1},...,z_{\ell}$ with $z_k \equiv 0$;
\item an integration over $z_{k+1},...,z_{\ell-1}$, where $z_k=z_\ell \equiv 0$;
\item an integration over $z_1,...,z_n$, with $z_{n+1}$ identified with $z_1$.
\end{claim}
In the first case, one can successively integrate the variables
using the assumption $\alpha_i>-4$,
for instance 
\begin{equ}
	\int_\Lambda \prod_{i=k}^{\ell-1}
	\|z_i  -  z_{i+1}\|_\s^{\alpha_i}  \, dz_k  \lesssim 
	\lambda^{\alpha_k + 4}
	\prod_{i=k+1}^{\ell-1}   \|z_i - z_{i+1}\|_\s^{\alpha_i}
\end{equ}
The second case follows in a similar way by starting to integrate from $z_k$ or $z_l$ that is not the one fixed to be $0$.

For the third case, we can integrate $z_{k+1}$:
\begin{equ}
	\int_\Lambda   \|z_{k+1}\|_\s^{\alpha_k}
	\|z_{k+1}  -  z_{k+2}\|_\s^{\alpha_{k+1}}  \, dz_{k+1}  \lesssim 
	  \| z_{k+2}\|_\s^{\bar\alpha_{k+1}}
\end{equ}
where $\bar\alpha_{k+1} \eqdef \alpha_k+\alpha_{k+1}+4$.
If $\bar\alpha_{k+1}\geq 0$, then we bound the right hand side above 
by $\lambda^{\bar\alpha_{k+1}}$, and the rest of the integral falls into the second case. If $\bar\alpha_{k+1} < 0$, then note that
by the assumption of the Lemma, $\bar\alpha_{k+1} > -2 - 4 +4=-2$,
and therefore 
 we can proceed to integrate $z_{k+2}$ in the same way as $z_{k+1}$.
We iterate this procedure until either it reduces to the second case, or
 $k+2=\ell-1$, namely $z_{k+2}$ is the last integration variable
 and we are left with
 \begin{equ}
	\int_\Lambda   \|z_{\ell-1}\|_\s^{\bar \alpha_{\ell-2}}
	 \|z_{\ell-1}\|_\s^{\alpha_{\ell-1}}  \, dz_{\ell-1}  
\end{equ}
where $\bar \alpha_{\ell-2}=\sum_{i=k}^{\ell-2} \alpha_i + 4(\ell-k-2)$.
Then since $\bar\alpha_{\ell-2} + \alpha_{\ell-1} >-2-2=-4$, it is integrable 
and bounded by $\lambda^{\sum_{i=k}^{\ell-1} \alpha_i + 4(\ell-k-1)}$.
Since there is an overall factor $\lambda^{-4(\ell-k-1)}$ from all the functions $\varphi^\lambda_0$, one obtains the desired bound.

The last case happens only when $g_{\alpha_i}(x,y)=\|x-y\|_\s^{\alpha_i}$ for every $i\in\{1,...,n\}$, so that we are in a situation of a whole cycle consisting of $n$ points.
 We can integrate the variables one by one from $z_1$ to $z_{n-2}$ as in the third case, and the condition $\alpha_i>-4$ guarantees integrability. Then we are left with an integration of $ \|z_{n-1}-z_n\|_\s^{\bar\alpha+4(n-2)}$, and by the assumption on $\bar\alpha$ one has $\bar\alpha+4(n-2)>-4$,
 so we can integrate $z_{n-1},z_n$ over $\Lambda$ to get a factor $\lambda^{\bar\alpha+4n}$.
With the overall factor  $\lambda^{-4n}$ from test functions we obtain the desired bound.
\end{proof}

\subsection{Moments of $\Psi_\eps^{kk}$}
\label{sec:Psikk}

We now turn to consider the objects defined in \eref{e:def-PsiKL},
whose $m$-th moment, with $m = 2N$, can be expressed as
\begin{equ}[e:PsiPP]
\E |\scal{\phi_0^\lambda,  \Psi^{kl}_\eps}|^{m}
= \E \,\Big[ \Big| \iint \phi_0^\lambda(x) \,  (K(x-y)-K(-y))  \,
	\Psi^k_\eps(x) \Psi^l_\eps(y) 
	\, dx\, dy  \Big|^{2N}  \Big]
\;.
\end{equ}
In this subsection we show the required bounds for the case $k=l$.


Similarly as in Subsection~\ref{subsec:PsiKBK-ren}, 
we would like to rewrite the $2N$-th power of the integral
as an integral over $4N$ variables,
which again leads us to a situation with $2m=4N$ charges,
and we denote by $M$ a set of cardinality $2m$ indexing them.
Again, each charge $i\in M$ comes with a sign $\sigma_i \in \{\pm\}$,
an index $h_i \equiv k$, and a location $x_i \in \R^3$.
There are again $m$ positive charges
(corresponding to the arguments of $\Psi_\eps^k$)
and $m$ negative charges
(corresponding to the arguments of $\bar\Psi_\eps^k$).

Observe that in \eref{e:PsiPP}, $x$ and $y$ are both arguments of $\Psi^k_\eps$,
rather than one for $\Psi^k_\eps$ and the other for $\bar \Psi^k_\eps$ 
as in the discussion for $\Psi_\eps^{k\bar k}$ in Subsection~\ref{subsec:PsiKBK-ren}. Due to this difference, 
we abandon the notation $\CR$ defined in Subsection~\ref{subsec:PsiKBK-ren},  and consider
in this subsection the situation
where the $2m$ charges are partitioned into a set $\CR'$ of $m$ disjoint oriented pairs 
such that there are $N$ pairs containing two positive charges and $N$ pairs containing two negative charges.

\begin{proof}[of Theorem~\ref{theo:second-order} for $\Psi^{k k}_\eps$] 
Recall our notation that for any pair $e=\{i,j\}$ (not necessarily in $\CR'$),
we define a quantity $\hat\J^{(\eps)}_e$ by
\begin{equ}
\hat \J_{\eps,e} \eqdef 
\J_\eps(x_i - x_j)^{\sigma_i \sigma_j }  \;.
\end{equ}
It is straightforward to check that
\begin{equ}
\E |\scal{\phi_0^\lambda,  \Psi^{kk}_\eps}|^{m}
=  e^{-\beta^2 m \left(k^2-1 \right) Q_\eps(0)}
\!\int_{ (\R^3)^M} \Big(
\!\prod_{e \in \CR'} 
\phi_0^\lambda(e_\downarrow) 
K(e)    \Big) 
\Big( \prod_{e\in\mathcal E (M)}  \hat\J_{\eps,e}^{k^2} \Big)
\,dx.
\end{equ}
By the procedure in Section \ref{sec:linear},
with $\J $ chosen to be $\J_\eps^{k^2}$ which certainly satisfies \eref{e:boundKeps},
one obtains a pairing $\CS$ for each configuration of the $2m$ charges.
Therefore,
\begin{equ}
\E |\scal{\phi_0^\lambda,  \Psi^{kk}_\eps}|^{m}
\lesssim \eps^{{\beta^2 \over 2\pi} m (k^2 -1)}
\sum_\CS \int 
\Big(
\prod_{e \in \CR'} 
\phi_0^\lambda(e_\downarrow) 
\big| K(e) \big|
\Big) 
\Big( \prod_{f\in\CS}
\hat\J_{\eps,f}^{k^2} \Big)
\,dx\;,
\end{equ}
where the sum runs over all possible positive-negative pairings of the $2m$ charges.
Note that this time, for every factor $K(e)$ appearing in the integrand,
the two charges in $e$ have the \textit{same} sign,
while in every factor $\hat\J_{\eps,f}$ appearing above,
the two charges of the pair $f$ have opposite signs. In other words, we have
\begin{equ}
\CS \cap \CR' = \emptyset\;,
\end{equ}
for every $\CS$ in the summation. This makes the construction
of the objects $\Psi^{kk}$ much easier.
One can then bound the integral for each $\CS$.
When $k=1$, the integration falls into the scope of Lemma \ref{lem:int-L}
and the required bound for $\Psi_\eps^\oplus$ follows immediately. 
The bounds for $\Psi^{\oplus}_\eps - \Psi^{\oplus}_{\bar\eps}$
and the independence of mollifiers can be shown analogously as before.
When $k>1$, the arguments are the same as for the case of 
$\Psi^{k\bar k}_\eps$ and the moments converge to zero
due to the factors $\eps$.
\end{proof}

\section{Second-order process bounds for $k\neq l$}
\label{sec:kneql}

This final section contains the proof of Theorem~\ref{theo:second-order} for
$\Psi_\eps^{k\bar l}$ and $\Psi_\eps^{kl}$ with $k\neq l$. These are
only required for the full proof of Theorem~\ref{theo:main}, but are not 
needed for the actual definition of the limiting process
loosely described by \eqref{e:model}.

We now prove Theorem \ref{theo:second-order} for 
$\Psi_\eps^{k\bar l}$ and $\Psi_\eps^{k l}$ where $k\neq l$.
As before, the $m$-th moment can be expressed as integrals over $2m$ variables.
Therefore we are now again in a situation with $2m=4N$ charges,
and we still denote by $M$ the set of cardinality $2m$.
Each charge $i\in M$ comes with a sign $\sigma_i \in \{\pm\}$,
an index $h_i \in \{k,l\}$, and a location $x_i \in \R^3$.
There are exactly $m$ positive charges and $m$ negative charges.

The current situation differs from that of Subsection~\ref{subsec:PsiKBK-ren} or \ref{sec:Psikk} since indices of charges vary. Therefore we should
define new sets of pairs in place of $\CR$ and $\CR'$ above.
\begin{claim}
\item
For the case of $\Psi_\eps^{k\bar l}$ (resp. $\Psi_\eps^{kl}$), 
the $2m$ charges are  partitioned into $m$ disjoint  (oriented) pairs, and
we call this set of pairs $\CR_1$ (resp. $\CR_2$), such that:
there are exactly $N$ pairs in $\CR_1$ of the type $(-;l) \to (+;k)$
\footnote{In other words the outgoing point is negative and indexed by $l$,
and the incoming point is positive and indexed by $k$.},
and the other $N$ pairs in $\CR_1$ are of the type $(+;l) \to (-;k)$;
also, there are exactly $N$ pairs in $\CR_2$ of the type
$(+;l) \to (+;k)$, and the other $N$ pairs in $\CR_2$ are of the type $(-;l) \to (-;k)$.
\end{claim}

\begin{proof}[of Theorem~\ref{theo:second-order} for $\Psi^{kl}_\eps$ and $\Psi^{k\bar l}_\eps$ with $k\neq l$]

For {\it any} pair $e=\{i,j\}$ (not necessarily in $\CR_{1,2}$),
we can define the quantity
\begin{equ}[e:J_e-kl]
\tilde \J_e^{(\eps)} \eqdef
	\J_\eps (x_i - x_j)^{\; \sigma_i \sigma_j h_i h_j}\;.
\end{equ}
It is then straightforward to check that
\begin{equ} [e:indexed]
\E |\scal{\phi_0^\lambda,  \Psi^{k\bar l}_\eps}|^{m}
=  e^{ {\beta^2\over 2} m \left(k^2 +l^2-2 \right)Q_\eps(0)}
\int_{(\R^3)^M} \Big(
\prod_{e \in \CR_1} 
\phi_0^\lambda(e_\downarrow) 
K(e)    \Big) 
\prod_{e\in\mathcal E (M)} \tilde\J_e^{(\eps)}
\,dx\;,
\end{equ}
and $\Psi^{kl}_\eps$ satisfies the same identity with $\CR_1$ replaced by $\CR_2$.

Our current situation is different from before and we can't directly apply
the procedure in Section \ref{sec:linear}, because
there is not a unique function which plays the role of $\CJ$ 
in the procedure of Section \ref{sec:linear} any more
(we have instead multiple ones $\CJ_\eps^{hh'}$ with $h,h'\in\{k,l\}$). In fact,
 when two charges of opposite signs
become close, the cancellations such as \eref{e:boundprodK} in that procedure 
do not necessarily hold anymore
since these two charges could have different indices.

Given such a configuration of indexed $2m$ charges, we construct a new configuration
of {\it un-indexed} $m(k+l)$ charges, in other words the charges all have index $1$. The new configuration
is simply defined as follows. For each of the $2m$ charges, assuming that it has a sign $\sigma$ and an index $h\in\{k,l\}$, one regards it 
as $h$ distinct charges, all having the sign $\sigma$ and the same location.
More formally, we denote by $\CM$ a set of cardinality $m(k+l)$ and we fix a map
$\pi \colon \CM \to M$ with the property that $|\pi^{-1}(i)| = h_i$.
For $a \in \CM$ we will make an abuse of notation and write again
$x_a$ for $x_{\pi(a)}$ and $\sigma_a$ for $\sigma_{\pi(a)}$.
We remark that we do not mean to integrate over these ``$m(k+l)$ space-time points": at the end 
we will still integrate over only $2m$ space-time points.
We claim that the following bound holds
\begin{equ} [e:unindexed]
\E |\scal{\phi_0^\lambda,  \Psi^{k\bar l}_\eps}|^{m}
\lesssim \eps^{ {\bar \beta m\over 2} \left(k +l-2 \right)} \!\!\!
\int_{(\R^3)^M}  \!\!\!
\Big(
\prod_{e \in \CR_1} 
\phi_0^\lambda(e_\downarrow) 
| K(e) |    \Big) 
\!\!\!\!   \prod_{ \{i,j\} \in\CE (\CM)}   \!\!\!\!
\J_\eps(x_i - x_j)^{\sigma_i \sigma_j}
\,dx
\end{equ}
where the integration is still over $x \in (\R^3)^{2m} $,
but the second product is now over pairs of un-indexed charges in $\CM$.
The function $\Psi^{kl}_\eps$ satisfies the same bound with $\CR_1$ replaced by $\CR_2$.

\begin{remark}
From now on we write $\CR$ as a shorthand for either $\CR_1$ or $\CR_2$, depending
on whether we are considering the bound for $\Psi^{k\bar l}_\eps$ or for
$\Psi^{k l}_\eps$. 
\end{remark}

To see that \eqref{e:unindexed} holds, note that for every $i\in M$ with index $h\in\{k,l\}$, a new factor $\J_\eps(x_i-x_i)^{{1\over 2} h(h-1)}$ appeared in the integrand when compared to \eqref{e:indexed}. In fact, the factor in front of the integral in \eref{e:indexed}
is bounded by $\eps^{{\bar \beta \over 2} m (k^2+l^2-2)}$.
For each $i\in M$ with index $h\in\{k,l\}$, we associate to it a factor
$\eps^{{\bar\beta \over 2} h(h-1)}  \lesssim \J_\eps(x_i-x_i)^{{1\over 2} h(h-1)}$.
There are then a total of ${\bar\beta m \over 2} (k(k-1)+l(l-1))$ 
factors of $\eps$ that are turned into the new factors $\J_\eps(0)$
in this way. We are left with a power of $\eps$
which is precisely the factor in front of the integral in \eref{e:unindexed}.

The above product of $\J_\eps$'s now falls again into the setting of Section \ref{sec:linear}
with the ``potential" function $\J$ simply being $\J_\eps$,
except that the points indexed by $\CM$ are not all distinct. This does not matter
because one can just start for $n$ sufficiently small so that 
\begin{equ}
\CA_n = \bigl\{A_1,A_2,\ldots, A_{2m}  \bigr\}
\end{equ}
where each of $A_p$ contains $k$ or $l$ un-indexed charges with the same sign at the same location.
The bound 
$\Bigl| \prod_{i\neq j \in A} \J^{\sigma_i\sigma_j}(x_i-x_j)\Bigr| \le C  \bar \J_n^{D_n(A)}$ then holds trivially for each $A\in\CA_n$ defined above and we can start the recursive construction of 
Section~\ref{sec:linear} from there.
That procedure then provides a pairing $\CS_\ast$ for $\CM$ and, writing
\begin{equ}
\CI_\eps \eqdef \eps^{ {\bar \beta m\over 2} \left(k +l-2 \right)}  \!\!\!
\prod_{ \{i,j\} \in\CE (\CM)} \J_\eps(x_i - x_j)^{\sigma_i \sigma_j}
\end{equ}
as a shorthand, one has the bound
\begin{equ}[e:boundIeps]
\CI_\eps \lesssim 
\eps^{ {\bar \beta m\over 2} \left(k +l-2 \right)}   \!\!\!
\prod_{ \{i,j\} \in \CS_\ast} \J_\eps^-(x_i - x_j) \;.
\end{equ}
Note that on the right hand side, the total number of factors $\J_\eps^-$ is ${m\over 2} (k+l)$, and the total number of factors $\eps^{ \bar \beta}$  is $ {m\over 2} (k +l-2 )$. 
In the following, we will use the fact that $ \eps^{\bar \beta}\lesssim\J_\eps$
to ``cancel" some of the factors $\J_\eps^-$ with the factors $\eps^{\bar \beta}$. 
We remark that after such cancellations the number of 
factors $\J_\eps^-$ will always be more by $m$ than the number of 
factors $\eps^{\bar\beta}$.

%

Given the pairing $\CS_\ast$, one can associate to it a graph $G$ with vertex set $M$ 
and edges $E$  
in such way that $\{i,j\} \in E$ if and only if there 
exist $a \in \pi^{-1}(i)$ and $b \in \pi^{-1}(j)$ such that $\{a,b\}\in\CS_\ast$.
Of course, one then automatically has $\sigma_i \sigma_j = -1$, i.e.\ the vertices
correspond to charges with opposite signs.
Using the bound $\eps^{\bar \beta}\J_\eps^- \lesssim 1$, we immediately obtain from \eqref{e:boundIeps} 
the bound
\begin{equ}[e:boundIepsGraph]
\CI_\eps \lesssim 
\eps^{ \bar \beta (|E| - m)}   \!\!\! \prod_{ \{i,j\} \in E} \J_\eps^-(x_i - x_j) \;.
\end{equ}
Since $S_\star$ is a pairing of $\CM$, every vertex in $G$ has degree at least one,
so that in particular $|E| \ge m$, but $G$ is not necessarily connected. 


%
%

The set of edges $E$ interplays with the set $\CR$ in the following way.
In the case of $\Psi_\eps^{k\bar l}$, every element in  $\CR = \CR_1$ is a pair
of charges having opposite signs. On the other hand, for the case $\Psi_\eps^{kl}$,
every element in $\CR = \CR_2$ is a pair
of charges having the same sign.
In both cases, every edge in $E$ connects two points of opposite signs, 
therefore $E\cap \CR_2=\emptyset$, while $E\cap \CR_1$ may not be empty.

We now proceed to simplify the graph $G$ in such a way that the bound \eqref{e:boundIepsGraph}
still holds at each stage of the simplification. Since $\eps^{\bar \beta}\J_\eps^- \lesssim 1$,
we are allowed to simply erase edges, but we want to do this in such a way that there are at least
$m$ edges left at the end (so that the prefactor contains a positive power of $\eps$) and
so that the resulting graph is as ``simple'' as possible.
This simplification step is slightly different between the bound on 
$\Psi_\eps^{k\bar l}$ and that on $\Psi_\eps^{kl}$.

%

For the case $\Psi_\eps^{kl}$, let $\CF_G$ be a spanning forest of $G$.
For each connected component $\CT_G$ (which is a tree) of  $\CF_G$,
let $i$ be a leaf of $\CT_G$, and $j$ be the unique vertex connected to $i$.
We erase all edges of the form $\{j,k\}$ where $k$ is a vertex but not a leaf of $\CT_G$.
We obtain in this way a connected component which is a star (consisting of at least two points) centred at $j$. 
Iterating this procedure, we can reduce ourselves to the case where every connected component of $G$ is 
a star with at least two vertices. Denote the resulting graph by $G_1$.
Note that the condition that every vertex has degree at least one still holds for $G_1$,
so that there are indeed still at least $m$ edges left.

In the case $\Psi_\eps^{k\bar l}$, we encounter one additional difficulty:
since $E \cap \CR$ may be non-empty in this case, the procedure described above
may create a graph in which one of the connected components is given by a
single edge $e$ which also happens to belong to $\CR$. Going back to 
\eqref{e:unindexed}, this implies that the right hand side is bounded by 
a quantity that containing a factor
\begin{equ}
\int |K(x \to y)|\,\J_\eps^-(x - y) \phi_0^\lambda(y)\,dx\,dy\;.
\end{equ}
Unfortunately, this quantity diverges as $\eps \to 0$, so we should avoid
such a situation.
The key observation is that since $k\neq l$, 
there does not exist any connected component of the original graph $G$
consisting of only one edge in $\CR$, so we tweak the procedure  described
above in order to avoid creating one.

As before, we consider a spanning forest $\CF_G$ of $G$, and we
denote by $E(\CF_G)$ the set of edges of  $\CF_G$. This time, we furthermore 
let $G_1$ be the graph defined by contracting all the edges in $E(\CF_G)\cap\CR$.
More precisely, 
define an equivalence relation $\sim$ on $M$ by: $i\sim j$ if and only if
$\{i,j\}\in E(\CF_G)\cap \CR$. Obviously each equivalence class consists of either 
one or two points of $M$. For every $i\in M$, write $[i]$ for its equivalence 
class. The contracted graph $\bar \CF_G$ has the set of equivalence classes as its 
set of vertices. 
for  $[i']\neq [j'] \in \bar \CF_G$,
$\{[i'],[j']\}$ is an edge of $\bar \CF_G$ if and only if there exist $i\in [i']$, $j\in [j']$ and 
$\{i,j\}$ is an edge of $\CF_G$.
Self loops of the form $\{[i],[i]\}$ are not considered as edges of $\bar \CF_G$.
Note that $\bar \CF_G$ is necessarily a forest,
with every tree component consisting of at least $2$ points.

We then erase edges of the forest  $\bar \CF_G$
according to the same procedure as in the case  $\Psi_\eps^{k l}$ and denote by 
$\bar E_1\subset E(\bar \CF_G)$ the set of erased edges. 
This procedure turns $\bar \CF_G$ into a graph $\bar G_1$ consisting of disjoint stars,
with each star consisting of at least two points. 
Each edge $e \in \bar E_1$ has an obvious counterpart in $E(\CF_G)$,
and we denote by $G_1 \subset \CF_G$ the graph obtained from $\CF_G$ by erasing these.
(In particular, $\bar G_1$ is obtained from $G_1$ via the contraction given by $\sim$.)
This graph has the following properties:
\begin{claim}
\item $E(G_1)\cap \CR = E(\CF_G)\cap \CR$, where $E(G_1)$ is the set of edges of $G_1$.
\item Every connected component $T$ of $G_1$ is
a tree, and contracting edges in $E(T)\cap \CR$ turns $T$ into a star.
\end{claim}
The two pictures below illustrates two possible configurations of 
such a connected component $T$, where solid lines stand for 
edges in $E(T)$ and dashed lines stand for edges in $\CR$.
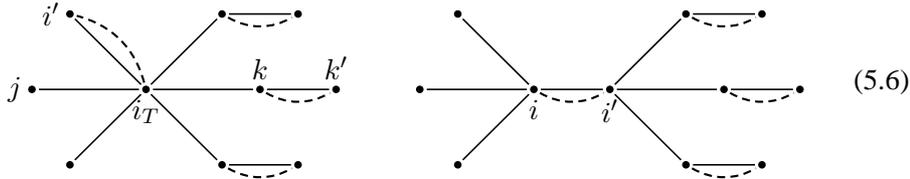
\begin{equ}[e:starFigure]
\begin{tikzpicture}[baseline=0]
\node [dot] at (0,0) (0) {}; \node  [below] at (0) {$i_T$};
\node [dot] at (-1,1) (a) {};  \node  [left] at (a) {$i'$};
\node [dot] at (-1.5,0) (b1) {};  \node  [left] at (b1) {$j$};
\node [dot] at (-1,-1) (b2) {};  
\node [dot] at (1,1) (c1) {};  
\node [dot] at (1.5,0) (c2) {};\node  [above] at (c2) {$k$};
\node [dot] at (1,-1) (c3) {};    
\node [dot] at (2,1) (cc1) {};
\node [dot] at (2.5,0) (cc2) {};\node  [above] at (cc2) {$k'$};
\node [dot] at (2,-1) (cc3) {};
\draw [kernel] (0) to (a);  \draw [bend right=30,KK] (0) to (a);
\draw [kernel] (0) to (b1);
\draw [kernel] (0) to (b2);
\draw [kernel] (0) to (c1);
\draw [kernel] (0) to (c2);
\draw [kernel] (0) to (c3);
\draw [kernel] (c1) to (cc1); \draw [bend right=30,KK] (c1) to (cc1);
\draw [kernel] (c2) to (cc2); \draw [bend right=30,KK] (c2) to (cc2);
\draw [kernel] (c3) to (cc3); \draw [bend right=30,KK] (c3) to (cc3);
\end{tikzpicture}
\qquad
\begin{tikzpicture}[baseline=0]
\node [dot] at (0,0) (0) {};  \node [dot] at (1,0) (1) {};
\node  [below] at (0) {$i$}; \node  [below] at (1) {$i'$};
\node [dot] at (-1,1) (a) {};  
\node [dot] at (-1.5,0) (b1) {};  
\node [dot] at (-1,-1) (b2) {};  
\node [dot] at (2,1) (c1) {};  
\node [dot] at (2.5,0) (c2) {};
\node [dot] at (2,-1) (c3) {};    
\node [dot] at (3,1) (cc1) {};
\node [dot] at (3.5,0) (cc2) {};
\node [dot] at (3,-1) (cc3) {};
\draw [kernel] (0) to (1);
\draw [bend right=30,KK] (0) to (1);
\draw [kernel] (0) to (a);  
\draw [kernel] (0) to (b1);
\draw [kernel] (0) to (b2);
\draw [kernel] (1) to (c1);
\draw [kernel] (1) to (c2);
\draw [kernel] (1) to (c3);
\draw [kernel] (c1) to (cc1); \draw [bend right=30,KK] (c1) to (cc1);
\draw [kernel] (c2) to (cc2); \draw [bend right=30,KK] (c2) to (cc2);
\draw [kernel] (c3) to (cc3); \draw [bend right=30,KK] (c3) to (cc3);
\end{tikzpicture}
\end{equ}
Every connected component $T \subset G_1$ correspondes to a connected component $\bar T$ of $\bar G_1$,
which is a star by construction. Denote by $[i]$ the centre of that star, choosing any of its 
two vertices if it only consists of one edge.
If $[i] = \{i,i'\}\in E(T)\cap \CR_1$, then at least one of $i$ and $i'$ necessarily 
have degree strictly larger than $1$ in $G_1$,
for otherwise $\bar T$ would consist of only one point.
If both have degree strictly larger than $1$ (as in the right hand figure above), 
then we erase the edge $\{i,i'\}$ and obtain two 
connected components, both consisting of stars having at least two points.
In this way, we can reduce ourselves to the case when either $[i] = \{i\}$,
or $[i] = \{i,i'\}$ and $i'$ is of degree $1$ in $G_1$.
In either case, we call $i$ the root of the connected component $T$ and we denote it
by $i_T$.

By construction, the root $i_T$ may connect to three types of edges:
\begin{claim}
\item an edge $\{i_T, i'\}\in \CR$ such that $i'$ has degree $1$ - call it an edge of type $i'$;
\item an edge $\{i_T, j\}\notin \CR$ such that $j$ has degree $1$ - call it an edge of type $j$;
\item an edge $\{i_T, k\}\notin \CR$ such that $k$ has degree $2$, and there exists $k'\in T$ such that $\{k,k'\}\in E(T)\cap \CR$ - call it an edge of type $k$.
\end{claim} 
See the left hand figure in \eqref{e:starFigure} for an example showing each type of edge.
Furthermore, it follows from the construction that $i_T$ is connected to at most one edge of type $i'$
and to at least one edge of type $j$ or $k$.
Lemma~\ref{lem:make-real-star} below then allows us to integrate out all edges of type $k$.
More precisely, if there exist edges of the type $j$ connected to $i_T$,
then we apply the first bound of Lemma~\ref{lem:make-real-star} 
to integrate over the variables corresponding to the vertices $k$ and $k'$ of
all the edges of type $k$ connected to $i_T$. After performing
such an integration, the bound \eqref{e:unindexed} still holds, but with $m$ lowered by $1$
and the graph $G_1$ replaced by the new graph where the vertices $k$ and $k'$,
as well as the edges $\{i_T,k\}$ and $\{k,k'\}$ have been erased.
Since the number of edges is reduced by $2$ and $m$ is lowered by $1$, we should indeed
``use'' one power of $\eps^{\bar \beta}$, as required by Lemma~\ref{lem:make-real-star}.
Note also that the bound $\lambda^{2-\bar\beta}$ appearing in the right hand side of
Lemma~\ref{lem:make-real-star} is consistent with the bound \eqref{e:boundkbark}
we are aiming for.

If on the other hand there is no edge of type $j$ connected to $i_T$,
then we integrate out all edges of type $k$ except for one. If 
there then still remains an edge of type $i'$ we apply the second bound of Lemma~\ref{lem:make-real-star} to integrate the entire connected component.
Again, this preserves the bound \eqref{e:unindexed} provided that we decrease $m$ by $2$ and
remove the entire connected components (now consisting of $3$ edges and $4$ vertices) from $G_1$.

\begin{lemma} \label{lem:make-real-star}
Let $K^{\varphi}(x,y)$ 
be a function that is given by either $\varphi^\lambda_0(y) K(x\to y)$ or  $\varphi^\lambda_0(x) K(y \to x)$.  Then,
\begin{equs}
\eps^{\bar\beta}  \int \J_\eps^{-}(x-y) \Big(|K^\varphi(y,z)|  \J_\eps^{-}(y-z)\Big) \,dydz 
	& \lesssim \lambda^{2-\bar\beta} \;,  \\
\eps^{\bar\beta}  \int \J_\eps^{-}(x-y) \Big(|K^\varphi(w,x)| \J_\eps^{-}(w-x)\Big)
	\Big( |K^\varphi(y,z)| & \J_\eps^{-}(y-z)\Big) 
	 \lesssim \lambda^{4-2\bar\beta} \;.
\end{equs}
where the second integral is over $x,y,z$ and $w$. Both bounds hold with proportionality constants that
are uniform over $\eps,\lambda \in (0,1]$, and the
first bound is furthermore uniform over $x \in \R^3$.
\end{lemma}
\begin{proof}
For the first bound, assume that $K^{\varphi}(y,z)=\varphi^\lambda_0(z) (K(y-z)-K(y))$. We bound the integral involving the two $K$ terms separately. For the term with $K(y-z)$, it suffices to bound $\eps^{\bar\beta} \J_\eps^-(y-z)\lesssim 1$, then integrate over $y$, and finally use the fact that $\int \|x-z\|_\s^{2-\bar\beta}\varphi^\lambda_0(z) \,dz\lesssim \lambda^{2-\bar\beta}$. The latter bound 
is obtained by discussing the two cases $\|x\|_\s<3\lambda$ and $\|x\|_\s\geq 3\lambda$ separately.
For the term with $K(y)$, bound $\eps^{\bar\beta} \J_\eps^-(x-y)\lesssim 1$, then integrate over $y$ and follow the same estimate as above. 
Assume on the other hand that $K^{\varphi}(y,z)=\varphi^\lambda_0(y) (K(z-y)-K(z))$.
For the term with $K(z-y)$, integrating over $z$ yields a factor $\eps^{2-\bar\beta}$ which,
when multiplied by $\eps^{\bar\beta}  \J_\eps^{-}(x-y)$, can be bounded 
by $\|x-y\|_\s^{2-\bar\beta}$. It then remains to integrate over $y$ in the same way as above. 
For the term with  $K(z)$, we bound $\eps^{\bar\beta}  \J_\eps^{-}(x-y)\lesssim 1$
and then integrate over $z$ and $y$ similarly as above.

For the second bound, 
integrating out $y,z$ in the same way as above would result in a non-integrable function
$\J_\eps^-K$ of $w - x$. Instead,
we first integrate out the point in $\{x,w\}$ at which $\varphi^\lambda_0$ is not evaluated
making use of a factor $\eps^{\bar\beta \over 2}$. Since 
the techniques are analogous with that used above we only give the result:
\begin{equ}
\eps^{\bar\beta \over 2} \int \varphi^\lambda_0(x') \Big(\|x'-y\|_\s^{2-{3\over 2}\bar\beta} + \|y\|_\s^{2-{3\over 2}\bar\beta} \Big)
	\Big( |K^\varphi(y,z)|  \J_\eps^{-}(y-z)\Big)  \,dx'\,dy\,dz \;,
\end{equ}
where $x'$ is the variable in  $\{w,x\}$ that is not integrated.
This integral can be bounded in the analogous way as the first bound above.
\end{proof}

In this way we obtain a graph $G_1$ such that \eref{e:boundIepsGraph} still holds, and 
such that every connected component of $G_1$
is a star, which is the same situation as in the case $\Psi_\eps^{kl}$.
For both cases of $\Psi_\eps^{kl}$ and $\Psi_\eps^{k\bar l}$, 
if one of these stars consists of more than $3$ points (i.e.\ has more than $2$ leaves), 
we can perform an additional simplification as follows. 
Denote by $j$ the centre of the star and by $X$ the set of its leaves. 
Among all the distances $\|x_i-x_j\|_\s$ for $i \in X$, 
let $k$ be such that $\|x_k - x_j\|_\s$ is the shortest one, 
and pick an $i \in X \setminus \{k\}$ such that $\{i,k\}\notin \CR$,
which is always possible since one has at least 
two distinct choices for $i$.
We then
use the bound  $\J_\eps^-(x_i - x_j) \lesssim \J_\eps^-(x_i-x_k)$ 
to change the edge $\{i,j\}$ into the edge $\{i,k\}$ and erase the edge $\{j,k\}$
without violating the bound \eqref{e:boundIepsGraph}.
Since in the case of $\Psi_\eps^{k\bar l}$,
$i$ and $k$ necessarily have the same sign, the newly formed edge is
such that  $\{i,k\}\notin \CR$.

The following picture shows an example of this operation,
where each solid line stands for an edge in the star, i.e. a factor $\J_\eps^-$.
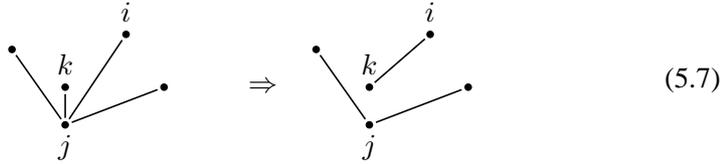
\begin{equ}[e:simplifyStar]
\begin{tikzpicture}[baseline=0]
\node at (1,-0.5) [dot] (0) {};   \node at (1,-0.8) {$j$};
\node at (0.3,0.5) [dot] (1) {};
\node at (1,0) [dot] (2) {};   \node at (1,0.3) {$k$};
\node at (1.8,0.7) [dot] (3) {};   \node at (1.8,1) {$i$};
\node at (2.3,0) [dot] (4) {};
\draw [kernel] (0) to (1);
\draw [kernel] (0) to (2);
\draw [kernel] (0) to (3);
\draw [kernel] (0) to (4);
\node at (3.6,0) {$\Rightarrow$};
\node at (5,-0.5) [dot] (0) {};   \node at (5,-0.8) {$j$};
\node at (4.3,0.5) [dot] (1) {};
\node at (5,0) [dot] (2) {};   \node at (5,0.3) {$k$};
\node at (5.8,0.7) [dot] (3) {};   \node at (5.8,1) {$i$};
\node at (6.3,0) [dot] (4) {};
\draw [kernel] (0) to (1);
\draw [kernel] (3) to (2);
\draw [kernel] (0) to (4);
\end{tikzpicture}
\end{equ}
Repeating this operation, we can reduce each star 
to disconnected components, where each component has
either two or three vertices.
Again, the condition that every vertex has degree at least one still holds, so that 
there are still at least $m$ edges left.

Summarising this discussion, we have just demonstrated that one can always build a graph $G_\star$ 
 consisting of disconnected components, where each component
is a star having either two or three vertices, and such that \eqref{e:boundIepsGraph} holds.
Furthermore, $G_\star$ can be chosen in such a way that 
its edges $E_\star = E(G_\star)$ satisfy $E_\star \cap \CR = \emptyset$.
In order to deal with the components with three vertices, we define the function
$T_\eps(x,y;z) \eqdef \J_\eps^-(x-z) \J_\eps^-(y-z)$.
Let $\tau$ be the total number of appearances of the factor $T_\eps$ in \eqref{e:boundIepsGraph}, i.e.\
the number of connected components of the type \tikz[scale=0.3,baseline=0.7] \draw [kernel] (0,0.7) node[dot] {} -- (1.5,0) node[dot] {} -- (3,0.7) node[dot] {}; in $G_\star$. 
Note that $\tau$ is necessarily an even number since the total number of charges is even
and each such component involves $3$ of them.

The total number of edges in the graph $G_\star$ is 
equal to $m + {\tau \over 2}$, so that the prefactor appearing in \eqref{e:boundIepsGraph} is given by
$\eps^{\bar \beta \tau / 2}$. In other words, \eqref{e:boundIepsGraph} contains exactly one factor $\J_\eps^-$ for each
connected component of the type \tikz[scale=0.3,baseline=-3] \draw [kernel] (0,0) node[dot] {} -- (2,0) node[dot] {}; and one factor $\eps^{\bar \beta / 2}T_\eps$ for
each connected component of the type \tikz[scale=0.3,baseline=0.7] \draw [kernel] (0,0.7) node[dot] {} -- (1.5,0) node[dot] {} -- (3,0.7) node[dot] {};.

We now return to the task of actually estimating the full right hand side of \eqref{e:unindexed}.
We can depict this by also drawing a dashed arrow \tikz[scale=0.3,baseline=-3] \draw [KK,->] (0,0) node[dot] {} -- (3,0) node[dot] {}; for every occurrence of $K(e)$, i.e.\ for every edge in $\CR$.
Consider now the graph $\hat G$ whose edges are the union of the ``plain'' edges in $G_\star$ and 
the ``dashed'' edges in $\CR$. 
If $\tau = 0$, then the topology of $\hat G$ is very simple: since the edges of $G_\star$ are disjoint from those of $\CR$ and since both sets of edges form a pairing of the vertex set $M$,
it simply consists of a finite number of cycles which 
alternate between plain and dashed edges, so that 
we are in the situation of \eref{e:int-loops} of Lemma~\ref{lem:int-L} with all the $\alpha_i$
given by $-\bar \beta \in(-4,-2]$. Furthermore, each of these cycles involves at least $4$ 
vertices as required by the definition of $F_\CL$ there, so that the assumptions of Lemma~\ref{lem:int-L} are satisfied
and do yield the required bound.

We therefore now consider all the possible ways in which the factors $T_\eps(x,y;z)$ can 
interplay with the kernels $K$ in the graph $\hat G$. The presence of these factors can either create
connections between cycles or it can terminate them and create ``ends''.
For $i=1,...,6$, denote by $V_i=V_i(x,y,z,\bar x,\bar y,\bar z)$ the 
following functions describing all possible ways of creating a connection, 
where a plain line connecting two variables $x$ and $y$
denotes a factor $(\|y-x\|_\s+\eps)^{-\bar\beta}$, which is an upper bound for $\J_\eps^-(x-y)$, 
and a dashed arrow connecting $x$ to $y$ denotes
a factor $|K(x\to y)|$. We ignore the presence of the test functions $\phi_0^\lambda$ in 
\eqref{e:unindexed} at this stage, but we will restrict ourselves to the situation where the corresponding variables
(i.e.\ the variables located at the tip of a dashed arrow) are of parabolic norm less than $\lambda$.
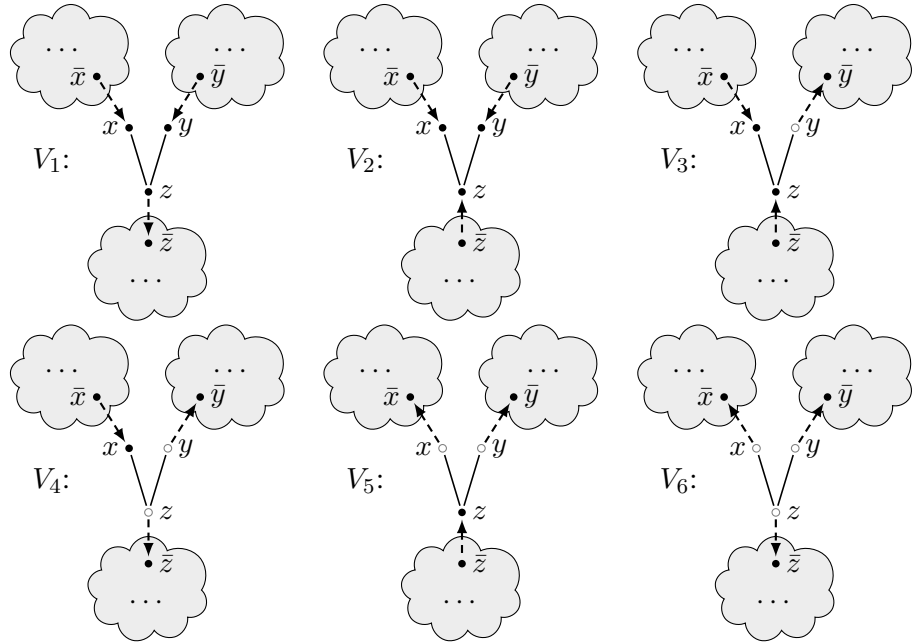
\begin{center}\def\scl{0.85}
\begin{tikzpicture}  [scale=\scl]
\node at (-1.5,0.5) {$V_1$:};
\node[cloud, cloud puffs=8.7, cloud ignores aspect, minimum width=1.6cm, minimum height=1.4cm,  draw, fill=black!7]  at (0,-1.2) {};
\node[cloud, cloud puffs=8.7, cloud ignores aspect, minimum width=1.6cm, minimum height=1.4cm,  draw, fill=black!7]  at (-1.2,2.1) {};
\node[cloud, cloud puffs=8.7, cloud ignores aspect, minimum width=1.6cm, minimum height=1.4cm,  draw, fill=black!7] at (1.2,2.1) {};
\node at (0,0) [dot] (0) {};  \node [right] at (0) {$z$};
\node at (-0.3,1) [dot] (1) {}; \node [left] at (1) {$x$};
\node at (-0.8,1.8) [dot] (2) {}; \node [left] at (2) {$\bar x$};
\node at (0.3,1) [dot] (3) {}; \node [right] at (3) {$y$};
\node at (0.8,1.8) [dot] (4) {}; \node [right] at (4) {$\bar y$};
\node at (0,-0.8) [dot] (5) {}; \node [right] at (5) {$\bar z$};
\draw [KK,->] (2) to (1);
\draw [KK,->] (4) to (3);
\draw [kernel] (0) to (1);
\draw [kernel] (0) to (3);
\draw [KK,->] (0) to (5);
\node  at (0,-1.4) {$\cdots$};
\node at (-1.3,2.2) {$\cdots$};
\node at (1.3,2.2) {$\cdots$};
\end{tikzpicture}
\quad
\begin{tikzpicture}  [scale=\scl]
\node at (-1.5,0.5) {$V_2$:};
\node[cloud, cloud puffs=8.7, cloud ignores aspect, minimum width=1.6cm, minimum height=1.4cm,  draw, fill=black!7] at (0,-1.2) {};
\node[cloud, cloud puffs=8.7, cloud ignores aspect, minimum width=1.6cm, minimum height=1.4cm,  draw, fill=black!7]  at (-1.2,2.1) {};
\node[cloud, cloud puffs=8.7, cloud ignores aspect, minimum width=1.6cm, minimum height=1.4cm,  draw, fill=black!7] at (1.2,2.1) {};
\node at (0,0) [dot] (0) {};  \node [right] at (0) {$z$};
\node at (-0.3,1) [dot] (1) {}; \node [left] at (1) {$x$};
\node at (-0.8,1.8) [dot] (2) {}; \node [left] at (2) {$\bar x$};
\node at (0.3,1) [dot] (3) {}; \node [right] at (3) {$y$};
\node at (0.8,1.8) [dot] (4) {}; \node [right] at (4) {$\bar y$};
\node at (0,-0.8) [dot] (5) {}; \node [right] at (5) {$\bar z$};
\draw [KK,->] (2) to (1);
\draw [KK,->] (4) to (3);
\draw [kernel] (0) to (1);
\draw [kernel] (0) to (3);
\draw [KK,<-] (0) to (5);
\node  at (0,-1.4) {$\cdots$};
\node at (-1.3,2.2) {$\cdots$};
\node at (1.3,2.2) {$\cdots$};
\end{tikzpicture}
\quad
\begin{tikzpicture} [scale=\scl]
\node at (-1.5,0.5) {$V_3$:};
\node[cloud, cloud puffs=8.7, cloud ignores aspect, minimum width=1.6cm, minimum height=1.4cm,  draw, fill=black!7]  at (0,-1.2) {};
\node[cloud, cloud puffs=8.7, cloud ignores aspect, minimum width=1.6cm, minimum height=1.4cm,  draw, fill=black!7]  at (-1.2,2.1) {};
\node[cloud, cloud puffs=8.7, cloud ignores aspect, minimum width=1.6cm, minimum height=1.4cm, align=center, draw, fill=black!7]  at (1.2,2.1) {};
\node at (0,0) [dot] (0) {}; \node [right] at (0) {$z$};
\node at (-0.3,1) [dot] (1) {}; \node [left] at (1) {$x$};
\node at (-0.8,1.8) [dot] (2) {}; \node [left] at (2) {$\bar x$};
\node at (0.3,1) [graydot] (3) {};\node [right] at (3) {$y$};
\node at (0.8,1.8) [dot] (4) {};\node [right] at (4) {$\bar y$};
\node at (0,-0.8) [dot] (5) {};\node [right] at (5) {$\bar z$};
\draw [KK,->] (2) to (1);
\draw [KK,<-] (4) to (3);
\draw [kernel] (0) to (1);
\draw [kernel] (0) to (3);
\draw [KK,<-] (0) to (5);
\node  at (0,-1.4) {$\cdots$};
\node at (-1.3,2.2) {$\cdots$};
\node at (1.3,2.2) {$\cdots$};
\end{tikzpicture}\\
\begin{tikzpicture} [scale=\scl]
\node at (-1.5,0.5) {$V_4$:};
\node[cloud, cloud puffs=8.7, cloud ignores aspect, minimum width=1.6cm, minimum height=1.4cm,  draw, fill=black!7] (cloud) at (0,-1.2) {};
\node[cloud, cloud puffs=8.7, cloud ignores aspect, minimum width=1.6cm, minimum height=1.4cm,  draw, fill=black!7] (cloud) at (-1.2,2.1) {};
\node[cloud, cloud puffs=8.7, cloud ignores aspect, minimum width=1.6cm, minimum height=1.4cm, draw, fill=black!7] (cloud) at (1.2,2.1) {};
\node at (0,0) [graydot] (0) {}; \node [right] at (0) {$z$};
\node at (-0.3,1) [dot] (1) {}; \node [left] at (1) {$x$};
\node at (-0.8,1.8) [dot] (2) {}; \node [left] at (2) {$\bar x$};
\node at (0.3,1) [graydot] (3) {};\node [right] at (3) {$y$};
\node at (0.8,1.8) [dot] (4) {};\node [right] at (4) {$\bar y$};
\node at (0,-0.8) [dot] (5) {};\node [right] at (5) {$\bar z$};
\draw [KK,->] (2) to (1);
\draw [KK,<-] (4) to (3);
\draw [kernel] (0) to (1);
\draw [kernel] (0) to (3);
\draw [KK,->] (0) to (5);
\node  at (0,-1.4) {$\cdots$};
\node at (-1.3,2.2) {$\cdots$};
\node at (1.3,2.2) {$\cdots$};
\end{tikzpicture}
\quad
\begin{tikzpicture} [scale=\scl]
\node at (-1.5,0.5) {$V_5$:};
\node[cloud, cloud puffs=8.7, cloud ignores aspect, minimum width=1.6cm, minimum height=1.4cm,  draw, fill=black!7] (cloud) at (0,-1.2) {};
\node[cloud, cloud puffs=8.7, cloud ignores aspect, minimum width=1.6cm, minimum height=1.4cm,  draw, fill=black!7] (cloud) at (-1.2,2.1) {};
\node[cloud, cloud puffs=8.7, cloud ignores aspect, minimum width=1.6cm, minimum height=1.4cm, draw, fill=black!7] (cloud) at (1.2,2.1) {};
\node at (0,0) [dot] (0) {}; \node [right] at (0) {$z$};
\node at (-0.3,1) [graydot] (1) {}; \node [left] at (1) {$x$};
\node at (-0.8,1.8) [dot] (2) {}; \node [left] at (2) {$\bar x$};
\node at (0.3,1) [graydot] (3) {};\node [right] at (3) {$y$};
\node at (0.8,1.8) [dot] (4) {};\node [right] at (4) {$\bar y$};
\node at (0,-0.8) [dot] (5) {};\node [right] at (5) {$\bar z$};
\draw [KK,<-] (2) to (1);
\draw [KK,<-] (4) to (3);
\draw [kernel] (0) to (1);
\draw [kernel] (0) to (3);
\draw [KK,<-] (0) to (5);
\node  at (0,-1.4) {$\cdots$};
\node at (-1.3,2.2) {$\cdots$};
\node at (1.3,2.2) {$\cdots$};
\end{tikzpicture}
\quad
\begin{tikzpicture} [scale=\scl]
\node at (-1.5,0.5) {$V_6$:};
\node[cloud, cloud puffs=8.7, cloud ignores aspect, minimum width=1.6cm, minimum height=1.4cm,  draw, fill=black!7] (cloud) at (0,-1.2) {};
\node[cloud, cloud puffs=8.7, cloud ignores aspect, minimum width=1.6cm, minimum height=1.4cm, draw, fill=black!7] (cloud) at (-1.2,2.1) {};
\node[cloud, cloud puffs=8.7, cloud ignores aspect, minimum width=1.6cm, minimum height=1.4cm, draw, fill=black!7] (cloud) at (1.2,2.1) {};
\node at (0,0) [graydot] (0) {}; \node [right] at (0) {$z$};
\node at (-0.3,1) [graydot] (1) {}; \node [left] at (1) {$x$};
\node at (-0.8,1.8) [dot] (2) {}; \node [left] at (2) {$\bar x$};
\node at (0.3,1) [graydot] (3) {};\node [right] at (3) {$y$};
\node at (0.8,1.8) [dot] (4) {};\node [right] at (4) {$\bar y$};
\node at (0,-0.8) [dot] (5) {};\node [right] at (5) {$\bar z$};
\draw [KK,<-] (2) to (1);
\draw [KK,<-] (4) to (3);
\draw [kernel] (0) to (1);
\draw [kernel] (0) to (3);
\draw [KK,->] (0) to (5);
\node  at (0,-1.4) {$\cdots$};
\node at (-1.3,2.2) {$\cdots$};
\node at (1.3,2.2) {$\cdots$};
\end{tikzpicture}
\end{center}
For example, one has 
\begin{equ}
V_1
= |K(\bar x\to x)K(\bar y \to y)K(z \to \bar z)|\, (\|z-x\|_\s+\eps)^{-\bar\beta} (\|z-y\|_\s+\eps)^{-\bar\beta}\;,
\end{equ}
and similarly for the other $V_i$.
Using the same graphical notation, the different possible ways of creating an ``ending'' 
are described by the following functions $E_i$ with $i \in \{1,2,3,4\}$:
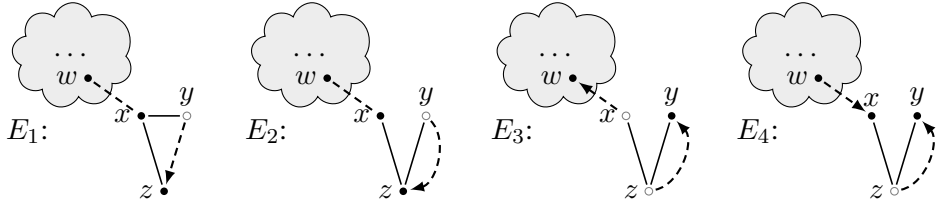
\begin{center}
\begin{tikzpicture} 
\node at (-1.8,0.8) {$E_1$:};
\node[cloud, cloud puffs=8.7, cloud ignores aspect, minimum width=1.6cm, minimum height=1.4cm,  draw, fill=black!7] at (-1.2,1.8) {$\cdots$};
\node at (0,0) [dot] (0) {}; \node [left] at (0) {$z$};
\node at (-0.3,1) [dot] (1) {}; \node [left] at (1) {$x$};
\node at (0.3,1) [graydot] (3) {};\node [above] at (3) {$y$};
\node at (-1,1.5) [dot] (w) {}; \node [left] at (w) {$w$};
\draw [KK] (w) to (1);
\draw [KK,->] (3) to (0);
\draw [kernel] (0) to (1);
\draw [kernel] (1) to (3);
\end{tikzpicture}
$\;\;$
\begin{tikzpicture} 
\node at (-1.8,0.8) {$E_2$:};
\node[cloud, cloud puffs=8.7, cloud ignores aspect, minimum width=1.6cm, minimum height=1.4cm,  draw, fill=black!7] at (-1.2,1.8) {$\cdots$};
\node at (0,0) [dot] (0) {}; \node [left] at (0) {$z$};
\node at (-0.3,1) [dot] (1) {}; \node [left] at (1) {$x$};
\node at (0.3,1) [graydot] (3) {};\node [above] at (3) {$y$};
\node at (-1,1.5) [dot] (w) {}; \node [left] at (w) {$w$};
\draw [KK] (w) to (1);
\draw [KK,<-,bend right=60] (0) to (3);
\draw [kernel] (0) to (1);
\draw [kernel] (0) to (3);
\end{tikzpicture}
$\;\;$
\begin{tikzpicture} 
\node at (-1.8,0.8) {$E_3$:};
\node[cloud, cloud puffs=8.7, cloud ignores aspect, minimum width=1.6cm, minimum height=1.4cm,  draw, fill=black!7] at (-1.2,1.8) {$\cdots$};
\node at (0,0) [graydot] (0) {}; \node [left] at (0) {$z$};
\node at (-0.3,1) [graydot] (1) {}; \node [left] at (1) {$x$};
\node at (0.3,1) [dot] (3) {};\node [above] at (3) {$y$};
\node at (-1,1.5) [dot] (w) {}; \node [left] at (w) {$w$};
\draw [KK,<-] (w) to (1);
\draw [KK,->,bend right=60] (0) to (3);
\draw [kernel] (0) to (1);
\draw [kernel] (0) to (3);
\end{tikzpicture}
$\;\;$
\begin{tikzpicture} 
\node at (-1.8,0.8) {$E_4$:};
\node[cloud, cloud puffs=8.7, cloud ignores aspect, minimum width=1.6cm, minimum height=1.4cm,  draw, fill=black!7] at (-1.2,1.8) {$\cdots$};
\node at (0,0) [graydot] (0) {}; \node [left] at (0) {$z$};
\node at (-0.3,1) [dot] (1) {}; \node [above] at (1) {$x$};
\node at (0.3,1) [dot] (3) {};\node [above] at (3) {$y$};
\node at (-1,1.5) [dot] (w) {}; \node [left] at (w) {$w$};
\draw [KK,->] (w) to (1);
\draw [KK,->,bend right=60] (0) to (3);
\draw [kernel] (0) to (1);
\draw [kernel] (0) to (3);
\end{tikzpicture}
\end{center}
These are viewed as functions of $x,y,z$ and $w$. Note that $E_3$ and $E_4$ only
differ by the direction of an arrow between $x$ and $w$, while for $E_1$ and $E_2$ the direction
of that arrow is not important when bounding them.

In order to bound the contributions coming from these factors, we integrate them
over those variables that are depicted by a circle (as opposed to a black dot) in the 
above pictorial representations. Note that these integration variables are never located
at the tip of a dashed arrow, so we do not need to take into account
the presence of the test functions $\phi_0^\lambda$ when we integrate them out.
We further introduce the notation $V^{(x)}$ as a shorthand for the function 
$\int V_1(x,y,z,\bar x,\bar y,\bar z)\,dx$, where we integrated out the $x$ variable,
$V^{(x,y)}$ for the function obtained from $V$ by integrating out both the $x$ and the
$y$ variable, etc.

With this notation, we then have the following bounds:
\begin{lemma}\label{lem:boundVE}
Let $V_i$ and $E_i$ be defined as above and assume that the variables located at the tip
of a dashed arrow are bounded by $\lambda$.
Then, for $\bar\beta\in[2,8/3)$, one has the bound
\begin{equ}
\eps^{{\bar \beta \over 2}} V_1
 \lesssim |K(\bar x \to x)| \, |K(\bar y \to y)| \, \Big(
	\|x-z\|_\s^{-{3\over 2}\bar \beta} + \|y-z\|_\s^{-{3\over 2}\bar \beta} \Big) \,
	|K(z \to \bar z)| \;,
\end{equ}
uniformly over $\eps,\lambda \in (0,1]$,
and $\eps^{{\bar \beta \over 2}} V_2$ is bounded by the same expression with $K(z \to \bar z)$ replaced by $K(\bar z \to  z)$. We furthermore have the bounds
\begin{equs} [2]
\multicol{4}{
\eps^{{\bar \beta \over 2}}   V_3^{(y)}
	\lesssim \lambda^{2-{\bar\beta\over 2}} \, |K(\bar x \to x)| \, 
	 	\|x-z\|_\s^{-\bar \beta} \, 
		|K(\bar z \to z)|  \;, 
} \\
\eps^{{\bar \beta \over 2}}  V_4^{(y,z)} \vee  \eps^{{\bar \beta \over 2}}  V_5^{(y,z)}
 	&\lesssim \lambda^{4-{3\over 2}\bar\beta} \, |K(\bar x \to x)|  \;,  
&\qquad
\eps^{{\bar \beta \over 2}}  V_6^{(x,y,z)} 
	&\lesssim \lambda^{6-{3\over 2}\bar\beta} \;, 
\qquad
\\
\eps^{{\bar \beta \over 2}}  E_1^{(y)} \vee \eps^{{\bar \beta \over 2}}  E_2^{(y)}
	&\lesssim \lambda^{2-\frac{\bar\beta}{2}}
	\|x-z\|_\s^{-\bar\beta} |K(x,w)|\;,
&\qquad
\eps^{{\bar \beta \over 2}}  E_3^{(x,z)} 
	&\lesssim \lambda^{4-{3\over 2}\bar\beta} \;, 
\qquad
\\
\multicol{4}{
 \eps^{{\bar \beta \over 2}}  E_4^{(z)}
	\lesssim \lambda^{2-{\bar\beta\over 2}}
	\Big( \|y-x\|_\s^{-\bar\beta} + \|y\|_\s^{-\bar\beta} \Big) |K(w\to x)| \;.
}
\end{equs}
\end{lemma}
\begin{proof}
In the proof \cite[Lemma~10.14]{Regularity} will be repeatedly applied without 
explicitly mentioning it every time. 
The bound for $V_1$ is then obtained using \eref{e:boundz1} 
and the uniform bound 
\begin{equ}[e:boundJepshalf]
\eps^{{\bar\beta/ 2}} (\|x\|_\s+\eps)^{-\bar \beta}\lesssim 
\|x\|_\s^{-{\bar\beta/ 2}}\;.
\end{equ}
The bound for $V_2$ follows in the same way.

The bound for $V_3$ is obtained by 
using again \eqref{e:boundJepshalf}, then integrating over $y$:
\begin{equ}
\int \|y-z\|_\s^{-{\bar\beta\over 2}} |K(\bar y -y)-K(-y)|\,dy 
\lesssim \|\bar y\|_\s^{2-{\bar\beta\over 2}}
\lesssim \lambda^{2-{\bar\beta\over 2}}\;,
\end{equ}
where the last inequality  uses the fact that $\bar y$ is a variable located at the tip
of a dashed arrow (and therefore eventually arising as an argument of $\phi^\lambda_0$),
so that we assumed $\|\bar y\|_\s \le \lambda$. 
Regarding the bound for $V_4$, 
integrating over $y$ results in a factor 
$\|\bar y - z\|_\s^{2-\bar\beta} + \| z\|_\s^{2-\bar\beta} $.
Then one uses again \eqref{e:boundJepshalf} and
applies \eref{e:boundz1}, integrates over $z$,
and finally observes that 
$4-{3\over 2}\bar\beta >0$ and $x, \bar y,\bar z$ 
all have norms bounded by $\lambda$ by assumption.

To obtain the bound for $V_5$, 
one uses the bound 
\begin{equ}[e:boundbeta4]
\eps^{{\bar\beta / 4}}(\|x-z\|_\s+\eps)^{-\bar \beta} \lesssim
\|x-z\|_\s^{-3\bar \beta /4}\;,
\end{equ}
integrates out $x$ using the fact
that the function $K$ arises as a difference, 
and $2-{3\over 4}\bar\beta \in (0,1)$,
and finally observe that $\bar x,\bar y,z$ are all within a 
distance $\lambda$ from the origin.
Then one treats $y$ in the same way as $x$.

To obtain the bound for $V_6$,
one uses again \eqref{e:boundbeta4},
then use gradient theorem for $K(x\to\bar x)$
to obtain a factor $\|\bar x\|_\s \lesssim \lambda$
times a function of $x$ and $\bar x$ of 
homogeneity $-3$. Then integrate out $x$
and obtain a function of homogeneity
$1-{3\over 4}\bar\beta <0$.  
One then treats $y$ in the same way as $x$,
and apply \eref{e:boundz1} to get 
a function of homogeneity
$2-{3\over 2}\bar\beta <0$.
Finally, we integrate out $z$ using that $K$ arises as a difference,
and obtain a power $4-{3\over 2}\bar\beta \in(0,1]$ of $\lambda$.

To obtain the bound for $E_1$, we simply note that 
if we set 
$\bar y = x$ in $V_3$ divided by $K(\bar x \to x)$, then the resulting function
is equal to $E_1$ after an obvious relabeling of variables,
except that $\bar z\to z$ in $V_3$
while the corresponding arrow between $x,w$ in $E_1$ can be pointing to either direction.
Since the bound we obtained on $V_3^{(y)}$ is independent
of $\bar y$ and its proof did not use
the fact that $z$ is the tip of an arrow, it immediately implies the required bound on $E_1$. 

The bound for $E_2$ can be obtained analogously
as that for $E_1$ by taking $\bar y=z$ now in $V_3$,
and noting that the proof of the bound for $V_3$
did not use the fact that $x$ is the tip of an arrow.
The bound for $E_3$ can be shown by setting $\bar z=x$ in $V_4$.


Regarding the  bound for $E_4$,
note that one has 
\begin{equ}
|K(y -z)-K(-z)| \lesssim \|y\|_\s^{2-{\bar\beta\over 2}-\delta}
\Big(\|y-z\|_\s^{{\bar\beta\over 2}-4+\delta}+\|z\|_\s^{{\bar\beta\over 2}-4+\delta}\Big)
\end{equ}
for any small $\delta>0$.
The integration over $z$ involving the first term above
is performed by applying \eqref{e:boundJepshalf} to $y-z$ to get a factor $\|y-z\|_\s^{-{\bar\beta/ 2}}$, followed by a convolution.
Regarding the second term above, apply \eqref{e:boundJepshalf} to $x-z$ to get a factor $\|x-z\|_\s^{-{\bar\beta/ 2}}$,
then bound $\|x-z\|_\s^{-{\bar\beta/ 2}} \|z\|_\s^{{\bar\beta\over 2}-4+\delta}
 \lesssim \|x-z\|_\s^{-4+\delta} + \|z\|_\s^{-4+\delta} $, and finally integrate over $z$.
Noting that $\|y\|_\s \lesssim \lambda$,
we obtain the desired bound.
%
%
\end{proof}

\begin{remark} \label{rem:homo-VE}
The bounds obtained in Lemma~\ref{lem:boundVE} all preserve the natural homogeneities associated
to each of the expressions appearing there in the following way. 
The natural homogeneity associated to $\eps^{\bar \beta/2}V_i$ is $-6-3\bar \beta/2$, since
$K$ has a singularity of order $-2$ at the origin. Furthermore, the scaling dimension of
parabolic space-time is $4$, so that each integration should increase the homogeneity by $4$.
For example, $\eps^{\bar \beta/2}V_4^{(y,z)}$ then has natural homogeneity $2-3\bar \beta/2$,
which is also the case for $\lambda^{4-3\bar \beta/2} K$.
\end{remark}

For each of the $V_i$, the bound of Lemma~\ref{lem:boundVE} greatly simplifies the dependency structure of
the resulting integrand. In the case of $V_1$--$V_3$, the ``triple junction'' is replaced by 
a ``double junction'' and an ``endpoint'', while it is replaced by three ``endpoints''
in the case of $V_4$--$V_6$. 
After applying the bounds of Lemma~\ref{lem:boundVE} to each occurrence of $T_\eps$,
one may obtain ``singletons'', i.e., a factor of the type $\int\varphi^\lambda_0(x)\,dx$.
This happens for instance in the  situation where the oriented edge $x\to \bar x$ of one instance of $V_6$ is the same 
as the oriented edge $\bar x \to x$ of one instance of $V_4$.
In this case, the bounds in Lemma~\ref{lem:boundVE}
yield a factor $\int\varphi^\lambda_0(y)\,dy$, where $y$ is the integration variable depicted by the vertex
located at the end of that oriented edge.
Since $\varphi^\lambda_0$ integrates to a constant independent of $\lambda$, such ``singletons'' can simply be discarded.

As a consequence, we are left with only cycles and chains
consisting of functions with known homogeneities and $\varphi^\lambda_0\cdot K$'s in an 
alternative way (one $\varphi^\lambda_0(x)K(y\to x)$ followed by such a homogeneous function, 
then followed by another $\varphi^\lambda_0\cdot K$ etc.)
Here, a function in a cycle, or in a chain but not at the two ends of the chain, can be one 
of the following three functions
\begin{equ}
\varphi^\lambda_0(x)K(y\to x),  \qquad
\|x-y\|_\s^{-\bar\beta},  \qquad
\|x-y\|_\s^{-{3\over 2}\bar\beta}. 
\end{equ}
A function at an end of a chain can be one of the two functions
\begin{equ}
\varphi^\lambda_0(x)\,K(y\to x), \qquad
 \varphi^\lambda_0(x) \, \|x-y\|_\s^{-\bar \beta},
\end{equ}
where, in both cases, the variable $x$ is the one that terminates the chain.
Note that the bound for $E_4$ gives a term
$\lambda^{2-{\bar\beta\over 2}}
	 \|y\|_\s^{-\bar\beta}  |K(w\to x)|$,
and since there is a test function $\varphi^\lambda_0(y)$,
we simply integrate over $y$ and obtain an end of chain of the first type.

Let us recapitulate now the situation so far. Recall that our aim is to prove 
that the bound \eref{e:boundkbark} holds, where the right hand side
is given by $\lambda$ to the power $(2-\bar\beta)\,m$. The left hand side on the other hand
is given by \eref{e:unindexed}, which is also naturally associated with the homogeneity
$(2-\bar\beta)\,m$, 
provided that one associates homogeneity $1$ to each power of $\eps$,
homogeneity $-2$ to each factor $K$, 
$4$ to each space-time integration variable,
$-4$ to each factor $\varphi^\lambda_0$,
and $\bar\beta$ to each factor $\J_\eps$, noting that 
the total homogeneity contributed from the product of factors $\J_\eps$
is $-{\bar \beta m\over 2} \left(k +l\right)$.

All of the subsequent simplifications (applying the procedure in Section~\ref{sec:linear};
applying the bound $\eps^{\bar \beta}\J_\eps^- \lesssim 1$, applying 
Lemma~\ref{lem:make-real-star}, applying the bound \eqref{e:simplifyStar}, and finally
applying Lemma~\ref{lem:boundVE})
retain this homogeneity.
At this point, as a consequence of the right hand sides of the bounds appearing
in Lemma~\ref{lem:boundVE}, our bound does not contain any factor $\eps$ anymore.
Summarising, the right hand side of 
\eref{e:unindexed} is bounded by a sum such that each summand 
is of the type $\lambda^\gamma$ (for some $\gamma \ge 0$), multiplied by 
a product of terms that have precisely the form
of the left hand sides of \eqref{e:int-loops} and / or \eqref{e:int-chains}.
The sum of the natural homogeneities (counted in the same way as above) 
of these factors is precisely equal to $\lambda^{(2-\bar\beta)m-\gamma}$, so that
the claim follows if we can guarantee that the assumptions of 
Lemma~\ref{lem:int-L} are satisfied for each factor.
This is because the powers $\mathfrak h(\CL)$, $\mathfrak h(\CC)$ of $\lambda$ 
appearing in Lemma~\ref{lem:int-L} are indeed equal to the 
natural homogeneities associated with the corresponding integrals counted in the same way as above.

Since we are considering the regime $\bar \beta \in [2,{8\over 3})$, we have in particular
$2 \le \bar \beta < {3\bar \beta \over 2} < 4$, which shows that the exponents $\alpha_i$
appearing in the formulation of Lemma~\ref{lem:int-L} do indeed belong to $(-4,-2]$ as required.
Also, each cycle resulting from the formation of ``double junction"
after applying Lemma~\ref{lem:boundVE}
to $V_1,V_2,V_3$ obviously has at least $4$ points, so that
the assumptions of Lemma~\ref{lem:int-L} are indeed satisfied. 
This immediately yields the required bound \eref{e:boundkbark}
with $\kappa = 0$.
To conclude, we note that just as in the bound for $\Psi_\eps^{k\bar k}$
for $k>1$,
one can gain a factor $\eps^\delta$ for a sufficiently small $\delta>0$
by ``pretending'' that the homogeneity of $\J_\eps$ is slightly worse than
what it really is,
so that the required moments of $\Psi_\eps^{k\bar l}$ and $\Psi_\eps^{kl}$ 
actually converge to zero as $\eps\to 0$. The same argument also covers the 
borderline case $\bar \beta = 2$.
\end{proof}

\bibliographystyle{./Martin}
\bibliography{./refs}

\def\cprime{$'$} \def\polhk#1{\setbox0=\hbox{#1}{\ooalign{\hidewidth
  \lower1.5ex\hbox{`}\hidewidth\crcr\unhbox0}}}
\begin{thebibliography}{BGN82}
\expandafter\ifx\csname url\endcsname\relax
  \def\url#1{\texttt{#1}}\fi
\expandafter\ifx\csname urlprefix\endcsname\relax\def\urlprefix{URL }\fi

\bibitem[AHR01]{AlbRuss}
\textsc{S.~Albeverio}, \textsc{Z.~Haba}, and \textsc{F.~Russo}.
\newblock A two-space dimensional semilinear heat equation perturbed by
  ({G}aussian) white noise.
\newblock \emph{Probab. Theory Related Fields} \textbf{121}, no.~3, (2001),
  319--366.
\newblock \ifx\href\undefined
  \texttt{doi:10.1007/s004400100153}\else\href{http://dx.doi.org/10.1007/s004400100153}{\texttt{doi:10.1007/s004400100153}}\fi.

\bibitem[BCD11]{BookChemin}
\textsc{H.~Bahouri}, \textsc{J.-Y. Chemin}, and \textsc{R.~Danchin}.
\newblock \emph{Fourier analysis and nonlinear partial differential equations},
  vol. 343 of \emph{Grundlehren der Mathematischen Wissenschaften [Fundamental
  Principles of Mathematical Sciences]}.
\newblock Springer, Heidelberg, 2011.

\bibitem[BGN82]{MR649810}
\textsc{G.~Benfatto}, \textsc{G.~Gallavotti}, and \textsc{F.~Nicol{\`o}}.
\newblock On the massive sine-{G}ordon equation in the first few regions of
  collapse.
\newblock \emph{Comm. Math. Phys.} \textbf{83}, no.~3, (1982), 387--410.

\bibitem[CW78]{DynRough}
\textsc{S.~T. Chui} and \textsc{J.~D. Weeks}.
\newblock Dynamics of the roughening transition.
\newblock \emph{Phys. Rev. Lett.} \textbf{40}, (1978), 733--736.
\newblock \ifx\href\undefined
  \texttt{doi:10.1103/PhysRevLett.40.733}\else\href{http://dx.doi.org/10.1103/PhysRevLett.40.733}{\texttt{doi:10.1103/PhysRevLett.40.733}}\fi.

\bibitem[DH00]{MR1777310}
\textsc{J.~Dimock} and \textsc{T.~R. Hurd}.
\newblock Sine-{G}ordon revisited.
\newblock \emph{Ann. Henri Poincar\'e} \textbf{1}, no.~3, (2000), 499--541.
\newblock \ifx\href\undefined
  \texttt{doi:10.1007/s000230050005}\else\href{http://dx.doi.org/10.1007/s000230050005}{\texttt{doi:10.1007/s000230050005}}\fi.

\bibitem[DPD02]{MR1941997}
\textsc{G.~Da~Prato} and \textsc{A.~Debussche}.
\newblock Two-dimensional {N}avier-{S}tokes equations driven by a space-time
  white noise.
\newblock \emph{J. Funct. Anal.} \textbf{196}, no.~1, (2002), 180--210.
\newblock \ifx\href\undefined
  \texttt{doi:10.1006/jfan.2002.3919}\else\href{http://dx.doi.org/10.1006/jfan.2002.3919}{\texttt{doi:10.1006/jfan.2002.3919}}\fi.

\bibitem[DPD03]{MR2016604}
\textsc{G.~Da~Prato} and \textsc{A.~Debussche}.
\newblock Strong solutions to the stochastic quantization equations.
\newblock \emph{Ann. Probab.} \textbf{31}, no.~4, (2003), 1900--1916.

\bibitem[Fal12]{Falco}
\textsc{P.~Falco}.
\newblock Kosterlitz-{T}houless transition line for the two dimensional
  {C}oulomb gas.
\newblock \emph{Comm. Math. Phys.} \textbf{312}, no.~2, (2012), 559--609.
\newblock \ifx\href\undefined
  \texttt{doi:10.1007/s00220-012-1454-7}\else\href{http://dx.doi.org/10.1007/s00220-012-1454-7}{\texttt{doi:10.1007/s00220-012-1454-7}}\fi.

\bibitem[Fr{\"o}76]{Jurg}
\textsc{J.~Fr{\"o}hlich}.
\newblock Classical and quantum statistical mechanics in one and two
  dimensions: two-component {Y}ukawa- and {C}oulomb systems.
\newblock \emph{Comm. Math. Phys.} \textbf{47}, no.~3, (1976), 233--268.

\bibitem[FS81]{MR634447}
\textsc{J.~Fr{\"o}hlich} and \textsc{T.~Spencer}.
\newblock The {K}osterlitz-{T}houless transition in two-dimensional abelian
  spin systems and the {C}oulomb gas.
\newblock \emph{Comm. Math. Phys.} \textbf{81}, no.~4, (1981), 527--602.

\bibitem[GIP14]{PAMPreprint}
\textsc{M.~{Gubinelli}}, \textsc{P.~{Imkeller}}, and \textsc{N.~{Perkowski}}.
\newblock Paracontrolled distributions and singular {PDE}s.
\newblock \emph{ArXiv e-prints} (2014).
\newblock \ifx\href\undefined
  \texttt{arXiv:1210.2684v3}\else\href{http://arxiv.org/abs/1210.2684v3}{\texttt{arXiv:1210.2684v3}}\fi.

\bibitem[Hai13]{Regularity}
\textsc{M.~Hairer}.
\newblock A theory of regularity structures.
\newblock \emph{ArXiv e-prints} (2013).
\newblock \ifx\href\undefined
  \texttt{arXiv:1303.5113}\else\href{http://arxiv.org/abs/1303.5113}{\texttt{arXiv:1303.5113}}\fi.
\newblock Inventiones mathematicae, to appear.

\bibitem[HQ14]{KPZJeremy}
\textsc{M.~Hairer} and \textsc{J.~Quastel}.
\newblock A class of growth models rescaling to {KPZ}, 2014.
\newblock Preprint.

\bibitem[KP93]{SurfModel2}
\textsc{B.~Kahng} and \textsc{K.~Park}.
\newblock Dynamics of the orientational roughening transition.
\newblock \emph{Phys. Rev. B} \textbf{47}, (1993), 5583--5588.
\newblock \ifx\href\undefined
  \texttt{doi:10.1103/PhysRevB.47.5583}\else\href{http://dx.doi.org/10.1103/PhysRevB.47.5583}{\texttt{doi:10.1103/PhysRevB.47.5583}}\fi.

\bibitem[KP94]{SurfModel1}
\textsc{B.~Kahng} and \textsc{K.~Park}.
\newblock Dynamic sine-{G}ordon renormalization of the {L}aplacian roughening
  transition below two dimensions.
\newblock \emph{Phys. Rev. B} \textbf{49}, (1994), 7026--7028.
\newblock \ifx\href\undefined
  \texttt{doi:10.1103/PhysRevB.49.7026}\else\href{http://dx.doi.org/10.1103/PhysRevB.49.7026}{\texttt{doi:10.1103/PhysRevB.49.7026}}\fi.

\bibitem[KT73]{KT}
\textsc{J.~M. Kosterlitz} and \textsc{D.~J. Thouless}.
\newblock Ordering, metastability and phase transitions in two-dimensional
  systems.
\newblock \emph{Journal of Physics C: Solid State Physics} \textbf{6}, no.~7,
  (1973), 1181.

\bibitem[LRV13]{RhodesVargas}
\textsc{H.~{Lacoin}}, \textsc{R.~{Rhodes}}, and \textsc{V.~{Vargas}}.
\newblock {Complex Gaussian multiplicative chaos}.
\newblock \emph{ArXiv e-prints} (2013).
\newblock \ifx\href\undefined
  \texttt{arXiv:1307.6117}\else\href{http://arxiv.org/abs/1307.6117}{\texttt{arXiv:1307.6117}}\fi.

\bibitem[Mey92]{MR1228209}
\textsc{Y.~Meyer}.
\newblock \emph{Wavelets and operators}, vol.~37 of \emph{Cambridge Studies in
  Advanced Mathematics}.
\newblock Cambridge University Press, Cambridge, 1992.
\newblock Translated from the 1990 French original by D. H. Salinger.

\bibitem[Neu83]{Neudecker}
\textsc{B.~Neudecker}.
\newblock Critical dynamics of the sine-{G}ordon model in $d=2-\eps$
  dimensions.
\newblock \emph{Zeitschrift f\"ur Physik B Condensed Matter} \textbf{52},
  no.~2, (1983), 145--149.
\newblock \ifx\href\undefined
  \texttt{doi:10.1007/BF01445296}\else\href{http://dx.doi.org/10.1007/BF01445296}{\texttt{doi:10.1007/BF01445296}}\fi.

\bibitem[Nic83]{MR702570}
\textsc{F.~Nicol{\`o}}.
\newblock On the massive sine-{G}ordon equation in the higher regions of
  collapse.
\newblock \emph{Comm. Math. Phys.} \textbf{88}, no.~4, (1983), 581--600.

\bibitem[NRS86]{MR849210}
\textsc{F.~Nicol{\`o}}, \textsc{J.~Renn}, and \textsc{A.~Steinmann}.
\newblock On the massive sine-{G}ordon equation in all regions of collapse.
\newblock \emph{Comm. Math. Phys.} \textbf{105}, no.~2, (1986), 291--326.

\bibitem[RY91]{RevYor}
\textsc{D.~Revuz} and \textsc{M.~Yor}.
\newblock \emph{Continuous martingales and {B}rownian motion}, vol. 293 of
  \emph{Grundlehren der Mathematischen Wissenschaften [Fundamental Principles
  of Mathematical Sciences]}.
\newblock Springer-Verlag, Berlin, 1991.

\bibitem[Tri83]{Triebel}
\textsc{H.~Triebel}.
\newblock \emph{Theory of function spaces}, vol.~78 of \emph{Monographs in
  Mathematics}.
\newblock Birkh\"auser Verlag, Basel, 1983.

\end{thebibliography}

\end{document}